\documentclass{amsart}
\usepackage{hyperref}
\hypersetup{colorlinks=true}
\usepackage{amsmath}
\usepackage{amssymb}
\usepackage{amsfonts}
\usepackage{upgreek}
\usepackage{a4}
\usepackage{todonotes}
\usepackage{latexsym, amsthm, mathrsfs, enumitem,hyphenat}
\usepackage{fancybox}
\usepackage{amscd}
\usepackage[T1]{fontenc}
\usepackage[all]{xy}
\input{xy}
\usepackage[all]{xy}
\numberwithin{equation}{section}

\newtheorem{theorem}{Theorem}[section]
\newtheorem*{maintheorem}{Main Theorem}
\newtheorem{lemma}[theorem]{Lemma}

\theoremstyle{definition} %fettes label, aber nich kursiv

\theoremstyle{plain} %fettes label. text kursiv
\newtheorem{question}[theorem]{Question}

\newtheorem{proposition}[theorem]{Proposition}
\newtheorem{corollary}[theorem]{Corollary}
\newtheorem{fact}[theorem]{Fact}

\newtheorem{claim}[theorem]{Claim}

\newtheorem*{maintheorem*}{Main Theorem}
\newtheorem*{conjecture*}{Conjecture}
\newtheorem{definition}[theorem]{Definition}

\theoremstyle{remark}
\newtheorem{remark}[theorem]{Remark}

\theoremstyle{remark}  %kursives label. text roman

\newcommand{\nc}{\newcommand}
\nc{\nothing}[1]{}
\nc{\added}[1]{#1} % Neu Sept 2000, fuer wichtige inhaltliche
\nc{\comment}[1]{#1} % Neu Sept 2000, fuer Kommentare
\nc{\nco}{\DeclareMathOperator}
\nco{\tp}{tp}
\nco{\halv}{half}
\nco{\order}{o}
\nco{\ppower}{pp}
\nco{\pcf}{pcf} %possible cofinality
\nco{\tcf}{tcf} %true cofinality
\nco{\tlim}{tlim} %true lim
\nco{\limtext}{lim} %true lim
\nco{\prodt}{{\textstyle \prod}}
\nco{\symdiff}{\triangle}
\nco{\dom}{dom}
\nco{\card}{card}
\nco{\lh}{lh}
\nco{\lhg}{lg}
\nco{\rge}{rge}
\nco{\otp}{otp}
\nco{\trunk}{tr}
\nco{\cf}{cf}
\nco{\nex}{next}
\nc{\uhr}{\restriction}
\nco{\supt}{supt}
\nco{\supp}{supp}
\nco{\Lim}{Lim}
\nco{\Leb}{Leb}
\nco{\modd}{mod}
\nco{\invariant}{inv}
\nco{\RO}{RO}
\nco{\Dp}{Dp} %depth all from here on for creatures
\nco{\pss}{ps}
\nco{\acc}{acc}
\nco{\spec}{spec}
\nco{\pr}{pr}
\nco{\rt}{rt}
\nco{\suc}{suc}
\nco{\splitt}{split}
\nc{\potom}{\ensuremath{{\cal P}(\omega)}}
\nc{\potinf}{\ensuremath{[\omega]^\omega}}
\nc{\pfin}{\ensuremath{{\cal P}(\omega)/{\rm fin}}}

\nc{\potfin}{\ensuremath{[\omega]^{<\omega}}}
\nc{\inn}{\ensuremath{{\omega^{\uparrow \omega}}}}
\nc{\hoch}{^{<\omega}}
\nc{\hocho}{^{\omega}}
\nc{\tree}[1]{{[} #1 {]}_0}
\nc{\tre}[2]{ {#1}_{#2}}

\nc{\prooff}[1]{{\bf Proof} of #1:}
\nc{\proofend}{\makebox{} \hfill ${\bf \dashv}$ \\}
\nc{\proofendof}[1]{\makebox{} \hfill $\boldmath{\square}_{\rm #1}$
\\} \nc{\beq}{\begin{eqnarray*}} \nc{\eeq}{\end{eqnarray*}}
\nc{\bde}{\begin{list}} \nc{\ede}{\end{list}}

{\end{list}}

\newcounter{subalph}
{\end{list}}

\newcommand{\greek}[1]{\ifthenelse{\value{#1}=1}{\mbox{$\alpha$}}%
  {\ifthenelse{\value{#1}=2}{\mbox{$\beta$}}{%
   \ifthenelse{\value{#1}=3}{\mbox{$\gamma$}}{%
   \ifthenelse{\value{#1}=4}{\mbox{$\delta$}}{%
   \ifthenelse{\value{#1}=5}{\mbox{$\varepsilon$}}{%
   \ifthenelse{\value{#1}=6}{\mbox{$\zeta$}}{%
   \ifthenelse{\value{#1}=7}{\mbox{$\eta$}}{%
   \ifthenelse{\value{#1}=8}{\mbox{$\theta$}}{%
   \ifthenelse{\value{#1}=9}{\mbox{$\iota$}}{%
   \ifthenelse{\value{#1}=10}{\mbox{$\kappa$}}{%
   \ifthenelse{\value{#1}=11}{\mbox{$\lambda$}}{%
   \ifthenelse{\value{#1}=12}{\mbox{$\mu$}}{%
   \ifthenelse{\value{#1}=13}{\mbox{$\nu$}}{%
   \ifthenelse{\value{#1}=14}{\mbox{$\xi$}}{%
   \ifthenelse{\value{#1}=15}{\mbox{$\rm o$}}{%
   \ifthenelse{\value{#1}=16}{\mbox{$\pi$}}{%
   \ifthenelse{\value{#1}=17}{\mbox{$\varrho$}}{%
   \ifthenelse{\value{#1}=18}{\mbox{$\sigma$}}{%
   \ifthenelse{\value{#1}=19}{\mbox{$\tau$}}{%
   \ifthenelse{\value{#1}=20}{\mbox{$\upsilon$}}{%
   \ifthenelse{\value{#1}=21}{\mbox{$\varphi$}}{%
   \ifthenelse{\value{#1}=22}{\mbox{$\chi$}}{%
   \ifthenelse{\value{#1}=23}{\mbox{$\psi$}}{\mbox{$\omega$}%
  }}}}}}}}}}}}}}}}}}}}}}}}%24 Zuklammern

\newcounter{subgreek}
{\end{list}}

\newcounter{subarabic}
{\end{list}}

\newcounter{subroman}
{\end{list}}

\newcount\skewfactor

\def\mathunderaccent#1#2 {\let\theaccent#1\skewfactor#2
\mathpalette\putaccentunder}
\def\putaccentunder#1#2{\oalign{$#1#2$\crcr\hidewidth
\vbox to.2ex{\hbox{$#1\skew\skewfactor\theaccent{}$}\vss}\hidewidth}}
\def\name{\mathunderaccent\tilde-3 }

\nc{\nname}{\name}
\nc{\even}{\ensuremath{\rm Even}}
\nc{\odd}{\ensuremath{\rm Odd}}

\nc{\al}{$\alpha$\  }
\nc{\om}{\omega}
\nc{\omm}{\ensuremath{\omega_1}}
\nc{\ep}{\varepsilon}
\nc{\tk}{\tilde{K}}
\nc{\concat}{{}^\smallfrown{}}   %im math mode: 2017 in Luca's style
\nc{\force}{\models}
\nc{\fb}{f_{\bar{M}}}
\nc{\such}{\, : \,}

\nc{\meager}{\ensuremath{{\cal M}}}
\nc{\lebesgue}{\ensuremath{{\cal N}}}
\nc{\nulll}{\ensuremath{{\cal N}}}
\nc{\ksigma}{\ensuremath{{\bf K}_\sigma}}
\nc{\ideal}{\ensuremath{{\cal I}}}
\nc{\ga}{\ensuremath{\frak a}}
\nc{\AAA}{{\cal A}}   %im math mode
\nc{\gc}{\ensuremath{\frak c}}
\nc{\gs}{\ensuremath{\frak s}}
\nc{\gh}{\ensuremath{\frak h}}
\nc{\gd}{\ensuremath{\frak d}}
\nc{\gb}{\ensuremath{\frak b}}
\nc{\gro}{\ensuremath{\frak g}}
\nc{\gu}{\ensuremath{\frak u}}
\nc{\gr}{\ensuremath{\frak r}}
\nc{\gt}{\ensuremath{\frak t}}
\nc{\fff}{\ensuremath{\frak f}}
\nc{\gm}{\ensuremath{\mathfrak{mcf}}}
\nc{\gge}{\ensuremath{\mathfrak e}}
\nc{\cfupro}{\ensuremath{\cf(\upro)}}
\nc{\cfvpro}{\ensuremath{\cf(\vpro)}}
\nc{\gp}{\ensuremath{\frak p}}
\nc{\gk}{\ensuremath{\frak k}}

\nc{\add}[1]{\mbox{\ensuremath{{\rm add}(#1)}}}
\nc{\cov}[1]{\mbox{\ensuremath{{\rm cov}(#1)}}}
\nc{\unif}[1]{\mbox{\ensuremath{{\rm unif}(#1)}}}
\nc{\addd}[2]{\mbox{\ensuremath{{\rm add}^{#1}(#2)}}}   %fuer ch4.
\nc{\covv}[2]{\mbox{\ensuremath{{\rm cov}^{#1}(#2)}}}   %unter #1 steht
\nc{\uniff}[2]{\mbox{\ensuremath{{\rm unif}^{#1}(#2)}}} %das Modell
\nc{\coff}[2]{{\mbox{\ensuremath{\rm cof}^{#1}(#2)}}}

\nc{\cd}{Cicho\'n's Diagram}
\nc{\COF}{\mbox{\bf Cof}}

\nc{\Pieinseins}{\mbox{${\bf \Pi}^1_1$}}
\nc{\seinseins}{\mbox{${\bf\Sigma}^1_1$}}
\nc{\seinszwei}{\mbox{${\bf\Sigma}^1_2$}}
\nc{\seinsdrei}{\mbox{${\bf\Sigma}^1_3$}}
\nc{\Deleinszwei}{\mbox{${\bf\Delta}^1_2$}}

\nc{\up}{\ensuremath{{\cal U}\mbox{\ensuremath{\rm -prod}}\,\omega}}
\nc{\upp}{\ensuremath{{\cal U}'\mbox{\ensuremath{\rm -prod}}\,\omega}}
\nc{\upro}{\ensuremath{{\cal U}\mbox{\ensuremath{\rm -prod}}\,\om}}
\nc{\fupro}{\ensuremath{f({\cal U})\mbox{\ensuremath{\rm -prod}}\,\om}}
\nc{\vpro}{\ensuremath{{\cal V}\mbox{\ensuremath{\rm -prod}}\,\om}}
\nc{\fpro}{\ensuremath{{\cal F}\mbox{\ensuremath{\rm -prod}}\,\om}}

\nc{\cff}[1]{{\text{cf}\,(#1)}}           %cf mit richtiger parameter
\nc{\cu}{\ensuremath{\cal U}}             %Vorsicht, dies gibt im
\nc{\ai}{\ensuremath{\forall^\infty}}     %geaendert ensuremath dazu
\nc{\ei}{\ensuremath{\exists^\infty}}     %auch geaendert
\nc{\ww}{\ensuremath{\omega^\omega}}      %auch geaendert
\nc{\N}{\mathbb N}
\nc{\gw}{groupwise dense}

\nc{\kk}{car\-dinal cha\-rac\-teris\-tic}
\nc{\joker}{\ast}
\nc{\gtc}{Galois-Tukey connection} %Vorsicht: heisst generalized

\nc{\av}[1]{{\rm Av}_{#1}}
\nc{\eps}{\varepsilon}
\nc{\n}{{\bf n}}                 %for the blueprints
\nc{\m}{{\bf m}}

\nc{\marginparr}[1]%{}
{\marginpar{#1}}
\nc{\footnoteee}{} % I think the first footnote is no more a question
\nc{\footnotee}{}  % for some ad libitum notes

\newcommand{\cal}{\mathcal}

\nc{\divs}{{c_0 \setminus \ell^1}}
\nc{\divser}{(\divs, \leq^*)/\thickapproy}
\nc{\bfin}{\RO(\pfin \setminus\{0\},\subseteq^*)}
\nc{\bdivser}{\RO(\divser)}
\nc{\inc}{{\rm INC}}
\nc{\com}{{\rm COM}}
\nc{\thickapproy}{\makebox{}\!\!\thickapprox}
\nc{\approy}{\makebox{}\!\!\approx}
\nc{\lessi}{\leqslant}
\nc{\gessi}{\geqslant}
\nc{\interior}[1]{{\rm int}(#1)}
\nc{\closure}[1]{{\rm cl}(#1)}
\nc{\Vo}{Vojt\'a\v{s}}

\nc{\precedeseq}{\leq^*} %%%take your favorite partial order!!!!
\nc{\precedes}{\prec}
\nc{\stronger}{\leqslant_{\bf P}}
\nc{\underlline}[1]{\hat{#1}}
\nc{\PO}{{\bf P}}
\nc{\charak}{\text{ch}}

\nc{\needed}{needed\ }
\nc{\neededc}{needed}
\nc{\Needed}{Needed\ }
\nc{\wneeded}{weakly needed\ }
\nc{\Wneeded}{Weakly needed\ }
\nc{\wneededc}{weakly needed}

\nc{\mup}{m_{\rm up}}
\nc{\mdn}{m_{\rm dn}}
\nco{\may}{may}
\nco{\aver}{av} % neu fuer F427
\nco{\norm}{nor} % neu fuer F427  SCHOENER ohne fett
\nco{\val}{val} % neu fuer F427
\nco{\dis}{dis} % neu fuer F427
\nco{\basis}{basis} \nco{\pos}{pos}
\nco{\hg}{ht} %height

\nc{\err}{\mbox{err}}
\nc{\eee}{\mbox{e}}
\nco{\Expect}{Exp}

\nc{\bc}{{\bf c}} \nc{\bd}{{\bf d}}

\nc{\la}{\langle}
\nc{\ra}{\rangle}
\nc{\bP}{\mathbb P}
\nc{\bQ}{\mathbb Q}
\nc{\rest}{\restriction}

\nc{\bV}{\bf V}
\nc{\bG}{\bf G}
\nc{\bT}{\bf T}

\nc{\geqqeq}{\geq}
\nc{\leqqeq}{\leq}
\nc{\reri}{\uparrow}
\nc{\cS}{\mathscr S}
\nc{\sL}{\mathscr L}
\nc{\cT}{\mathcal T}
\nc{\sG}{\mathscr G}

\renewcommand{\phi}{\varphi}
\nco{\Pred}{Pred}
\nco{\Cone}{Cone}
\nc{\bZ}{{\mathbb Z}}

\newcommand\V{\mathsf{v}}

\newcommand\bN{\boldsymbol{{\rm N}}}
\renewcommand{\L}{\mathcal{L}}
\newcommand{\sS}{\mathscr{S}}

\newcommand{\T}{\mathcal{T}}

\newcommand{\ZZ}{\mathbb{Z}}

\newcommand{\Hes}{\mathcal{H}}
\newcommand{\pre}[2]{{}^{#1} #2}
\newcommand{\seq}[2]{\langle #1 \mid #2 \rangle}
\newcommand{\set}[2]{\{ #1 \mid #2 \}}

\newcommand{\proj}{\operatorname{p}}

\newcommand{\Mod}{\operatorname{Mod}}
\newcommand{\range}{\operatorname{Range}}

\newcommand{\pred}{\operatorname{Pred}}
\newcommand{\cone}{\operatorname{Cone}}
\newcommand{\SUCC}{\operatorname{Succ}}
\newcommand{\leng}[1]{\operatorname{length}(#1)}

\newcommand{\On}{{\sf On}}

\newcommand{\Linf}{\L_{\kappa^+ \kappa}}

\newenvironment{enumerate-(a)}{\begin{enumerate}[label={\upshape (\alph*)}, leftmargin=2pc]}{\end{enumerate}}
\newenvironment{enumerate-(a)-r}{\begin{enumerate}[label={\upshape (\alph*)}, leftmargin=2pc,resume]}{\end{enumerate}}
\newenvironment{enumerate-(A)}{\begin{enumerate}[label={\upshape (\Alph*)}, leftmargin=2pc]}{\end{enumerate}}
\newenvironment{enumerate-(A)-r}{\begin{enumerate}[label={\upshape (\Alph*)}, leftmargin=2pc,resume]}{\end{enumerate}}
\newenvironment{enumerate-(i)}{\begin{enumerate}[label={\upshape (\roman*)}, leftmargin=2pc]}{\end{enumerate}}
\newenvironment{enumerate-(i)-r}{\begin{enumerate}[label={\upshape (\roman*)}, leftmargin=2pc,resume]}{\end{enumerate}}
\newenvironment{enumerate-(I)}{\begin{enumerate}[label={\upshape (\Roman*)}, leftmargin=2pc]}{\end{enumerate}}
\newenvironment{enumerate-(I)-r}{\begin{enumerate}[label={\upshape (\Roman*)}, leftmargin=2pc,resume]}{\end{enumerate}}
\newenvironment{enumerate-(1)}{\begin{enumerate}[label={\upshape (\arabic*)}, leftmargin=2pc]}{\end{enumerate}}
\newenvironment{enumerate-(1)-r}{\begin{enumerate}[label={\upshape (\arabic*)}, leftmargin=2pc,resume]}{\end{enumerate}}
\newenvironment{itemizenew}{\begin{itemize}[leftmargin=2pc]}{\end{itemize}}
\newenvironment{enumerate-(Ia)}{\begin{enumerate}[label={\upshape (I\alph*)}, leftmargin=2pc]}{\end{enumerate}}
\newenvironment{enumerate-(IIa)}{\begin{enumerate}[label={\upshape (II\alph*)}, leftmargin=2pc]}{\end{enumerate}}
\newenvironment{enumerate-(1a)}{\begin{enumerate}[label={\upshape (1\alph*)}, leftmargin=2pc]}{\end{enumerate}}
\newenvironment{enumerate-(2a)}{\begin{enumerate}[label={\upshape (2\alph*)}, leftmargin=2pc]}{\end{enumerate}}
\newenvironment{enumerate-(C1)}{\begin{enumerate}[label={\upshape (C\arabic*)}, leftmargin=2pc]}{\end{enumerate}}

\begin{document}
%\selectlanguage{english}

%____________Title-------------------

\title[Bi-embeddability on uncountable structures]{Uncountable structures are not classifiable \\ up to bi-embeddability}

\author{Filippo Calderoni}
\address{Institut f\"ur Mathematische Logik, Mathematisches Institut,
Universit\"at M\"unster,
Einsteinstrasse 62, 48149 M\"unster, Germany}
\email{calderfi@uni-muenster.de}

\author{Heike Mildenberger}
%\nothing{
\address{Abteilung f\"ur Mathematische Logik,
Mathematisches Institut, Universit\"at Freiburg, Eckerstr.~1,
 79104 Freiburg im Breisgau, Germany}
\email{heike.mildenberger@math.uni-freiburg.de}

\author{Luca Motto Ros}
\address{Dipartimento di matematica \guillemotleft{Giuseppe Peano}\guillemotright, Universit\`a di Torino, Via Carlo Alberto 10, 10123 Torino, Italy}
 \email{luca.mottoros@unito.it}

\begin{abstract}
Answering some of the main questions from~\cite{mottoros2011}, we show that whenever \( \kappa \) is a cardinal satisfying \( \kappa^{< \kappa} = \kappa > \omega \), then the embeddability relation between \( \kappa \)-sized structures is strongly invariantly universal, and hence complete for (\( \kappa \)-)analytic quasi-orders. We also prove that in the above result we can further restrict our attention to various natural classes of structures, including (generalized) trees, graphs, or groups. This fully generalizes to the uncountable case the main results of~\cite{louveau-rosendal,friedman-mottoros,Wil14,calderoni-mottoros}.
\end{abstract}

%\subjclass{03E15, 03E17, 03E35, 03D65}
%}
\date{\today} %Dec 17, 2017

\thanks{The authors would like to thank the anonymous referee for carefully reading the manuscript and providing valuable suggestions.}

\maketitle

% ----------------end title-----------------------------
%\tableofcontents %(fuer lange Artikel)
\section{Introduction}\label{S0}

The problem of classifying countable structures up to isomorphism and bi\hyp{}embeddability has been an important theme in modern descriptive set theory (see e.g.\ \cite{friedman-stanley,thomas-velickovic,camerlo-gao,thomas-velickovic2,gao,thomas2,hjorth,thomas3,clemens,mottorosPAMS,coskey,williams3} and~\cite{louveau-rosendal,friedman-mottoros,thomas-williams,Wil14,thomas-williams2,calderoni-mottoros, calderoni-thomas}, respectively).
In this framework, such classification problems are construed as analytic equivalence relations on standard Borel spaces, and their complexity is measured using the theory of Borel reducibility.

If one wants to perform a similar analysis for classification problems concerning uncountable structures, then the usual setup is of no use, as there is no natural way to code uncountable structures as elements 
of a Polish or standard Borel space. The natural move is thus to consider what is now called \emph{generalized descriptive set theory}. In this theory, one fixes an arbitrary uncountable cardinal \( \kappa \) and then considers 
the so-called \emph{generalized Cantor space}, that is, the space \( {}^\kappa 2 \) of binary \( \kappa \)-sequences equipped with the \emph{bounded topology}, which is the one generated by the sets of the form
\[ 
\bN_s = \set{ x \in
\pre{\kappa}{2} }{s \subseteq x } %change 1 by Heike \lambda ---> 2	
 \] 
for \( s \) a binary sequence of length \(  < \kappa \). (The \emph{generalized Baire space} \( {}^\kappa \kappa \) is defined analogously.) Notice that this naturally generalizes the topology of the (classical) Cantor space, 
which corresponds to the case \( \kappa = \omega \); however, when \( \kappa > \omega \) the bounded topology no longer coincides with the product topology, and other unexpected quirks suddenly show up. 

Building on 
the topology just defined, one can in turn recover in a straightforward way all other descriptive set-theoretical notions like (\( \kappa \)-)Borel sets, (\( \kappa \)-)analytic sets, also called \( \boldsymbol{\Sigma}^1_1 \) sets, standard 
Borel (\( \kappa \)-)spaces, and so on (see Section~\ref{subsec:basicstopology} for more details).

Using characteristic functions of its predicates, every (relational) structure with domain \( \kappa \) can be naturally coded as an element of (a space homeomorphic to) \( {}^\kappa 2 \). For example, if \( G \) is a graph 
on \( \kappa \), then it can be coded as a point \( x \in \pre{\kappa \times \kappa}{2} \approx {}^\kappa 2\) by stipulating that \( x(\alpha,\beta) = 1 \) if and only if \( \alpha \) and \( \beta \) are adjacent in \( G \). 
This coding procedure allows us to construe the relations of isomorphism and embeddability between structures of size \( \kappa \) as (\( \kappa \)-)analytic relations on a suitable standard Borel \hbox{(\( \kappa \)-)space}
(see Section~\ref{subsec:inflogic}). 
Finally, by introducing the analogue of the notion of Borel reducibility in this generalized context (see Section~\ref{subsec:Borelreducibility}), 
one can then analyze the complexity of such relations mimicking what has been done for countable structures in the classical setup.

The first two seminal papers exploiting this approach were~\cite{FriedmanHyttinenKulikov}, where the complexity of the isomorphism relation between uncountable structures is remarkably connected to Shelah's stability theory, 
and~\cite{mottoros2011}, where it is shown that if \( \kappa \) is a weakly compact cardinal, then structures of size \( \kappa \) belonging to various natural classes (graphs, trees, and so on) are unclassifiable 
up to bi-embeddability. The latter is a generalization of a similar result first obtained for countable structures in~\cite{louveau-rosendal}, and then strengthened in~\cite{friedman-mottoros}.

The fact that in~\cite{mottoros2011} only the case of a weakly compact cardinal \( \kappa \) was treated relies on the fact that in such a situation the behavior of the space \( {}^\kappa 2 \) is somewhat closer to the one 
of the usual Cantor space \( {}^\omega 2 \), while when we lack such a condition its behavior is much wilder (see~\cite{luecke-schlicht,luecke-mottoros-schlicht,andmot}  for more on this). For example, it is not hard to see that \( \pre{\kappa}{2} \), endowed with the 
bounded topology, is never compact, but it is at least \( \kappa \)-compact (i.e.\ every open covering of it can be refined to a subcovering of size \( < \kappa \)) if and only if \( \kappa \) is weakly compact, if and only if 
\( {}^\kappa 2 \) is not homeomorphic to the generalized Baire space \( {}^\kappa \kappa \). 

Nevertheless, we are going to show that the assumption that \( \kappa \) be a large cardinal is not necessary to prove that uncountable structures are unclassifiable up 
to bi-embeddability, answering in particular Question 11.1 and the first part of Question 11.5 from~\cite{mottoros2011}. More precisely, we prove that 
\begin{maintheorem}
For every uncountable cardinal \( \kappa \) satisfying \( \kappa^{ < \kappa} = \kappa \),  the embeddability relation on all structures of size \( \kappa \) is \emph{strongly invariantly universal}, that is: For every (\( \kappa \)-)analytic quasi-order \( R \) on \( {}^\kappa 2 \) there is an \(\mathcal{L}_{\kappa^+ \kappa} \)-sentence 
\( \upvarphi \) such that the embeddability relation on the \( \kappa \)-sized models of \( \upvarphi \) is \emph{classwise Borel isomorphic}%
\footnote{Classwise Borel isomorphism is a natural strengthening of Borel bi-reducibility, see Section~\ref{subsec:Borelreducibility}.}
 to \( R \). 
 \end{maintheorem}
 In particular, this implies that all (\( \kappa \)-)analytic equivalence relations on \( {}^\kappa 2 \) are Borel reducible to the bi-embeddability relation on structures of size \( \kappa \), so that the latter relation is 
 as complicated as possible. This technical fact proves (in a very strong sense!) that uncountable structures are essentially unclassifiable up to bi-embeddability.
As done in~\cite{mottoros2011}, we also show that in the Main Theorem one could further restrict the 
 attention to some particular classes of structures, such as generalized trees or graphs.
 Notice also that the required cardinal condition \( \kappa^{< \kappa} = \kappa \) is very mild: in a model of 
 \( \mathsf{GCH} \), the Generalized Continuum Hypothesis, all regular cardinals satisfy it.

Our construction  follows closely the one from~\cite{mottoros2011}. In the original argument, the fact that \( \kappa \) was assumed to be weakly compact was crucially exploited several times:
\begin{itemize}
\item
when providing a sufficiently nice tree representation for the (\( \kappa \)-)analytic quasi-order \( R \) on \( {}^\kappa 2 \), it was used the fact that \( \kappa \) is inaccessible and has the tree property%
\footnote{Recall that an uncountable cardinal is weakly compact exactly when it is inaccessible and has the tree property.}
  (see~\cite[Lemma 7.2]{mottoros2011});
\item
when defining the complete quasi-order \( \leq_{\mathrm{max}} \), the inaccessibility of \( \kappa \) was used to provide the auxiliary map \( \# \), a key tool in the main construction (see~\cite[Proposition 7.1 and 
Theorem 9.3]{mottoros2011});
\item
when constructing suitable labels to code up the quasi-order \( R \), it was again used the fact that \( \kappa \) is inaccessible  (see~\cite[Section 8]{mottoros2011});
\item
finally, when proving strongly invariant universality of the embeddability relation, it was used the fact that \( {}^\kappa 2 \) is a \( \kappa \)-compact space, which as recalled is a condition equivalent to \( \kappa \) being 
weakly compact  (see~\cite[Section 10, and in particular (the proof of) Theorem 10.23]{mottoros2011}).
\end{itemize}
The main technical contribution of this paper is to show how to overcome all these difficulties when \( \kappa \) is not even inaccessible. This lead us to a substantial modification of all the coding processes 
(Sections~\ref{sec:max}--\ref{sec:completeness}), as well as to a
%brand (Change 2 by Heike, Dec 17, 2017)
new argument to establish the strongly invariant universality of the embeddability relation between uncountable structures 
(Section~\ref{sec:universality}).

In Section~\ref{sec:groups} we further show that in the main result one could also consider groups of size \( \kappa \) instead of trees or graphs, a result which is new also in the case of a weakly compact \( \kappa \) 
and generalizes to the uncountable case one of the main results of~\cite{calderoni-mottoros}. This is obtained by providing a way for interpreting (in a very strong model-theoretic sense) graphs into groups. Such technique works well 
also in the countable case, and provides an alternative proof of~\cite[Theorem 3.5]{calderoni-mottoros}. Finally, in  Section~\ref{sec:questions} we collect some 
further corollaries of our Main Theorem concerning non-separable complete metric spaces and non-separable Banach spaces, and ask some questions motivated by our analysis.

We conclude this introduction with a general remark. There is a common trend in generalized descriptive 
set theory:  the natural generalizations to the uncountable context of any nontrivial result from classical 
descriptive set theory are either simply false, or independent of \( \mathsf{ZFC} \) --- their truth can be 
established only under extra assumptions (in particular, large cardinal assumptions on \( \kappa \) itself), 
or by working in some very specific model of \( \mathsf{ZFC} \). Somewhat unexpectedly after~\cite{mottoros2011}, the results of this paper constitute a 
rare exception: indeed, invariant universality results tend to be quite sophisticated (see~\cite{friedman-mottoros,CamerloMarconeMottoros2013,calderoni-mottoros,CamerloMarconeMottoros2017}), yet here we 
demonstrate that some of them fully generalize to the uncountable setup  without any extra set-theoretical 
assumption (and with the only commonly accepted requirement that \( \kappa^{< \kappa} = \kappa \)).

\medskip

\textbf{Acknowledgments.} The results in Sections~\ref{sec:max}--\ref{sec:completeness} are due to the second and third authors, while the results in Section~\ref{sec:groups} are due to the first and third authors.
Until September 2014 the third author was a member of the Logic Department of the Albert-Ludwigs-Universit\"at Freiburg, which supported him at early stages of this research. After that, he was supported by the Young
Researchers Program ``Rita Levi Montalcini'' 2012 through the project ``New advances in Descriptive Set Theory''. The first author was funded by the Deutsche Forschungsgemeinschaft (DFG, German Research Foundation) under Germany's Excellence Strategy --- EXC 2044 --- <ID 390685587>, Mathematics M\"unster:
Dynamics -- Geometry -- Structure.

\section{Preliminaries and basic notions} \label{S1}

Throughout the paper, we will use the terminology and notation from~\cite{mottoros2011}. For the reader's convenience, we will recall in this section all the relevant basic facts and definitions, referring him/her to~\cite{mottoros2011} for motivations and more detailed discussions on these notions and results. For all other undefined notation and concepts, we refer the reader to~\cite{Kechris, Jech3, Kanamori}.

\subsection{Ordinals and cardinals}

We let $\On$ be the class of all ordinals. The Greek letters \( \alpha, \beta, \gamma, \delta \) (possibly 
with various decorations) will usually denote ordinals, while the letters \( \nu, \lambda, 
\kappa \) will usually denote cardinals. 

We let \( |A| \) be the \emph{cardinality of the set \( A \)}, i.e.\ 
the unique cardinal \( \kappa \) such that \( A \) is in bijection with \( \kappa 
\). Given a cardinal \( \kappa \), we denote with \( \kappa^+ \) the smallest 
cardinal (strictly) greater than \( \kappa \). Moreover, we let \( [A]^\kappa \) 
be the collection of all subsets of \( A \) of cardinality \( \kappa \), and \( 
[A]^{< \kappa}  = \bigcup_{\gamma < \kappa} [A]^\gamma \) be the collection of all subsets of \( A \) of cardinality \( < 
\kappa \).

We denote by $\Hes \colon \On \times \On \to \On$ the Hessenberg pairing function for the class of all ordinals
$\On$ (see e.g.\ \cite[p.\ 30]{Jech3}), i.e.\ the unique surjective function such that for all \( \alpha, \alpha',\beta, \beta' \in \On \)
\begin{align*} 
\Hes((\alpha,\beta)) \leq \Hes((\alpha',\beta')) \iff &\max { \{ \alpha,\beta \} < \max \{ \alpha', \beta' \} }\vee \\
& ({\max\{ \alpha,\beta \} = \max \{ \alpha', \beta' \}} \wedge {(\alpha , \beta ) \leq_{\mathsf{lex}} (\alpha',\beta')}),
 \end{align*} 
where \( \leq_{\mathsf{lex}} \) is the lexicographical ordering on \( \On \times \On \).
%Then define by induction the bijections $\Hes_n \colon \pre{n}{\On} \to \On$
%(for $n \geq 2$) by letting $\Hes_2 = \Hes$ and
%$\Hes_{n+1}(\alpha_0, \dotsc , \alpha_n) = \Hes( \alpha_0 ,
%\Hes_n (\alpha_1, \dotsc, \alpha_n))$ for $\alpha_0, \dotsc,
%\alpha_n \in \On$.

\subsection{Sequences and functions}

Given a nonempty set $A$ and \( \kappa \in \On \), we denote by $\pre{<\gamma}{A}$ the set of all sequences of length \( < \gamma \) with values in \( A \), i.e.\ the set of all 
functions of the form $f \colon \alpha \to A$ for some $\alpha <\gamma$ (we call $\alpha$
the \emph{length of $f$} and denote it by $\lh(f)$). The set of all functions from \(\gamma\) to \( A \) is denoted by \( \pre{\gamma}{A} \), so that \( \pre{< \gamma}{A} = \bigcup_{\alpha < \gamma} \pre{\alpha}{A} \).  We also set
\[ 
\pre{\SUCC(<\gamma)}{A} = \{ s \in \pre{< \gamma}{A} \mid \leng{s}\text{ is a successor ordinal} \} = \bigcup\nolimits_{\alpha + 1 < \gamma} \pre{\alpha + 1}{A}.
\]

When \( f \in \pre{\gamma}{A} \) and \( \alpha < \gamma \), we let \( f \restriction \alpha \) be the restriction of \( f \) to \(\alpha\). We write $s\concat t$  to denote the concatenation of the sequences \( s \) and \( t \), $\la \alpha \ra$ for the singleton sequence $\{(0,\alpha)\}$, and we write $\alpha \concat s$ and
$s\concat \alpha$ for $\la \alpha \ra \concat s$ and $s \concat \la \alpha \ra$.  For \( \gamma \in 
\On \) and \( a \in A \), we denote by \( a^{(\gamma)} \) the \( \gamma \)-sequence constantly equal to \( a \). 
If $A = A_0 \times 
\dotsc \times A_k$ we will identify each element
 $s \in \pre{< \kappa}{A}$ with a sequence $(s_0,\dots, s_k)$ of elements of the same length such that $s_i \in \pre{< \kappa}{A_i}$. 
% If \( f \colon A \to B \) is a function and \( \gamma \in \On \), we let \( f^{(\gamma)} \colon \pre{\gamma}{A} \to \pre{\gamma}{B} \) be defined coordinatewise, i.e.\ \( f^{(\gamma)}(s) = \seq{f(s_i)}{i < \leng{s}} \). We also set \( f^{(< \gamma)} = \bigcup_{\alpha < \gamma} f^{(\alpha)} \colon \pre{< \gamma}{A} \to \pre{<\gamma}{B} \).

If \(f \) is a function between two sets \( X \) and \( Y \) and \( C \subseteq X \) we set
\[ 
f``C = \{ y \in Y \mid \exists a \in C \, (f(a) = y) \}.
 \]

\subsection{Trees}

In this paper we will consider several kind of trees, so to disambiguate the terminology we recall here the main definitions. 
\begin{definition}\label{2.1}
  Let \( \L \) be a language consisting of just one binary relation symbol \( \preceq \). An \( \L \)-structure \( T = (T, \preceq^T ) \) will be called
  a \emph{generalized tree} if \( \preceq^T \) is a partial order on the set \( T \) such that 
the set
\[ 
\pred(x) = \{ y \in T \mid y \preceq^T x \wedge y \neq x \}
 \] 
of predecessors of any point \( x \in T \) is linearly ordered by \( \preceq^T \) (in particular, any linear order is a generalized tree).

A \emph{set-theoretical tree}  $T$ is a generalized tree such that \( (\pred(x), \preceq^T \restriction \pred(x)) \) is well-founded (hence a well-order) for every \( x \in T \).

A \emph{descriptive set-theoretical tree (on a set \( A \))}  is a set-theoretical tree such that there is an ordinal \( \gamma \) such that \( T \subseteq \pre{< \gamma}{A} \), \( T \) is closed under initial segments, and \( {\preceq^T} = {\subseteq} \) is the initial segment relation between elements of \( T \). Descriptive set-theoretical trees will be sometimes briefly called \emph{DST-trees}.
\end{definition}

We often write just tree when we mean a generalized tree.
%We also informally call \emph{mixed tree} any \( \L \)-structure which is a combination of various kind of trees described above. 
The elements of a tree (of any kind) are called indifferently points or nodes.
%For the sake of simplicity, a generalized tree will be always called just \emph{tree}. 
%A tree is called a \emph{rooted tree} if it has exactly one minimal element, called the root of \( T \) and denoted by $\rt({\mathcal T})$. \todo{Is the notion of root used elsewhere?}

Given a tree \( T \) and a point \( x \in T \), the \emph{upper cone above \( x \)} is the set
\[ 
\cone(x) = \{ y \in T \mid x \preceq y \}.
 \] 
Two distinct nodes \( x,y \in T \) are said \emph{comparable} 
if \( {x \in \pred(y)} \vee {y \in \pred(x)} \), and are said 
\emph{compatible} if there is 
\( z \in \pred(x) \cap \pred(y) \) (given such a \( z \), we will also say that 
\( x \) and \( y \) are \emph{compatible via \( z \)}).  A tree \( T \)  is \emph{connected} if every two points in \( T \) are compatible. A subtree \( T' \) of a tree \( T \) is called \emph{maximal connected component of \( T \)} if   
it is connected and such that all the points in \( T \) which are comparable (equivalently, compatible) with an element of \( T' \) belong to \( T' \) themselves.

Notice that if \( T_0, T_1 \) are trees and \( i \) is an embedding of \( T_0 \) 
into \( T_1 \) then for every point \( x \) of \( T_0 \) we have 
\( i``\pred(x) \subseteq \pred(i(x)) \) and 
\( i``\cone(x) \subseteq \cone(i(x)) \): in particular, \( i \) preserves 
(in)comparability. Notice however that compatibility is preserved by \( i \) in the forward direction but, in general, not in the backward direction.

If \( T \) is a DST-tree on \(A \), we call \emph{height of \( T \)} the minimal \( \alpha \in \On \) such that \( \lh(x) < \alpha \) for every \( x \in T \) (such an ordinal must exist because by definition \( T \) is a set). Let \( \kappa \) be a cardinal. If \( T \subseteq \pre{< \kappa}{A} \) is a DST-tree, we call \emph{branch} (of \( T \)) any maximal linearly ordered subset of \( T \). A branch \( b \subseteq T\) is called \emph{cofinal} if the set \( \{ \leng{s} \mid s \in b \} \) is cofinal in \( \kappa \), i.e.\ if \( \bigcup b \in \pre{\kappa}{A} \). We call \emph{body of \( T \)} the set
\begin{align*}
[T] &= \{ f \in \pre{\kappa}{A} \mid \forall \alpha < \kappa \, (f \restriction \alpha \in T) \} \\
& = \Big \{ \bigcup b \mid b \text{ is a cofinal branch of } T \Big \}.
 \end{align*}
When \( X = Y \times \kappa \) we let
\[ 
\proj[T] = \{ f \in \pre{\kappa}{Y} \mid \exists g \in \pre{\kappa}{\kappa} \, ((f,g) \in [T]) \}
 \] 
be the \emph{projection} (on the first coordinate) of the body of \( T \).

\subsection{Standard Borel $\kappa$-spaces} \label{subsec:basicstopology}

Given cardinals \( \lambda \leq \kappa \), we endow the space $\pre{\kappa}{\lambda}$ with the topology 
${\mathscr O} = \mathscr{O}(\pre{\kappa}{\lambda})$ generated by the basis consisting of sets of the form
\begin{equation} \label{eqN_S}
\bN_s = \set{ x \in
\pre{\kappa}{\lambda} }{s \subseteq x }
\end{equation}
for $s \in \pre{< \kappa}{\lambda}$. Finite products of spaces of the form \( \pre{\kappa}{\lambda} \) will be endowed with the corresponding products of the topologies \( \mathscr{O}(\pre{\kappa}{\lambda}) \). The topology \( \mathscr{O} \) is usually called \emph{bounded topology}, and when \( \kappa > \omega \) differs from the product topology of the discrete topology on \(\lambda\). If instead \( \kappa = \omega \), then \( \mathscr{O}(\pre{\kappa}{\lambda}) \) is the usual topology on the Baire space \( \pre{\omega}{\omega} \) and its subspace of the form \( \pre{\omega}{n} \), which are all homeomorphic to the Cantor space \( \pre{\omega}{2}\).
Here we collect some basic properties of the bounded topology \( \mathscr{O} (\pre{\kappa}{\lambda}) \).

\begin{fact}
\begin{enumerate-(1)}
\item
The intersection of fewer than $\cf(\kappa)$ basic open sets is either empty or basic open.
\item
The intersection of fewer than $\cf(\kappa)$ open sets is open.
\item
Each basic open set is closed.
\item
There are exactly $\lambda^{<\kappa}$ basic open sets and $ 2^{(\lambda^{<\kappa})}$ 
open sets in $\pre{\kappa}{\lambda}$.
\nothing{Is the number of open sets really $2^{(\lambda^{<\kappa})}$? Answer by Luca: yes, it is... Indeed one can embed \( \lambda^{< \kappa} \) into itself so that the range consists of pairwise incomparable sequences...}
\item
For each closed subset $C$ of $\pre{\kappa}{\lambda}$ the DST-tree
$T=\{ s\in \pre{<\kappa}{\lambda} \such N_s \cap C \neq \emptyset\}$ is \emph{pruned} (i.e.\ such that for all \( s \in T \) and \( \alpha < \kappa \) there is some \( t \in T \) such that \( \lh(t) = \alpha \) and \( t \) is comparable with \( s \)) and such that 
$[T] = C$. Conversely, for every DST-tree \( T \subseteq \pre{<\kappa}{\lambda} \) the set \( [T] \) is closed in \( \pre{\kappa}{\lambda} \).	
\end{enumerate-(1)}
\end{fact}

When \( \kappa \) is regular, the topology \( \mathscr{O}(\pre{\kappa}{\lambda}) \) is also generated by the basis
\begin{equation}\label{eqbasistau_b}
\mathcal{B} = \{ \bN_s \mid s \colon d \to \lambda \text{ for some } d \in [\kappa]^{< \kappa} \}.
 \end{equation}
This definition of \( \mathscr{O} \) can be easily generalized to arbitrary spaces of the form \( \pre{B}{A} \) where \( |B| = \kappa \) and \( |A| = \lambda \) in the obvious way, i.e.\ we can let \( \mathscr{O } = \mathscr{O}(\pre{B}{A}) \) be the topology on \( \pre{B}{A} \) generated by the basis
\begin{equation}\label{eqbasic} 
\mathcal{B} = \{ \bN_s = \{ x \in \pre{B}{A} \mid s \subseteq x \} \mid s \colon d \to A \text{ for some } d \in [B]^{< \kappa} \}. 
\end{equation} 
It is easy to check that any pair of bijections between, respectively, \( B \) and \( \kappa \) and and \( A \) and \( \lambda \) canonically induce an homeomorphism between the spaces \( \pre{B}{A} \) and \( \pre{\kappa}{\lambda } \).

As noticed e.g.\  in~\cite{mottoros2011}, to have an acceptable descriptive set theory on spaces of the form \( \pre{\kappa}{\lambda} \) for \( \lambda \leq \kappa \) one needs to require at least that
\begin{equation} \label{eq:kappa}
\kappa^{< \kappa} = \kappa.
\end{equation}
For this reason, 
\begin{quote}
\emph{unless otherwise explicitly stated we will tacitly assume throughout this paper that \( \kappa \) is an uncountable cardinal satisfying~\eqref{eq:kappa}}, which implies that \(\kappa\) is regular.
\end{quote}

\begin{definition} \label{def:Borel}
Let $X,Y$ be a topological spaces and \( \mu \) be an ordinal.
\begin{enumerate-(1)}
\item
The \emph{Borel $\mu$-algebra on $X$, ${\bf B}_\mu(X)$} is the smallest subset of ${\mathcal P}(X)$ that contains every open set and is closed under complements and under unions of size $< \mu$.
A set $B \subseteq X$ is \emph{$\mu$-Borel} if it is
in the Borel $\mu$-algebra. 
\item
A function $f \colon X \to Y$ is \emph{$\mu$-Borel
 measurable} if $f^{-1}(U) \in {\bf B}_\mu(X)$ for every open set
$ U \subseteq Y$ (equivalently, \( f^{-1}(B) \in {\bf B}_\mu(X) \) for every \( B \in {\bf B}_\mu(Y) \).
\item
The spaces $X$ and $Y$ are \emph{$\mu$-Borel isomorphic} 
if there is a bijection $f\colon X\to Y$ such that both $f$ and $f^{-1}$ are 
$\mu$-Borel functions.
\end{enumerate-(1)}
\end{definition}

\begin{remark}
When \( \mu = \kappa^+ \) for some cardinal \( \kappa \) satisfying~\eqref{eq:kappa}, we will systematically suppress any reference to it in all the terminology and 
notation introduced in Definition~\ref{def:Borel} whenever \( \kappa \) will be clear from the context (as in the rest of this subsection): therefore, in such a situation the name Borel  will be used as a 
synonym of \( \kappa^+ \)-Borel. 
\end{remark}

Let \( B \in {\bf B}_\mu(X) \) be endowed with the relative topology inherited from \( X \): then
\( {\bf B}_\mu(B) \subseteq {\bf B}_\mu(X) \). Moreover, it is easy to check that any two spaces of the form 
\( \pre{\kappa}{\lambda} \) (for \( \lambda \leq \kappa \) and \( \kappa \) satisfying~\eqref{eq:kappa}) 
are Borel isomorphic, but as noticed in~\cite[Remark 3.5]{mottoros2011} there can be Borel (and 
even closed) subsets of \( \pre{\kappa}{\kappa} \) which are not Borel isomorphic to 
e.g.\ \( \pre{\kappa}{2} \). Notice also that by our assumption~\eqref{eq:kappa}, the collection \( {\bf B}_{\kappa^+}(\pre{\kappa}{\lambda}) \) coincides with the collection of all (\( \kappa^+ \)-)Borel subsets of \( \pre{\kappa}{\lambda} \) when this space is endowed with the product topology instead of \( \mathscr{O}(\pre{\kappa}{\lambda}) \).

%Note that under $|A|^{<\kappa} \geq \mu$ there are open sets 
%that are not the union of $<\mu$ basis open sets.
%\nothing{So we have to reckon that descriptive set theory differs 
%from the case $\kappa=\omega$, in particular in the case of
%$|A|^{<\kappa} =\kappa$ when there are various possibilities for $\mu$.
%}

\begin{definition}[Definition 3.6 in~\cite{mottoros2011}]
A topological space is called a \emph{$\kappa$-space} if it has a basis
 of size $\leq \kappa$.  A $\kappa$-space is called a \emph{standard 
Borel space} if it is ($\kappa^+$-)Borel isomorphic to a 
($\kappa^+$-)Borel subset of $\pre{\kappa}{\kappa}$.
\end{definition} 

\noindent
Thus, in particular,  every space of the form \( \pre{\kappa}{\lambda} \) (for \( \lambda \leq \kappa \)) is a standard Borel \( \kappa \)-space when endowed with \( \mathscr{O}(\pre{\kappa}{\lambda}) \), \emph{provided that \( \kappa \) satisfies~\eqref{eq:kappa}}.

When \( \kappa = \omega \), the notion of standard Borel \( \kappa \)-space coincides with that of a standard 
Borel space as introduced e.g.\ in~\cite[Chapter 12]{Kechris}. The collection of standard Borel 
\( \kappa \)-spaces is closed under Borel subspaces and products of size \( \leq \kappa \), and a reasonable 
descriptive set theory can be developed for these spaces as long as we are interested in results concerning 
only their Borel structure (as we do in this paper). 
%In this kind of situation, one can without loss of generality restrict the attention to one specific standard Borel \( \kappa \)-space, like e.g.\ \( \pre{\kappa}{2} \) or its finite powers.

\begin{definition}
Let \( X \) be a standard Borel \( \kappa \)-space. A set
$A \subseteq X$ is called \emph{analytic} if it is either empty or a continuous image of a closed subset of \( \pre{\kappa}{\kappa} \). The collection of all analytic subsets of \( X \) will be denoted by \( \boldsymbol{\Sigma}^1_1 (X) \). 
\end{definition}

As for the classical case \( \kappa = \omega \), we get that every Borel set is analytic by~\cite[Proposition 3.10]{mottoros2011}, and that a nonempty set \( A \subseteq X \) is analytic if and only if it is a Borel image of a Borel subset of \( \pre{\kappa}{\kappa} \) (equivalently, of a standard Borel space) by~\cite[Proposition 3.11]{mottoros2011}. In particular, \( \boldsymbol{\Sigma}^1_1(B) \subseteq \boldsymbol{\Sigma}^1_1(X) \) for every Borel subset \( B \) of the standard Borel \( \kappa \)-space \( X \). When \( X \) is of the form \( \pre{\kappa}{S} \) for some set \( S \) of cardinality \( \leq \kappa \),
then $X$ is a Hausdorff space and hence a set \( A \subseteq X \) is analytic if and only if \( A = \proj[T] \) for some DST-tree \( T \) on \( S \times \kappa \).
For a proof, see \cite[end of Section 3]{mottoros2011}. We will work with this definition of analytic for the rest of the paper.

\subsection{Infinitary logics and structures of size $\kappa$} \label{subsec:inflogic}

\emph{For the rest of this section, we fix a countable language $\L = \set{R_i }{i \in I }$ ($|I| \leq \omega$) consisting of finitary
relational symbols, and let $n_i$ be the arity of $R_i$.} The symbol $R^X_i$ denotes the interpretation of
 $R_i \in \L$ in the \( \L \)-structure $X$, so that $R^X_i \subseteq \pre{n_i}{X}$. 
With a little abuse of notation, when there is no danger of
confusion the domain of $X$ will be denoted by $X$ again. Therefore, unless otherwise specified the \( \L \)-structure denoted by \( X \) is construed as \( (X, \{ R_i^X \mid i \in I \} ) \), where \( X \) is a set and each \( R^X_i \) is an \( n_i \)-ary relation on \( X \). If  \( X \) is an \( \L \)-structure and \( Y \subseteq X \), we denote by \( X \restriction Y \) the \emph{restriction of the \( \L \)-structure \( X \) to the domain \( Y \)}, i.e., the substructure
\( ( Y, \{ R_i^X \cap {}^{n_i} Y \mid i \in I \}) \).

For cardinals $\lambda \leq \kappa$, we let $\L_{\kappa \lambda}$ be defined as
in \cite{tarski58} (see also~\cite{andmot, mottoros2011} for more details): there are fewer than $\lambda$ free variables in every 
\( \L_{\kappa \lambda} \)-formula, and quantifications can range over $<\lambda$ variables. Conjunctions and disjunctions can be of any size $<\kappa$. We will often use the letters $x,y,z,x_j,y_j,z_j$
(where the \( j \)'s are elements of some set of indexes \( J \) of size \( < \lambda \)) as (meta-)variables, and when writing $\upvarphi(\seq{x_j}{j \in J})$
($|J| < \lambda$) we will always tacitly assume that the
variables $x_j$ are distinct (and similarly with the $y_j$'s and the
$z_j$'s in place of the $x_j$'s). 
%The semantic of
%$\L_{\kappa \lambda}$, as well as all other standard logic
%notions like e.g.\ subformula, sentence, universal closure,
%derivation and so on, are defined in the obvious way. In particular, 
If \( \upvarphi(\langle x_j \mid j \in J \rangle) \) is an \( \L_{\kappa \lambda} \)-formula and \( \langle a_j \mid j \in J \rangle \in \pre{J}{X} \) is a sequence of elements of an \( \L \)-structure \( X \), 
\[ 
X \models \upvarphi[\langle a_j \mid j \in J \rangle]
 \] 
will have the usual meaning, i.e.\ that the formula obtained by replacing each variable \( x_j \) with the corresponding \( a_j \) is true in \( X \).

We naturally identify each \( \L \)-structure \( X \) of size $\kappa$ (up to isomorphism) with an element $y^X = \seq{y^X_i}{i \in I}$ of the space $\Mod^\kappa_\L = \prod_{i \in I}
\pre{({}^{n_i} \kappa)}{2}$, which  is endowed with the product\footnote{The regular product or the $<\kappa$-box product or anything in between are fine, since
  only the Borel structure of the space matters.}
of the topologies \( \mathscr{O}(\pre{({}^{n_i} \kappa)}{2}) \) (this is possible because $|\pre{n_i}{\kappa}| = \kappa$). When \( \kappa \) satisfies~\eqref{eq:kappa}, any bijection $\nu \colon I \times
\kappa \to \kappa$ canonically induces
a natural homeomorphism \label{homeomorphism} between
$\Mod^\kappa_\L$ and $\pre{\kappa}{2}$ (in fact, the requirement that \( \kappa \) satisfies~\eqref{eq:kappa} is not really needed when \( I \) is finite).

\begin{definition}\label{defModkappaphi}
Given an infinite cardinal $\kappa$ and an
$\L_{\kappa^+ \kappa}$-sentence $\upvarphi$, we denote by
$\Mod_\upvarphi^\kappa$ the set of those structures in $\Mod^\kappa_\L$ which are models of $\upvarphi$.
\end{definition}

By the generalized Lopez-Escobar theorem (see e.g.\ \cite[Theorem 4.7]{andmot} and \cite[Theorem
24]{FriedmanHyttinenKulikov}), if \( \kappa \) satisfies~\eqref{eq:kappa} then a set $B \subseteq
\Mod^\kappa_\L$ is Borel and
closed under isomorphism if and only if there is an
$\L_{\kappa^+ \kappa}$-sentence $\upvarphi$ such that $B =
\Mod^\kappa_\upvarphi$.\footnote{As discussed in~\cite[Section 4]{mottoros2011}, both directions of the generalized Lopez-Escobar theorem may fail if \( \kappa^{< \kappa} > \kappa \).}
Therefore, under the usual assumption on \( \kappa \) the space \( \Mod^\kappa_\upvarphi \) is a standard Borel \( \kappa \)-space when endowed with the relative topology inherited from \( \Mod^\kappa_\L \).

We denote by $\sqsubseteq$ the relation of embeddability between \( \L \)-structures, and by $\equiv$ the corresponding relation of bi-embeddability, i.e.\ for \( \L \)-structures \( X \) and \( Y \) we set
$X \equiv Y \iff X \sqsubseteq Y \sqsubseteq X$.%
\footnote{Here is a caveat for model theorists: The relation $\equiv$ is not the usual elementary equivalence!}
The relation of isomorphism between \( \L \)-structures will be denoted by \( \cong \). In this paper, we will mainly be concerned with the restrictions of \( \sqsubseteq \), \( \equiv \), and \( \cong \) to spaces of the form \( \Mod^\kappa_\upvarphi \) for suitable \( \L_{\kappa^+ \kappa } \)-sentences \( \upvarphi \).

\subsection{Analytic quasi-orders and Borel reducibility} \label{subsec:Borelreducibility}

\begin{definition}
Let \( X \) be a standard Borel \( \kappa \)-space. A binary relation \( R \) on \( X \) is called \emph{analytic quasi-order} (respectively, \emph{analytic equivalence relation}) if \( R \) is a quasi-order (respectively, an equivalence relation) and is an analytic subset of the space \( X \times X \).
\end{definition}

When \( \kappa \) satisfies~\eqref{eq:kappa} and \( \upvarphi \) is an \( \Linf \)-sentence, the relations \( {\sqsubseteq} \restriction \Mod^\kappa_\upvarphi \) and \( {\cong} \restriction \Mod^\kappa_\upvarphi \) are (very important) examples of, respectively, an analytic quasi-order and an analytic equivalence relation.

Given a quasi-order \( R \) on \( X \), we denote by \( E_R \) the associated equivalence relation defined by
\( x \mathrel{E_R} y \iff x \mathrel{R} y \wedge y \mathrel{R} x \) (for every \( x,y \in X \)). Notice that if \( R \) is analytic then so is \( E_R \). The partial order canonically induced by \( R \) on the quotient space \( X / E_R \) will be called \emph{quotient order of \( R \)}. To compare the complexity of two analytic quasi-orders we use the (nowadays standard) notion of Borel reducibility.

\begin{definition}
Let $R,S$ be quasi-orders on the standard Borel \( \kappa \)-spaces \( X ,Y \) respectively. A \emph{reduction} of \( R \) to \( S \) is a function \( f \colon X \to Y \) such that for every \( x,y \in X \)
\[ 
x \mathrel{R} y \iff f(x) \mathrel{S} f(y).
 \] 
When \( R \) and \( S \) are analytic quasi-orders, we say that \( R \) is \emph{Borel reducible} to \( S \) (in symbols \( R \leq_B S \)) if there is a Borel reduction of \( R \) to \( S \), and that \( R \) and \( S \) are \emph{Borel bi-reducible} (in symbols \( R \sim_B S \)) if \( R \leq_B S \leq_B R \).
\end{definition}

By~\cite[Lemma 6.8]{mottoros2011}, every analytic quasi-order is Borel bi-reducible with (in fact, even classwise Borel isomorphic to, see below for the definition) an analytic quasi-order defined on the whole \( \pre{\kappa}{2} \). Therefore, when we are interested in analytic quasi-orders up to these notions of equivalence, as we do here, we can restrict our attention to analytic quasi-orders on \( \pre{\kappa}{2} \).

\begin{definition}
An analytic quasi-order $S$ on a standard Borel \( \kappa \)-space \( X \) is said to be \emph{complete} if 
$R\leq_BS$ for every analytic quasi-order $R$, and similarly 
for equivalence relations.
\end{definition}

Notice that under assumption~\eqref{eq:kappa} there are universal analytic sets by e.g.\ \cite[Lemma~3]{MeklerVaananen}, and therefore the proof of \cite[Proposition~1.3]{louveau-rosendal} shows that
then there are also complete quasi-orders and equivalence relations on \( \pre{\kappa}{2} \). 

When \( S \) is of the form \( {\sqsubseteq} \restriction \Mod^\kappa_\upvarphi \), the notion of completeness can be naturally strengthened to the following.

\begin{definition}[Definition 6.5 in~\cite{mottoros2011}] \label{def:invariantuniversality}
Let $\kappa$ be an infinite cardinal satisfying~\eqref{eq:kappa}, $\L$ be a countable relational 
language, and $\upvarphi$ be an $\Linf$-sentence.
 The embeddability relation $ {\sqsubseteq} \restriction \Mod^\kappa_\upvarphi$ is called 
\emph{invariantly universal} if for every analytic quasi-order $R$ 
there is an $\Linf$-sentence $\uppsi$ such that 
$\Mod^\kappa_\uppsi \subseteq\Mod^\kappa_\upvarphi$ (i.e.\ such that \( \uppsi \Rightarrow \upvarphi \)) and $R \sim_B 
{{\sqsubseteq} \restriction \Mod^\kappa_\uppsi}$.

Invariant universality of \( {\equiv} \restriction \Mod^\kappa_\upvarphi \) is defined in a similar way by replacing quasi-order $\sqsubseteq$  with the  equivalence relation $\equiv$.
\end{definition}

Notice: For \( \upvarphi \) such that \( {\sqsubseteq} \restriction \Mod^\kappa_\upvarphi \) is invariantly universal, also the relation \( {\equiv} \restriction \Mod^\kappa_\upvarphi \) is invariantly universal as well, and both relations are complete.	

If \( R \) and \( S \) are analytic quasi-orders such that \( R \sim_B S \), then their quotient orders are mutually embeddable, but not necessarily isomorphic. Based on this observation, it is natural to introduce
the following strengthening of the notion of Borel bi-reducibility.

\begin{definition}[Definition 6.6 in~\cite{mottoros2011}]
Let $X$, $Y$ be two standard Borel $\kappa$-spaces and 
$R$ and $S$ be analytic quasi-orders on $X$, $Y$ respectively.  We say that 
$R$ and $S$ are \emph{classwise Borel isomorphic} (in symbols $R \simeq_B S$)
 if there is an isomorphism \( f \colon X/E_R \to Y/E_S \) between the quotient orders of $R$ and $S$ such that both $f$ and 
$f^{-1}$ admit Borel liftings.
\end{definition}

Replacing Borel bi-reducibility with classwise Borel isomorphism in Definition~\ref{def:invariantuniversality} we get the following notion.

\begin{definition}[Definition 6.7 in~\cite{mottoros2011}] \label{def:stronglyinvariantlyuniversal}
Let $\kappa$, \( \L \) and \( \upvarphi \) be as in Definition~\ref{def:invariantuniversality}.
The relation of (bi-)embeddability on $\Mod_\upvarphi^\kappa$ is called 
\emph{strongly invariantly universal} if for every analytic quasi-order 
(respectively, equivalence relation) $R$ 
there is an $\Linf$-sentence $\uppsi$ such that 
$\Mod^\kappa_\uppsi \subseteq\Mod^\kappa_\upvarphi$ and $R\simeq_B 
{{\sqsubseteq} \restriction \Mod^\kappa_\uppsi}$ (respectively, $R\simeq_B 
{{\equiv} \restriction \Mod^\kappa_\uppsi}$).
\end{definition}

As for invariant universality, we again have that if \( \upvarphi \) is such that \( {\sqsubseteq} \restriction \Mod^\kappa_\upvarphi \) is invariantly universal, then so is \( {\equiv} \restriction \Mod^\kappa_\upvarphi \).

\section{The quasi-order $\leq_{\mathrm{max}}$} \label{sec:max}

Following
%\footnote{The reader may notice that our definitions of \( \rho \) and \( \oplus \) are slightly different from those in~\cite{mottoros2011}. However, it can be easily checked that these modifications do not affect the proofs of the results from~\cite{mottoros2011} that will be needed here.}
~\cite[Section 7]{mottoros2011}, 
we let $\rho \colon \On \times \On \to \On \setminus \{ 0 \}
\colon (\gamma, \gamma') \mapsto \Hes(\gamma,
\gamma')+1$, so that in particular $\rho `` \, \omega \times \omega =
\omega \setminus \{ 0 \}$, and for every infinite cardinal
$\kappa$ and $\gamma, \gamma' < \kappa$ one has $\gamma,
\gamma' \leq \rho(\gamma, \gamma') < \kappa$. Then we define by recursion on $\gamma \leq \kappa$ a Lipschitz (i.e.\ a monotone%
\footnote{With respect to the end-extension order on \( {}^{\leq \kappa} (\kappa \times \kappa) \).} 
and length preserving) map
$\oplus \colon {}^{\leq \kappa} (\kappa \times \kappa) \to
{}^{\leq \kappa} \kappa$ as follows: 
If $x \in {}^{\leq \kappa} (\kappa \times \kappa)$ for $\alpha \in \dom(x)$,
we let $x(\alpha) = (s(\alpha), t(\alpha))$. Henceforth we write $s \oplus t $ for $\oplus(x)$.
\begin{description}
\item[$\gamma = 0$] $\emptyset \oplus \emptyset = \emptyset$;
\item[$\gamma = 1$] if $s = \langle \mu \rangle$ and $t = \langle \nu \rangle$, then $s \oplus t = \langle \rho(\mu,\nu) \rangle$;
\item[$\gamma = \gamma'+1$ for $\gamma' \neq 0$] let $s , t \in {}^{\gamma} \kappa$. Then 
\[ s \oplus t = (s' \oplus t') {}^\smallfrown{}  \rho\left( (\sup_{\alpha \leq \gamma'}s(\alpha))+\omega, \rho(s(\gamma'),t(\gamma'))\right), \]
where $s' = s \restriction \gamma'$ and $t' = t \restriction \gamma'$;
\item[$\gamma$ limit] $s \oplus t = \bigcup_{\alpha < \gamma} (s \restriction \alpha \oplus t \restriction \alpha )$.
\end{description}

Our next goal is to drop the requirement that \( \kappa \) be inaccessible in the second part of~\cite[Proposition 7.1]{mottoros2011}.

\begin{proposition}\label{propoplus}
Let \( \kappa \) be a regular  uncountable cardinal and let \( \oplus \) be defined as above.
\begin{enumerate-(i)}
\item \label{propoplusc1}
The map $\oplus$ is injective.
\item \label{propoplusc2}
If \( \kappa \) further satisfies~\eqref{eq:kappa}, then there is a map \( \# \colon \pre{\SUCC(< \kappa)}{\kappa} \to \kappa \) such that
\begin{enumerate-(a)}
\item \label{propoplusc2a}
for every \( s,t \in \pre{\SUCC(<\kappa)}{\kappa} \) such that \( \leng{s} = \leng{t} \)
\[ 
\# s \leq \#(s \oplus t);
 \] 
\item \label{propoplusc2b}
for every \( \gamma < \kappa \), \( \# \restriction \pre{\gamma+1}{\kappa} \) is a bijection between \( \pre{\gamma+1}{\kappa} \) and \( \kappa \).
\end{enumerate-(a)}
\end{enumerate-(i)}
\end{proposition}

\begin{proof}
Part~\ref{propoplusc1} follows from the injectivity of \(\rho\). For part~\ref{propoplusc2}, we define the function \( \# \) separately  on each \( \pre{\gamma+1}{\kappa} \) (for \( \gamma < \kappa \)).
The case $\gamma = 0$ is trivial, as one can simply take \( \# \) to be the identity function, so let us assume that $\gamma \geq 1$. Let \( \zeta_\gamma \colon \pre{\gamma+1}{\kappa} \to \kappa \) be any bijection. Define
\[ 
\sigma_\gamma \colon \pre{\gamma+1}{\kappa} \to \kappa \times \kappa \colon s \mapsto (\sigma_\gamma^0(s), \sigma_\gamma^1(s))
 \] 
by induction on the well-order of \( \pre{\gamma+1}{\kappa} \) given by
\[ 
s \preceq t \iff \sup(s) < \sup(t) \vee (\sup(s) = \sup(t) \wedge \zeta_\gamma(s) \leq \zeta_\gamma(t)).
 \] 
Given \( s \in \pre{\gamma+1}{\kappa} \), set \( \sigma_\gamma^0(s) = \sup(s) \). By definition of \(\oplus \) we have for any $r,t$, \( \sigma_\gamma^0(r) = \sup(r) < \sup(r \oplus t) = \sigma_\gamma^0(r \oplus t) \). Hence when considering a sequence of the form \( r \oplus t \) we may assume that \( \sigma^1_\gamma (r) \) is already defined. Let 
\[ 
\pi(s) = \sup \{ \sigma_\gamma^1(t) \mid t \in \pre{\gamma+1}{\kappa}, \sup(t) = \sup(s), \zeta_\gamma(t) < \zeta_\gamma(s) \} +1,
\]
and set
\[ 
\sigma^1_\gamma(s) = 
\begin{cases}
\pi(s) & \text{if \( s \) is not of the form } r \oplus t \\
\Hes(\sigma^1_\gamma(r), \pi(s)) & \text{if } s = r \oplus t.
\end{cases}
\]
Then we have that $\sigma_\gamma$ is injective.
Suppose $s_1  \neq s_2 \in {}^{\gamma+1} \kappa$ are given. We assume that
$\zeta_\gamma(s_1) < \zeta_\gamma(s_2)$.
Now there are four cases:

Case 1: For $i=1,2$, $s_i$ is not of the form $r_i \oplus t_i$. Then either
$\sigma_\gamma^0(s_1) \neq \sigma^0_\gamma(s_2)$ and we are done,
or $\sup(s_1)=\sigma_\gamma^0(s_1) = \sigma^0_\gamma(s_2)=\sup(s_2)$. 
In this case, by the definition of the mapping $\pi$,  $\sigma^1_\gamma(s_1)=\pi(s_1) < \pi(s_2)=\sigma^1_\gamma(s_2)$.

Case 2: For $i=1,2$, $s_i$ is of the form $r_i \oplus t_i$. If
$\sup(s_1) \neq \sup(s_2)$, we are done. So we assume that $\sup(s_1) = \sup(s_2)$. Then   $\sigma^1_\gamma(s_2) = \Hes(\sigma^1_\gamma(r_2),\pi(s_2)) \geq \pi(s_2) > \sigma^1_\gamma(s_1)$.

Case 3: $s_1$ is of the form $r_1 \oplus t_1$,
and $s_2$ is not of the form $r_2 \oplus t_2$.
Again we assume that $\sup(s_1) = \sup(s_2)$.
Then $\sigma^1_\gamma(s_2) = \pi(s_2) > \sigma^1_\gamma(s_1)$.

Case 4: $s_1$ is not of the form $r_1 \oplus t_1$,
and $s_2$ is of the form $r_2 \oplus t_2$.
Again we assume that $\sup(s_1) = \sup(s_2)$.
Then $\sigma^1_\gamma(s_2) = \Hes(\sigma^1_\gamma(r_2), \pi(s_2))
\geq  \pi(s_2) > \sigma^1_\gamma(s_1)$.

For every $r,t \in \pre{\gamma+1}{\kappa}$ 
\[
\sigma^0_\gamma(r) < \sigma^0_\gamma(r \oplus t) \quad \text{and} \quad \sigma^1_\gamma(r) < \sigma^1_\gamma(r \oplus t),
\]
whence
\[
\Hes(\sigma_\gamma(r)) < \Hes(\sigma_\gamma(r\oplus t)).
\]
Let $\# \restriction \pre{\gamma+1}{\kappa}$ be the collapsing map of \( \Hes \circ \sigma_\gamma  \), i.e.\ the map recursively defined by setting for every \( s \in \pre{\gamma+1}{\kappa} \) 
\[ 
\#(s) = \max (0, \sup \{ \#(t) +1 \mid t \in \pre{\gamma+1}{\kappa}, \Hes(\sigma_\gamma(t)) < \Hes(\sigma_\gamma(s) \}).
 \] 
Then the resulting \( \# \) is as required.
\end{proof}

%Let $b \colon \kappa \times \kappa \to \kappa$ \todo{\( b \) is actually denoted by \( \Hes \) in the previous parts...} be a bijection
%as in the proof of Hessenberg's theorem, see e.g. \cite[p.~30]{Jech3}.
%We let $B \colon \{(f,g) \in 
%\kappa^{\leq \kappa} \times \kappa^{\leq \kappa}\such \lh(f) = \lh(g)\} \to
%\kappa^{\leq\kappa} $ be defined
%by: For $\alpha < \lh(f)$,  $B(f,g)(\alpha) = b(f(\alpha),g(\alpha))$. 
%Then $\lh(B(f,g))=\lh(f)$, and $B\rest{}^\kappa \kappa$ is continuous and 
%$B \rest{}^{<\kappa} \kappa$ is {\em Lipshitz}, that is, it is injective,
%it preserves the 
%lengths, i.e., $\lh(\phi(s))=\lh(s)$, and  
%and it preserves extension, i.e., $s\subseteq t$ implies $\phi(s) \subseteq \phi(t)$.
% Moreover for every $\alpha, \beta\in \kappa$
%$b(\alpha,\beta) \geq\alpha, \beta$ and hence
%if we define $\leq$ on ${}^{\leq\kappa}  \kappa$ by $f\leq g$ iff for every 
%$\alpha \in \dom(f)\cap \dom(g)$,
%$f(\alpha)\leq g(\alpha)$, then $f,g\leq B(f,g)$.

The following is a modification of known constructions (see~\cite[Prop.~2.4]{louveau-rosendal} and~\cite[Prop.~2.1]{friedman-mottoros} for \( \kappa = \omega \), and~\cite[Section~7]{mottoros2011} for a weakly compact $\kappa$;
similar constructions for uncountable \( \kappa \)'s may also be found in~\cite{andmot}.

Let \( \kappa \) be an arbitrary uncountable cardinal. For \( \emptyset \neq u \in \pre{<
\kappa}{2} \) set \( u^- = u \restriction \leng{u}-1 \) if \( \leng{u} \) is finite
and \( u^- = u \) otherwise. Similarly, for \( \gamma \in \On \) we set \( \gamma^- = \gamma -1 \) if \( \gamma < \omega \) and \( \gamma^- = \gamma \) otherwise. 
Moreover, consider the variant \( \widetilde{\oplus} \colon  \pre{ \leq \kappa}{(\kappa \times \kappa)} \to \pre{\leq \kappa}{\kappa} \) of \( \oplus \) defined by
\[ 
s \mathrel{\widetilde{\oplus}} t = \langle ({0 {}^\smallfrown{} s} \oplus {0 {}^\smallfrown{} t})(1+\gamma) \mid \gamma < \leng{s}. \rangle
 \] 
Then \( \widetilde{\oplus} \) is monotone and on the infinite sequences it is lengths preserving (since \( \oplus \) is such a function), and it is straightforward  to check that for every \( n,m \in \omega \) and \( s,t \in \pre{\leq \kappa}{\kappa} \) it holds
\begin{equation} \label{eq:variantoplus}
{( n {}^\smallfrown{} s )} \oplus {( m {}^\smallfrown{}  t )} = \rho(n,m) {}^\smallfrown{}  ( s \mathrel{\widetilde{\oplus}} t ).
\end{equation}
(This uses the definition of \( \oplus \) in the successor step: in fact, \( (( n {}^\smallfrown{} s ) \oplus ( m {}^\smallfrown{}  t )) (1)= \rho(s(0), t(0))  \) if one of \( s(0), t(0) \) is infinite, and \( (( n {}^\smallfrown{} s ) \oplus ( m {}^\smallfrown{}  t ))(1) = \omega \) otherwise. In any case, the value of \( (( n {}^\smallfrown{} s ) \oplus ( m {}^\smallfrown{}  t ))(1) \) is independent from the natural numbers \( n,m \) we are using.)

Given a DST-tree \( T \) on \( 2 \times 2 \times \kappa \) of
height \(  \leq \kappa \), let 
\[ 
\hat{T} = T \cup \set{ (u,u,0^{(\gamma)}) \in
{}^{< \kappa} 2 \times {}^{< \kappa} 2 \times {}^{< \kappa}
\kappa }{\leng{u} = \gamma } .
\]
Then inductively define \( S^T_n \) as
follows:
\begin{align*}
S^T_0 = & \{ \emptyset \} \cup \left \{ (u,v,0 {}^\smallfrown{}  s)  \mid  (u^-,v^-,s) \in
\hat{T} \right\}\\
S^T_{n+1} = & \{ \emptyset \} \cup \left\{ (u,v,(n+1) {}^\smallfrown{}  s)  \mid  (u,v,n
{}^\smallfrown{}  s) \in S^T_n \right\} \cup \\
& \cup \left\{ (u,w,(n+1) {}^\smallfrown{}  s \mathrel{\widetilde{\oplus}} t)  \mid  \exists v
\left[(u,v,n {}^\smallfrown{}  s),(v,w,n {}^\smallfrown{}  t) \
\in S^T_n \right] \right\} .
\end{align*}
Finally, set 
\begin{equation} \label{eqS_T}
S_T = \bigcup\nolimits_n S^T_n .
%\setminus \left\{ (u,v,0^{(\leng{u})})
%\in \pre{< \kappa}{2} \times \pre{< \kappa}{2} \times \pre{< \kappa}
%{\kappa}  \mid  u \neq v \right\}. 
\end{equation}
Notice that \( \bigcup_n S^T_n \) and  \( S_T \) are  DST-trees on
\( 2 \times 2 \times \kappa \) (because an easy induction on \( n \in
\omega \) shows that each \( S^T_n \) is a DST-tree on the same space) of
height \( \leq \kappa \), and that if \( (u,v,s) \in S_T \setminus \{ \emptyset \} \) then \( s(0) \in \omega \).

\begin{lemma}\label{lemmanormalform}
Let \( \kappa \) be an uncountable cardinal, and \( R \) be an analytic quasi-order. Then there is a DST-tree \( T \) on \( 2 \times 2 \times \kappa \) such that \( R = \proj[T] \) and the following conditions hold:
\begin{enumerate-(i)}
\item \label{lemmanormalformc1}
 for any $t \in \pre{< \kappa}{\kappa}$ there is at most one $(u,v)$ with $(u,v,t) \in T$;
\item \label{lemmanormalformc2}
for every \( u,v \in {}^{< \kappa} 2 \),  \( (u,u,0^{(\lh(u))}) \in S_T \), and if \( (u,v,0^{(\lh(u)}))\in S_T \) with \( \lh(u) \geq \omega \), then \( u = v \); 
\item \label{lemmanormalformc3}
if \( u,v,w \in {}^{< \kappa} 2 \) and \( s,t
\in {}^{< \kappa} \kappa \) are such that \( (u,v,s) , (v,w,t) \in S_T \) then \( (u,w,s \oplus t) \in S_T \);
\item\label{lemmanormalformc4} 
for all \( s \in \pre{< \kappa}{\kappa} \) of infinite length \(\gamma\), either \( s = 0^{(\gamma)} \), or else there are only finitely many pairs \( (u,v) \in \pre{\gamma}{2} \times \pre{\gamma}{2} \) such that \( u \neq v \) and \( (u,v,s) \in S_T \);
\item \label{lemmanormalformc5}
\( R = \proj[S_T] = \set{ (x,y) \in \pre{\kappa}{2} \times
\pre{\kappa}{2} }{ \exists \xi \in \pre{\kappa}{\kappa} \left
((x,y,\xi) \in [S_T] \right) } \).
\end{enumerate-(i)}
\end{lemma}

\begin{proof}
\ref{lemmanormalformc1}
Let $R = p[T']$ for some tree 
$T' $ on $2 \times 2 \times \kappa$. Let \( \Hes_3 \colon \On \times \On \times \On \to \On \) be the bijection defined by \( \Hes_3(\alpha,\beta,\gamma) = \Hes( \alpha , \Hes(\beta. \gamma)) \) (so that, in particular, \( \Hes_3(0,0,0) = 0 \)), and let \( \Hes_3^{\uparrow} \) be its (coordinatewise) extension to \( \pre{\leq \kappa}{(\On \times \On \times \On)} \), namely for all \( r,s,t \in \pre{\leq \kappa}{(\On \times \On \times \On)} \) of the same length let \( \Hes_3^{\uparrow}(r,s,t) = \langle \Hes_3(r_i,s_i,t_i) \mid i < \lh(s) \rangle \).
Set 
\[
T = \{ (u,v,\Hes_3^{\uparrow}(u,v,s)) \mid (u,v,s) \in T'\}.
\]
Then
\begin{equation} \label{eq:uniquewitness}
\text{for all } t \in \pre{< \kappa}{\kappa} \text{ there is at most one  } (u,v) \in \pre{< \kappa}{2} \times \pre{< \kappa}{2} \text{ with } (u,v,t) \in T.
\end{equation}
Namely, \( u,v \) are the unique sequences such that \( t = \Hes_3^{\uparrow}(u,v,s) \) for some \( s \in \pre{< \kappa}{\kappa} \).

Then we have \( \proj[T] = \proj[T'] = R \). Indeed, if \( (x,y,\xi) \in [T'] \) then \( (x,y,\Hes_3^{\uparrow}(x,y,\xi)) \in [T] \), and conversely, using the fact that \( \Hes_3 \) is injective and the definition of \( T \), from any \( \xi \) witnessing \( (x,y) \in \proj[T] \) we can decode a unique \( \xi' \) (namely, the unique \( \xi' \) such that \( \xi = \Hes_3^{\uparrow}(x,y,\xi') \)) witnessing \( (x,y) \in \proj[T'] \).

For property~\ref{lemmanormalformc2}, observe  that for all \( u \in \pre{< \kappa}{2} \) we have \( (u,u, 0^{(\lh(u))}) \in \hat{T} \), whence \( (u,u, 0^{(\lh(u))}) \in S^T_0 \subseteq S_T\). Conversely, given \(u,v \) with \( \lh(u) = \lh(v) = \gamma \geq \omega \) we have \( (u,v,0^{(\gamma)}) \in S_T \) if and only if \( (u,v,0^{(\gamma)}) \in S^T_0 \), if and only if \( (u^-, v^-, 0^{(\lh(u^-))}) \in \hat{T} \). Since \(  (u^-, v^-, 0^{(\lh(u^-))}) \in T \) implies \( u^- = v^- = 0^{(\lh(u^-))} \) by \( \Hes_3^{-1}(0) = (0,0,0) \), by the definition of \( \hat{T} \) we have \( u^- = v^- \), whence  \( u = v \) because \( u = u^- \) and \( v = v^- \) by \( \lh(u) = \lh(v) \geq \omega \).

To see that  also satisfies~\ref{lemmanormalformc3} is satisfied, we argue as in the proof of~\cite[Lemma 7.2]{mottoros2011}.  Clearly  we can assume that both \( s \) and \( t \) are nonempty, and hence let \( s',t' \in \pre{< \kappa}{\kappa} \) and \( n,m \in \omega \) be such that \( s = n {}^\smallfrown{}  s' \) and \( t = m {}^\smallfrown{}  t' \), so that, in particular
 \( (u,v,s) \in S^T_n \) and
\( (v,w,t) \in S^T_m \). Since%
\footnote{Notice that \( \rho(n,m) -1 \) is always defined because \( \rho(n,m) \) is by definition a successor ordinal.}
 \( n,m \leq
\rho(n,m)-1 = k \), we have that both \( (u,v, k {}^\smallfrown{}  s') \) and
\( (v,w,k {}^\smallfrown{}  t') \) belong to \( S^T_k \): hence 
\begin{multline*} 
(u,w,(k+1)
{}^\smallfrown{}  (s' \mathrel{\widetilde{\oplus}} t')) = (u,w, \rho(n , m)
{}^\smallfrown{}  (s' \mathrel{\widetilde{\oplus}} t')) \\ = (u,w,s \oplus t) \in S^T_{\rho(n,m)} \subseteq \bigcup\nolimits_l S^T_l  = S_T
\end{multline*}
by
definition of the \( S^T_n \)'s and~\eqref{eq:variantoplus}. % Since \( (s \oplus t) (0) = \rho(n,m) \neq 0 \), \( (u,w, s \oplus t) \in S_T \), as required.

Let us now consider condition~\ref{lemmanormalformc4}. Fix \( s \in \pre{< \kappa}{\kappa} \) of length \( \gamma \geq \omega \) such that \( s \neq 0^{(\gamma)} \). If \( (u,v,s) \in S_T \) for some \( u,v \in \pre{< \kappa}{2} \), then \( s(0) = n < \omega \). Moreover, by definition of the \( S^T_n \) we have \( (u,v,s) \in S_T \iff (u,v,s) \in S^T_n \), so it is enough to prove by induction on \( n < \omega \) that 
the set
\begin{equation}\label{eq:finite}
\{ (u,v) \in \pre{\gamma}{2} \times \pre{\gamma}{2} \mid u \neq v \wedge (u,v,s) \in S^T_n \}
 \end{equation}
is finite. Let \( s' \) be the unique sequence such that \( s = n {}^\smallfrown{}  s' \). The case \( n = 0 \) is easy. Since \( s \neq 0^{(\gamma)} \), we have \( (u,v,s) \in S^T_0 \iff (u^-, v^-, s')  \in T \), and since \( \lh(u) = \lh(v) = \lh(s) = \gamma \geq \omega \) we have that \( u = u^- \) and \( v = v^- \). By~\eqref{eq:uniquewitness}, there is at most one pair \( (u,v) \) that can satisfy \( (u,v,s') \in T \), hence we are done. 
Let now consider the inductive step \( n = k+1 \). By definition of \( S^T_n \), the set in~\eqref{eq:finite} is the union of 
 \begin{equation} \label{eq:finite1}
 \{ (u,v)  \in \pre{\gamma}{2} \times \pre{\gamma}{2} \mid u \neq v \wedge (u,v,k {}^\smallfrown{} s') \in S^T_k\}
 \end{equation}
 and
 \begin{equation} \label{eq:finite2}
 \{ (u,v) \in \pre{\gamma}{2} \times \pre{\gamma}{2} \mid u \neq v \wedge \exists w \,  ( (u,w,k {}^\smallfrown{} t_0) \in S^T_k \wedge (w,v,k {}^\smallfrown{} t_1) \in S^T_k \},
 \end{equation}
 where \( t_0, t_1 \) are the unique sequences in \( \pre{\lh(s')}{\kappa} \) such that \( s' = t_0 \mathrel{\widetilde{\oplus}} t_1 \).
If \( k \neq 0 \) or \( s' \neq 0^{(\lh(s'))} \), the first set is finite by inductive hypothesis; otherwise it is empty by~\ref{lemmanormalformc2}. Let us now consider the second set. By inductive hypothesis, if \( k \neq 0 \) or \( t_0, t_1 \neq 0^{(\lh(s'))} \) then there are only finitely many \( w \) such that \(  (u,w,k {}^\smallfrown{} t_0) \in S^T_k \), and finitely many \( w \) such that \( (w,v,k {}^\smallfrown{} t_1) \in S^T_k \), whence also the set in~\eqref{eq:finite2} is finite and we are done. If \( k = 0 \) and \(  t_1 = 0^{(\lh(s'))} \) (so that \( k {}^\smallfrown{} t_1 = 0^{(\gamma)} \)), then by~\ref{lemmanormalformc2} any \( w \) such that
 \( (w,v,k {}^\smallfrown{}  t_1) \in S^T_k \) must equal  \( v \). Therefore the set in~\eqref{eq:finite2} reduces to 
\[ 
\{ (u,v) \in  \pre{\gamma}{2} \times \pre{\gamma}{2} \mid u \neq v \wedge (u,v, 0 {}^\smallfrown{} t_0) \in S^T_0  \} .
\] 
If \( t_0 \neq 0^{(\lh(s'))} \), then the latter set is finite  by inductive hypothesis; if instead \( t_0 = 0^{(\lh(s'))} \), then  by~\ref{lemmanormalformc2} such set is empty, and thus so is the set  in~\eqref{eq:finite2}.
The case when \( k = 0 \) and \(  t_0 = 0^{(\lh(s'))} \) can be dealt similarly, hence we are done.

It remains only to prove~\ref{lemmanormalformc5}. Arguing as in the proof of~\cite[Claim 7.2.1]{mottoros2011}, we have \( R = \proj[\hat{T}] = \proj[S^T_0] \subseteq \proj[\bigcup_{n } S^T_n] = \proj[S_T] \) because \( R \) is reflexive
and \( R = \proj[T] \). Since every branch $(x,y,\xi)$ of $\bigcup_{n } S^T_n$
is a branch of $S^T_{\xi(0)}$, we have
\(\proj[\bigcup_{n } S^T_n] = \bigcup_n \proj[S^T_n]\).
Hence for the reverse inclusion is enough to prove by induction on \( n \) that
\( \proj[S^T_n] \subseteq R \). 
The case \( n=0 \) is obvious because \( \proj[S^T_0] = \proj[\hat{T}] =
R \), so assume
\( \proj[S^T_n] \subseteq R \), choose an arbitrary \( (x,y) \in
\proj[S^T_{n+1}] \) and
let \( \xi \in \pre{\kappa}{\kappa} \) be such that \( (x,y,(n+1)
{}^\smallfrown{} 
\xi) \in [S^T_{n+1}] \). We distinguish two cases: if for
cofinally many \( \gamma
< \kappa \) we have \( (x \restriction \gamma, y \restriction
\gamma , n
{}^\smallfrown{}  {\xi \restriction \gamma^-}) \in S^T_n \)
 then \( (x,y,n
{}^\smallfrown{}  \xi)
\in [S^T_n] \), so that \( (x,y) \in \proj[S^T_n] \subseteq R \) by inductive hypothesis.
Otherwise, for almost
all \( \gamma < \kappa \) (hence for every \( \gamma < \kappa \),
since \( S^T_n \) is a
DST-tree) there is a \( v_\gamma \in {}^{< \kappa} 2 \) such that \( (x
\restriction
\gamma , v_\gamma, n {}^\smallfrown{}  \xi_0 \restriction
\gamma^-), (v_\gamma, y
\restriction \gamma, n {}^\smallfrown{}  \xi_1 \restriction
\gamma^-) \in S^T_n \),
where%
\footnote{Such  \( \xi_0 \) and \( \xi_1 \) exist and are unique by the fact that \( \mathrel{\widetilde{\oplus}} \) is injective and that clearly \( (s \mathrel{\widetilde{\oplus}} t ) \restriction \alpha = (s \restriction \alpha) \mathrel{\widetilde{\oplus}} (t \restriction \alpha) \) for every \( \alpha \leq \leng{s} = \leng{t} \).}
 \( \xi_0,\xi_1 \in \pre{\kappa}{\kappa} \) are  such that
\( \xi = \xi_0 \mathrel{\widetilde{\oplus}} \xi_1 \). 
%
%If \( v_\gamma = x \restriction \gamma \) for cofinally many \( \gamma < \kappa \), then \( (x \restriction \gamma, y \restriction \gamma, n {}^\smallfrown{}  \xi_1 \restriction (\gamma-1)) \in S^T_n \) for those \(\gamma\)'s and hence \( (x,y) \in \proj[S^T_n] \subseteq R \) (by inductive hypothesis). Similarly, if \( v_\gamma = y \restriction \gamma \) for cofinally many \( \gamma < \kappa \), then \( n {}^\smallfrown{} x_0 \) would be a witness for \( (x,y) \in \proj[S^T_n] \subseteq R \). Therefore for all large enough \( \gamma < \kappa \) we can assume \( v_\gamma \neq x \restriction \gamma, y \restriction \gamma \).
%

Assume first that \( n = 0 \) and \( \xi_0 = 0^{(\kappa)} \). Then by~\ref{lemmanormalformc2} we have \( v_\gamma = x \restriction \gamma \) for all \( \gamma < \kappa \), and thus \( n {}^\smallfrown{}  \xi_1 \) is a witness for \( (x,y) \in \proj[S^T_n] \subseteq R \) (the latter inclusion follows from the inductive hypothesis). The case \( n = 0 \) and \( \xi_1 = 0^{(\kappa)} \) is similar, hence we can assume without loss of generality
that \( n \neq 0 \) or both \( \xi_0 \) and \( \xi_1 \) are different from \( 0^{(\kappa)} \). Consider the DST-tree 
\[ 
V = \{ w \in \pre{< \kappa}{2} \mid \exists \gamma < \kappa \, (\gamma \geq \omega \wedge n {}^\smallfrown{}  \xi_0 \restriction \gamma^- \neq 0^{(\gamma)} \wedge  n {}^\smallfrown{}  \xi_1 \restriction \gamma^- \neq 0^{(\gamma)} \wedge w \subseteq v_\gamma) \} 
\]
generated by all large enough
\( v_\gamma \)'s. It  is a subtree of \( {}^{< \kappa} 2 \) of
height \( \kappa \) (as
\( \leng{v_\gamma} = \gamma \)). Let \( \gamma \geq \omega \) be smallest such that \( n {}^\smallfrown{}  \xi_0 \restriction \gamma^-, n {}^\smallfrown{}  \xi_1 \restriction \gamma^- \neq 0^{(\gamma)} \). Assume towards a contradiction that there is 
\( \gamma \leq \alpha < \kappa \) such that \( V_{=\alpha} =  \{ w \in V \mid \lh(w) = \alpha \} \) is infinite. Then infinitely many elements of \( V_{=\alpha} \) would be different from both \( x \restriction \alpha \) and
 \( y \restriction \alpha \), and since \( n {}^\smallfrown{} \xi_0 \restriction \alpha^-, n {}^\smallfrown{}  \xi_1 \restriction \alpha^- \neq 0^{(\alpha)} \) and 
 \( (x \restriction \alpha, w, n {}^\smallfrown{} \xi_0 \restriction \alpha^-) \in S^T_n \) and \( (w, y \restriction \alpha , n {}^\smallfrown{} \xi_1 \restriction \alpha^-) \in S^T_n \) for all such \( w \) (because 
 all of them are restrictions of the \( v_\gamma \)'s), this would contradict property~\ref{lemmanormalformc4}. It follows that all \( V_{= \alpha} \) with \( \alpha \geq \gamma \) are finite, and so 
 \( V \subseteq \pre{<\kappa}{2} \) is a tree of height \( \kappa \) all of whose levels \( V_{= \alpha} \) are finite (for levels \( < \gamma \) notice that they consist exactly of the restriction of the sequences in 
 \( V_{=\gamma} \), hence they are finite as well).

\begin{claim} \label{claim:narrowtrees}
If \( T \) is a descriptive set-theoretical tree of infinite height \( \kappa \) and all of whose levels are finite, then there is a cofinal branch \( z \) through \( T \).
\end{claim}

\begin{proof}
  This follows for $\cf(\kappa) = \omega$ from K\"onig's lemma and for $\cf(\kappa) > \omega$ from a theorem by Kurepa, see \cite[Proposition 7.9]{Kanamori}
  or \cite[Proposition 2.32Ac]{Levy}, the latter explicitly includes
  singular $\kappa$ of uncountable cofinality.
  \nothing{
    If \( \cf(\kappa) > \omega \), this follows from~\cite[Proposition 2.32Ac]{Levy}.  Assume now \( \cf(\kappa) = \omega \), and let \( \langle \mu_i \mid i < \omega \rangle \) be a strictly increasing sequence of ordinals which is cofinal in \(\mu\). Consider the  set \( T' = \{ x \in T \mid \lh(x) = \mu_i \text{ for some } i < \omega \} \) ordered by the inclusion relation. By our assumptions, \( T' \) is a set-theoretical finitely branching tree of height \( \omega \), hence there is an infinite branch \( z' \) through it. Then \( z = \bigcup z' \) is a cofinal branch through \( T \).
  }
\end{proof}

By Claim~\ref{claim:narrowtrees}, let \( z \) be a cofinal branch through \( V \). Then \( (x \restriction \gamma, z
\restriction
\gamma, n {}^\smallfrown{}  \xi_0 \restriction \gamma^-), (z
\restriction \gamma,
y \restriction \gamma, n {}^\smallfrown{}  \xi_1 \restriction
\gamma^-) \in S^T_n \)
for every \( \gamma < \kappa \). Therefore \( (x,z),(z,y) \in
\proj[S^T_n] \subseteq R \), hence \( (x,y) \in R \) by the transitivity of \( R \).  
\end{proof}

Recall that a map \( \varphi \colon \pre{< \kappa}{\kappa} \to \pre{<\kappa}{\kappa} \) is called \emph{Lipschitz} if it is monotone and length-preserving. Clearly, every Lipschitz map is completely determined by its values on \( \pre{\SUCC(< \kappa)}{\kappa} \).

\begin{definition}\label{definmax}
Let \( \kappa \) be an infinite cardinal. Given two DST-trees \( \T,\T' \subseteq \pre{< \kappa}{(2
\times \kappa)} \), we let \( \T \leq_{\max} \T' \) if and only if there
is a Lipschitz 
\emph{injective} function \( \varphi \colon {}^{< \kappa} \kappa \to {}^{<
\kappa} \kappa \) such that for all \( (u,s) \in \pre{< \kappa}{(2 \times \kappa)} \)
\[ (u,s) \in \T \Rightarrow (u, \varphi(s)) \in \T'. \]
\end{definition}

Assume now that \( \kappa^{< \kappa} = \kappa \). By identifying each DST-tree \( T \subseteq \pre{< \kappa}{(2 \times \kappa)} \) with its characteristic function, the quasi-order \( \leq_{\max} \) may be construed as a quasi-order on the space \( \pre{ \pre{< \kappa}{(2 \times \kappa)} }{2} \), which is in turn naturally identified to \( \pre{\kappa}{2} \) via the homeomorphism induced by any bijection between \(  \pre{< \kappa}{(2 \times \kappa)}  \) and \( \kappa \); it is easy to check that once coded in this way, the quasi-order \( \leq_{\max} \) is analytic. In fact, it can be shown that it is also complete arguing as follows. 
Given a DST-tree \( T \subseteq \pre{< \kappa}{(2 \times 2 \times \kappa)} \) of height \( \leq
\kappa \), let \( S_T \) be the DST-tree defined in~\eqref{eqS_T}. Then define the map \( s_T \) from \( \pre{\kappa}{2} \)
to the space of the DST-subtrees of \( \pre{< \kappa} {(2 \times \kappa)} \) by setting 
\begin{equation} \label{eqs_T}
s_T(x) =
S_T^x = \left\{ (u,s)  \mid  (u,x \restriction \leng{u}, s) \in S_T \right\}.
\end{equation}
Notice that by  Lemma~\ref{lemmanormalform}~\ref{lemmanormalformc2} the map \( s_T \) is injective in a strong sense, that is for every \( x,y \in \pre{\kappa}{2} \)
\begin{equation} \label{eqs_Tinjective}
x \neq y \Rightarrow s_T(x) \cap \pre{\SUCC(< \kappa)}{(2 \times \kappa)} \neq s_T(y) \cap \pre{\SUCC(< \kappa)}{(2 \times \kappa)}.
\end{equation}

Indeed, if \(\omega < \alpha+ 1 < \kappa \) is such that \( x \restriction \alpha+1 \neq y \restriction \alpha + 1 \) then \( (x \restriction \alpha+1 , 0^{(\alpha+1)})  \in s_T(x) \setminus s_T(y) \).

The proof of the following lemma is identical to that of~\cite[Lemma 7.4]{mottoros2011} (together with~\cite[Remark 7.5]{mottoros2011}) and thus will be omitted here --- the unique difference is that, because of Lemma~\ref{lemmanormalform}\ref{lemmanormalformc2}, in the first part of such proof one should systematically take \( \xi = 0^{(\kappa)} \) rather than an arbitrary \( \xi \in \pre{\kappa}{\kappa} \) with \( \xi(0)  \in \omega \).

\begin{lemma}\label{lemmamax}
Let \( \kappa \) be an uncountable cardinal satisfying~\eqref{eq:kappa}. Let 
\( R \) be an analytic quasi-order on \( \pre{\kappa}{2} \) and let \( T \) be the tree given by Lemma~\ref{lemmanormalform}. Then  for every \( x,y \in \pre{\kappa}{2} \)
\begin{enumerate-(i)}
\item
if \( S_T^x \leq_{\max} S_T^y \), then \( x \mathrel{R} y \);
\item
  conversely, if \( x \mathrel{R} y \) and this is witnessed by $(x,y,\xi) \in S_T$ then there is the Lipschitz map \( \varphi \) with $\varphi(s) = s \oplus \xi \restriction \lg(s)$  witnessing \( S_T^x \leq_{\max} S_T^y \).
Moreover,  \( \# s \leq \# \varphi(s) \), where \( \# \) is as in Proposition~\ref{propoplus}\ref{propoplusc2}).
\end{enumerate-(i)}
In particular, \( s_T \) reduces \( R \) to
\( \leq_{\max} \), and thus \( \leq_{\max} \) is complete for analytic
quasi-orders.
\end{lemma}

\section{Labels} \label{sec:labels}

Recall that we fixed an uncountable cardinal \( \kappa \)  satisfying~\eqref{eq:kappa}. Further assuming that \( \kappa \) be inaccessible, in~\cite[Section 8]{mottoros2011} three sets of labels \( \{ \sL_\gamma = (D_\gamma, \preceq_\gamma) \mid \gamma < \kappa \} \), \( \{ \sL_s =  (D_s, \preceq_s) \mid s \in \pre{\SUCC(<\kappa)}{\kappa} \} \), and \( \{ \sL^*_u = (D^*_u, \preceq^*_u) \mid u \in \pre{\SUCC(<\kappa)}{2} \} \) (called respectively \emph{labels of type I}, \emph{II}, and \emph{III}) were constructed so that the following conditions were satisfied.
\begin{enumerate-(C1)}
\item \label{cond:C1}
Each of the labels is a generalized tree of size \( \leq  \kappa  \).
%\item \label{cond:C2}
%Labels \( \sL_\gamma \) of type I and labels \( \sL^*_u \) of type III have size \( \kappa \), while labels \( \sL_s \) of type II have size \( < \kappa \).
\item \label{cond:C3}
If \( \sL \), \( \sL' \) are labels of a different type, then \( \sL \not \sqsubseteq \sL' \). In particular, two label \( \sL' \) and \( \sL'' \) of different type cannot be simultaneously embedded into the same label \( \sL \).
\item \label{cond:C4}
  If \( \gamma, \gamma' < \kappa \) are distinct, then \( \sL_\gamma \not \sqsubseteq \sL_{\gamma'} \).
  
\item \label{cond:C5}
  If \( \gamma < \kappa \) and \( s,t \in \pre{\gamma+1}{\kappa} \), then \( \sL_s \sqsubseteq \sL_t \iff \# s \leq \# t \), and moreover \( \sL_s \cong \sL_t \iff s = t \) for every \( s,t \in \pre{\SUCC(<\kappa)}{\kappa} \).

\item \label{cond:C6}
If \( \gamma < \kappa \) and \( u,v \in \pre{\gamma+1}{2} \) are distinct, then \( \sL^*_u \not \sqsubseteq \sL^*_v  \), and moreover \( \sL^*_u \cong \sL^*_v \iff u = v \) for every \( u,v \in \pre{\SUCC(<\kappa)}{2} \).
\end{enumerate-(C1)}

Our next goal is to provide a construction of such labels
for an arbitrary \( \kappa \) satisfying~\eqref{eq:kappa} so that conditions~\ref{cond:C1}--\ref{cond:C6}  are still satisfied.
The definition of the labels of type III required the inaccessibility of \( \kappa \) and now we replace it by a different construction.
As the reader may easily check, labels of type I and of type II are instead minor simplifications of the structures
given in \cite{mottoros2011} which do not destroy their main properties.

%Let  \todo{Why do we need this function?}
%\[
%\theta \colon 2 \times\kappa \times\kappa \to \kappa \colon 
%(i,\gamma,\alpha) \mapsto 1 + 2 \cdot \Hes_2(\gamma,\alpha)  + i.
%\]
%The map $\theta$ is a bijection and for every $i\in 2$,
%$\alpha, \beta, \gamma < \kappa$,
%\[
%\alpha\leq\beta \iff 
%\theta(i,\gamma,\alpha) \leq \theta(i,\gamma,\beta).
%\]
We will use the following result of Baumgartner.

\begin{lemma}[{\cite[Corollary~5.4]{Baumgartner}}] \label{2.3}
Let \( \kappa \) be an uncountable regular cardinal. Then
 there are $2^\kappa$-many linear orders
$\seq{ L_{\gamma} =  (\kappa,\preceq^{L_{\gamma}})}{\gamma<2^\kappa}$
such that  $L_{\gamma}
\not\sqsubseteq  L_{\gamma'}$ for distinct $\gamma, \gamma'<2^\kappa$.
\end{lemma}

%For technical reasons, we replace each \( L_{ \gamma} \) with \( L_{ \gamma } \times \ZZ \) (ordered lexicographically), and then add a minimal element to it, that can be always assumed to be the ordinal \( 0 \): 
%the resulting linear orders contain infinite descending chains, have a minimum but no maximal element, and all their points except the minimum have an immediate successor. Notice that these modifications do not 
%destroy the property that such linear orders are mutually non-embeddable because if \( L \times \ZZ  \sqsubseteq L' \times \ZZ \) then \( L \sqsubseteq L' \).

For technical reasons, we replace each \( L_{ \gamma} \) with (an isomorphic copy with domain \( \kappa \) of) \( 1+ \ZZ + L_\gamma + \ZZ \): the resulting linear orders have a minimum but no maximal element, and the minimum, that can be assumed to be the ordinal \( 0 \), has no immediate successor. Notice that this modification does not destroy the property that such linear orders are mutually non-embeddable.

\medskip

\textbf{Labels of type I}. 
Take the first \( \kappa \)-many linear orders $\seq{L_{\kappa,\gamma} = (\kappa, \preceq^{L_{\kappa,\gamma}})}{ \gamma <\kappa}$ from the modifications after Lemma~\ref{2.3}, 
so that $L_{\gamma}
\not\sqsubseteq L_{\gamma'}$ for distinct $\gamma, \gamma' < \kappa$.
We let $\sL_{\gamma} = (D_{\gamma}, \preceq_{\gamma})$ be defined as follows:
\begin{itemize}
\item 
$D_{\gamma} = \kappa \cup \{(\alpha,\beta) \mid 0 < \alpha<\kappa 
\wedge \beta< \alpha\}$,

\item 
$\preceq_{\gamma}$ is the partial order on $D_{\gamma}$ defined by

\begin{enumerate-(1)}
\item 
$\forall \alpha,\alpha' <\kappa \, [\alpha \preceq_{\gamma} \alpha' \iff \alpha \preceq^{L_{\kappa,\gamma}} \alpha']$

\item 
$\forall \alpha' <\kappa \, \forall 0 < \alpha < \kappa,  \forall \beta<\alpha \, [ \alpha'  \preceq_{\gamma} (\alpha,\beta) \iff \alpha' \preceq_\gamma \alpha]$

\item 
$\forall 0 <  \alpha<\kappa \, \forall \beta,\beta' <\alpha \, [(\alpha,\beta) \preceq_{\gamma} (\alpha,\beta') \iff  \beta\leq\beta']$

\item 
no other $\preceq_{\gamma}$-relation holds.
\end{enumerate-(1)}
\end{itemize}

The restriction of \( \sL_\gamma \) to \( \kappa \) (i.e.\ the linear order \( L_{\kappa,\gamma} \)) is called \emph{spine} of \( \sL_\gamma \).
Notice that each \( \sL_\gamma \) has size \( \kappa \), has a minimum, that is \( 0 \), and such a minimum has no immediate successor. 
%All other points of \( \sL_\gamma \) have at least one immediate successor (unless they are terminal nodes), and 
Moreover, a point \( x \) is in the spine if and only if \( \Cone(x) \) is not a linear order, and if \( x,y \) are incomparable in \( \sL_\gamma \) then at least one of \( \Cone(x) \) and \( \Cone(y) \) is a linear order.
%, if and only if either \( x \) is \( 0 \) or it has exactly two immediate successors (if \(x = \alpha \) these successors are \( \alpha+1 \) and \( (\alpha, 0 ) \)). 
Finally, we say that a tree is a \emph{code for \(\gamma\)}  if it is isomorphic to \( \mathscr{L}_\gamma \). 

\medskip

\textbf{Labels of type II}.
Let \( \gamma < \kappa \). Given \( s \in \pre{\gamma+1}{\kappa} \), set \( \theta(s) = \Hes(\dom(s), \# s) \) with \( \# \) is as in Proposition~\ref{propoplus}\ref{propoplusc2}, and let \( \sL_s = (D_s, \preceq_s) \) be the tree defined as follows:\footnote{Here we do not identify elements of ${}^1 \kappa$ with $\kappa$.}
\begin{itemize}
\item
\( D_s \) is the disjoint union of the ordinal \(\theta(s)\), \( \omega^* = \{ n^* \mid n \in \omega \} \), and \( A_s = \{ a,a^+,a^-,b,b^+,b^- \} \), ;

\item
\( \preceq_s \) is the partial order on \( D_s \) defined by

\begin{enumerate-(1)}
\item
\( \forall \alpha,\beta <  \theta(s)  \, \left[{\alpha \preceq_s \beta} \iff {\alpha \leq \beta}\right] \)

\item
\( \forall n,m \in \omega \, \left[{n^* \preceq_s m^*} \iff {n \geq m}\right] \)

\item
\( x \preceq_s x^+, x^- \) for \( x \in \{ a,b \} \)

\item
\( \forall \alpha < \theta(s) \, \forall n \in \omega \, \forall x \in A_s \, \left[{\alpha \preceq_s n^*} \wedge {n^* \preceq_s x}\right] \)

\item
no other \( \preceq_s \)-relation holds.
\end{enumerate-(1)}
\end{itemize}

Notice that each \( \sL_s \) has size strictly smaller than \( \kappa \), and that there are two incomparable points, namely \( a \) and \( b \), whose upper cone is not a linear order. 
A tree isomorphic to \( \sL_s \) is called a \emph{code for \( s \)}. 
%The restriction of \( \sL_s \) to \( \theta(1, \gamma, \#s) \) is called the \emph{initial part} of \( \sL_s \). In particular, the initial part of \( \sL_s \) is isomorphic to the ordinal \( \theta(1, \gamma, \# s) \) (hence it is well-founded). 

\medskip 

\textbf{Labels of type III}. 
Fix another sequence \( \langle L_u  = (\kappa, \preceq^{L_u}) \mid u \in \pre{\SUCC(< \kappa)}{2} \rangle \) of pairwise non-embeddable linear orders of size \( \kappa \) such that \( L_u \not \sqsubseteq L_{\kappa,\gamma} \) and \( L_{\kappa,\gamma} \not \sqsubseteq L_u \) for every \( \gamma < \kappa \), where the \( L_{\kappa,\gamma} \)'s are the linear order used to construct the labels of type I (for example, we can choose the \( L_u \)'s in the set \( \{ L_{\kappa,\gamma} \mid \kappa \leq \gamma < \kappa+\kappa \} \)). Then for every \( \gamma < \kappa \) and \( u \in \pre{\gamma+1}{2} \), we let $\sL^*_{u} = (D^*_u, \preceq^*_{u})$ be defined as follows:
\begin{itemize}
\item 
$D^*_{u} = \kappa \cup \{ c_u \} \cup \{(\alpha,\beta) \mid 0 < \alpha<\kappa 
\wedge \beta<\alpha\}$,

\item 
$\preceq^*_{u}$ is the partial order on $D^*_{u}$ defined by

\begin{enumerate-(1)}
\item 
$\forall \alpha,\alpha' <\kappa \, [\alpha \preceq^*_{u} \alpha' \iff \alpha \preceq^{L_u} \alpha']$

\item 
$0 \preceq^*_u c_u$ and \( c_u \) is \( \preceq^*_u \)-incomparable with any other point of \( D^*_u \)

\item 
$\forall \alpha' <\kappa \, \forall 0 < \alpha < \kappa,  \forall \beta<\alpha \, [ \alpha'  \preceq^*_{u} (\alpha,\beta) \iff \alpha' \preceq^*_u \alpha]$

\item 
$\forall 0 < \alpha<\kappa \, \forall \beta,\beta' <\alpha \, [(\alpha,\beta) \preceq^*_u (\alpha,\beta') \iff 
\beta\leq\beta']$

\item 
no other $\preceq^*_u$-relation holds.
\end{enumerate-(1)}
\end{itemize}

Thus labels of type III are constructed exactly as the labels of type I, except that we add a unique immediate successor \( c_u \) (which is also a terminal node in \( \sL^*_u \)) to its minimum \( 0 \).
As in the case of type I labels, we call the restriction of \( \sL^*_u \)  to \( \kappa \) (i.e.\ the linear order \( L_u \))  \emph{the spine of \( \sL^*_u \)}. Points in the spine are distinguished from the other ones by the fact that their upper cone is not a linear order. 
%Each \( \sL^*_u \) has size \( \kappa \), has a minimum, that is \( 0 \), but, as opposed to the case of labels of type I, such a minimum has exactly one immediate successor, namely \( c_u \) (which is a terminal node in \( \sL^*_u \)). All other points of \( \sL_\gamma \) have at least one immediate successor, and a point \( x \) is in the spine if and only if \( \Cone(x) \) is not a linear order, if and only if either \( x \) is \( 0 \) or it has exactly two immediate successors (if \(x = \alpha \) these successors are again \( \alpha+1 \) and \( (\alpha, 0 ) \)). So a label of type III is distinguished from a label of type I only by the fact that in the former case the minimum has an immediate successor. 
Similarly to the previous cases, we say that a tree is a \emph{code for \(u\)} if it is isomorphic to \( \mathscr{L}^*_u \). Notice that all the \( \sL^*_u \)'s have size exactly \( \kappa \).

\medskip

We now argue that also with our new definitions conditions~\ref{cond:C1}--\ref{cond:C6} are satisfied. This is obvious for condition~\ref{cond:C1}. Conditions~\ref{cond:C4}--\ref{cond:C5} can be proved as in~\cite[Lemmas 8.3 and 8.4]{mottoros2011} (the reader can easily check that our minor modifications have no influence on the arguments used there). Finally, the following proposition ensures that also the remaining conditions~\ref{cond:C3} and~\ref{cond:C6}  are still satisfied.

\begin{proposition} \label{prop:succofsingular}
\begin{enumerate-(i)}
\item  \label{prop:succofsingular-i}
If \( \sL \), \( \sL' \) are labels of different type, then \( \sL \not \sqsubseteq \sL' \).
\item \label{prop:succofsingular-ii}
If \( u,v \in \pre{\SUCC(< \kappa)}{2} \) are distinct, then \( \sL^*_u \not \sqsubseteq \sL^*_v \).
\end{enumerate-(i)}
\end{proposition}

\begin{proof} \ref{prop:succofsingular-i}
If \( \sL' \) is of type II (and \( \sL \) is of a different type), then \( \sL \not \sqsubseteq \sL' \) because \( |\sL| = \kappa > | \sL' | \). Vice versa, if \( \sL \) is of type II and \( \sL' \) is either of type I or of type III, then \( \sL \not \sqsubseteq \sL' \) because in \( \sL \) there are incomparable points whose upper cone is not a linear order  (e.g.\ the points \( a \) and \( b \)), while this property fails for \( \sL' \).

If \( \sL = \sL_\gamma \) is of type I and \( \sL' = \sL^*_u \) is of type III, then any embedding of \( \sL \) into \( \sL' \) would map the spine of \( \sL \) into the spine of \( \sL' \) because in both cases the points in the spine are characterized by the fact that their upper cone is not a linear order. It would then follow that \( L_ \gamma \sqsubseteq L_u \), a contradiction. The case where \( \sL \) is of type III and \( \sL' \) is of type I is similar.

For ~\ref{prop:succofsingular-ii}, arguing as in the previous paragraph we get that if \( \sL^*_u \sqsubseteq \sL^*_v \), then \( L_u \sqsubseteq L_v \) because these are the spines of \( \sL^*_u \) and \( \sL^*_v \), respectively, whence \( u = v \).
\end{proof}

\section{Completeness}\label{sec:completeness}

Following~\cite[Section 9]{mottoros2011}, we now show that the embeddability relation on generalized trees of size \( \kappa \) is complete as soon as \( \kappa \) satisfies~\eqref{eq:kappa}, thus dropping the previous large 
cardinal requirements from~\cite{mottoros2011}.
The construction we use here is exactly the one employed there (except that our labels are now defined differently): indeed the reader may check that all proofs 
in~\cite[Section 9]{mottoros2011} needs only that the labels \( \sL_\gamma \), \( \sL_s \), and \( \sL^*_u \) satisfy conditions~\ref{cond:C1}--\ref{cond:C6} --- the appeal to inaccessibility or weak compactness of \( \kappa \) 
was necessary only because the construction of the old labels \( \sL^*_u \) required the first condition, while the proof of the analogue of our Lemma~\ref{lemmanormalform} (namely,~\cite[Lemma 7.2]{mottoros2011}) 
required the latter. For the reader's convenience, and because it will turn out to be useful to precisely know how the involved trees are constructed, we report here the definition of the trees \( G_0 \) and \( G_{\mathcal{T}} \) and state the relevant results related to them. Proofs will be systematically omitted, but the interested reader may consult the analogous results from~\cite[Section 9]{mottoros2011} which are mentioned before each of the statements.

\emph{We assume that \( \kappa \) is an uncountable cardinal  satisfying~\eqref{eq:kappa} as before}. Considering suitable isomorphic copies, we can assume without loss of generality that for every 
\( \gamma < \kappa \), \( s,t \in \pre{\SUCC(< \kappa)}{\kappa} \), and 
\( u,v \in \pre{\SUCC(< \kappa)}{2} \) our labels further satisfy the following conditions:
\begin{enumerate-(i)}
\item
\( \sL_\gamma\), \( \sL_s \), and \( \sL^*_u \) have pairwise disjoint 
domains;
\item
\( \sL^*_u \) and \( \sL^*_v \) have disjoint domains if and only if 
\( u \neq v \);
\item
 if \( \leng{s} = \leng{t} \), then the domain of \( \sL_s \) is contained in the 
domain of \(\sL_t \) if and only if \( \# s \leq \# t \).
\end{enumerate-(i)}

\noindent
These technical assumptions will ensure that the trees \( G_{\mathcal{T}} \) are well-defined avoiding unnecessary complications in the notation.

Let us now first define the generalized tree \( G_0 \) (which is independent of the choice of the DST-tree
\( \mathcal{T} \)). Roughly speaking, \( G_0 \) will be constructed by appending 
to the nodes of the tree \( \left(\pre{\SUCC(< \kappa)}{\kappa}, \subseteq \right) \)
 some labels as follows. Let \( \bar{\gamma} \colon \pre{\SUCC(<\kappa)}{\kappa} \to \kappa \colon s \mapsto \leng{s}-1 \). For every \( s \in 
\pre{\SUCC(< \kappa)}{\kappa} \) we fix a distinct copy of \( (\ZZ , \leq ) \)  and append it to \( s \): each of these copies of \( \ZZ \) will be called a
\emph{stem}, and if such a copy is appended to \( s \) it will be called the 
\emph{stem of \( s \)}. Then for every such \( s \) we fix also distinct 
copies \( \sL_{\bar{\gamma}(s),s} \) and \( \sL_{s,s} \) of, respectively, \( \sL_{\bar{\gamma}(s)} \) and \( \sL_s \), and then append both of them to the stem of \( s \).   
More formally, we have the following definition.

\begin{definition}
The tree \( G_0  \) is defined by the following conditions:
\begin{itemize}
\item
  \( G_0 = \pre{\SUCC(< \kappa)}{\kappa}\cup \bigcup_{ s \in \pre{\SUCC(< \kappa)}{\kappa}} ( \{ (s,x) \mid x \in \ZZ \cup D_{\bar{\gamma}(s)}  \cup  D_s \}   ) \), where \( D_{\bar{\gamma}(s)} \)'s and \( D_s \) are the domains of, respectively,
  the labels \( \sL_{\bar{\gamma}(s)} \) of type I and the label \( \sL_s \) of type II;
\item
the partial order \( \preceq^{G_0} \) on \( G_0 \) is defined as follows:
\begin{enumerate-(1)}
\item
\( \forall s,t \in \pre{\SUCC(<\kappa)}{\kappa}  \left({s \preceq^{G_0} t} \iff {s \subseteq t}\right) \)
\item
\( \forall s \in \pre{\SUCC(<\kappa)}{\kappa} \, \forall z,z' \in \ZZ  \left({(s,z) \preceq^{G_0} (s,z')} \iff {z \leq z'}\right) \)
\item
\( \forall s \in \pre{\SUCC(< \kappa)}{\kappa}\, \forall x,x' \in D_{\bar{\gamma}(s)}  \left({(s,x) \preceq^{G_0} (s,x')} \iff {x \preceq_{\bar{\gamma}(s)} x'}\right) \)
\item
\( \forall s \in \pre{\SUCC(<\kappa)}{\kappa}\, \forall x,x' \in D_s  \left({(s,x) \preceq^{G_0} (s,x')} \iff {x \preceq_s x'}\right) \)
\item
\( \forall s,t\in \pre{\SUCC(<\kappa)}{\kappa}\, \forall x \in \ZZ \cup D_{\bar{\gamma}(t)} \cup D_t  \left({s \preceq^{G_0} (t,x)} \iff {s \subseteq t}\right) \)
\item
\( \forall s \in \pre{\SUCC(< \kappa)}{\kappa}\,  \forall z \in \ZZ \, \forall x \in D_{\bar{\gamma}(s)} \cup D_s  \left({(s,z) \preceq^{G_0} (s,x)} \right) \)
\item
no other \( \preceq^{G_0} \)-relation holds.
\end{enumerate-(1)}
\end{itemize}
\end{definition}

\noindent
So the \emph{stem of \( s \)} is \( G_0 \restriction \{ s \} \times \ZZ \). Substructures of the form \( G_0 \restriction \{ s \} \times D_{\bar{\gamma}(s)} \) and \( 
G_0 \restriction \{ s \} \times D_s \) will be called \emph{labels} (of type I 
and II, respectively). 

Let now \( \mathcal{T} \) be a DST-tree on \( 2 \times \kappa \) of height \( \kappa \). The tree \( G_{\mathcal{T}} \) will be constructed by appending a distinct copy of the label \( \sL^*_u \) to the stem of \( s \) for every \( (u,s) \in \mathcal{T} \) with \( s \) of successor length.

\begin{definition}
The tree \( G_{\mathcal{T}} = (D_{\mathcal{T}}, \preceq_{\mathcal{T}}) \) is defined as follows:
\begin{itemize}
\item
\( D_{\mathcal{T}} = G_0 \cup \bigcup_{\substack{(u,s) \in \mathcal{T} \\ s \in \pre{\SUCC(< \kappa)}{\kappa}}} \{ (s,x) \mid x \in D^*_u \} \), where \( D^*_u \) is the domain on \( \sL^*_u \);
\item
\( \preceq_{\mathcal{T}} \) is the partial order on \( D_{\mathcal{T}} \) defined by:
\begin{enumerate-(1)}
\item
\( \forall x,y \in G_0  \left({x \preceq_{\mathcal{T}} y } \iff {x \preceq^{G_0} y}\right) \)
\item
\(  \forall (u,s) \in \mathcal{T}  \left[s \in \pre{\SUCC(<\kappa)}{\kappa} \Rightarrow \forall x,y \in D^*_u \left ({(s,x) \preceq_{\mathcal{T}} (s,y)} \iff {x \preceq^*_u y}\right) \right]\)
\item
\( \forall t \in \pre{\SUCC(<\kappa)}{\kappa} \, \forall (u,s) \in \mathcal{T}  \left [s \in \pre{\SUCC(<\kappa)}{\kappa} \Rightarrow \forall x \in D^*_u  \left ({t \preceq_{\mathcal{T}} (s,x) } \iff {t \subseteq s} \right)\right]  \)
\item
\( \forall (u,s) \in \mathcal{T}  \left[s \in \pre{\SUCC(<\kappa)}{\kappa} \Rightarrow \forall x \in D^*_u \, \forall z \in \ZZ \left ((s,z) \preceq_{\mathcal{T}} (s,x)\right)\right] \)
\item
no other \( \preceq_{\mathcal{T}} \)-relation holds.
\end{enumerate-(1)}
\end{itemize}
\end{definition}

Substructures of the form \( G_{\mathcal{T}} \restriction \{ s \} \times \ZZ \),  \( G_{\mathcal{T}} \restriction \{ s \} \times D_{\bar{\gamma}(s)} \), and \( G_{\mathcal{T}} \restriction \{ s \} \times D_s \) will again be called, respectively, \emph{stem of \( s \)},  \emph{labels of type I} and \emph{labels of type II}, and be denoted by, respectively, \( \sS^{\mathcal{T}}_{s} \), \( \sL^{\mathcal{T}}_{\bar{\gamma}(s),s} \), and \( \sL^{\mathcal{T}}_{s,s} \). Similarly, substructures of the form \( G_{\mathcal{T}} \restriction \{ s \} \times D^*_u \) (for \( (u,s) \in \mathcal{T} \)) will be called \emph{labels of type III}, and be denoted by \( \sL^{\mathcal{T}}_{u,s} \).
Notice that if \( \sL \) is a label of \( G_{\mathcal{T}} \) with domain \( D_{\sL} \) and \( x  \in D_{\sL} \), then \( \cone(x) \subseteq D_{\sL} \).

For \( s \in \pre{\SUCC(< \kappa)}{\kappa} \), we let
\[ 
\cone(\sS^{\mathcal{T}}_s) = \bigcap_{z \in \ZZ} \cone((s,z)).
 \] 
Therefore,  \( \cone(\sS^{\mathcal{T}}_s) \) consists of a disjoint union of 
labels of various type. In particular, it contains exactly one label of type I 
(namely, \( \sL^{\mathcal{T}}_{\bar{\gamma}(s),s} \)), one label of type II (that is, 
\( \sL^{\mathcal{T}}_{s,s} \)) and, depending on \( \mathcal{T} \), a variable 
number of labels of type III (namely, a label of the form 
\(\sL^{\mathcal{T}}_{u,s} \) for every \( (u,s) \in \mathcal{T} \)). Notice 
also that every label \( \sL \subseteq \cone(\sS^\mathcal{T}_s) \) is a 
maximal connected component of \( G_\mathcal{T} \restriction 
\cone(\sS^\mathcal{T}_s) \). The next theorem is the analogue of~\cite[Theorem 9.3]{mottoros2011} and can be proved in the same way.

\begin{theorem} \label{theorcomplete}
Let \( \kappa \) be any cardinal satisfying~\eqref{eq:kappa}, let \( \mathcal{T}, \mathcal{T}' \) be two DST-trees on \( 2 \times \kappa \) of height \( \kappa \), and let \( \# \) be as in 
Proposition~\ref{propoplus}\ref{propoplusc2}.
\begin{enumerate-(1)}
\item\label{theorcomplete1}
\( {G_{\mathcal{T}} \sqsubseteq G_{\mathcal{T}'}} \iff \) there is a witness \( \varphi \colon \pre{< \kappa}{\kappa} \to \pre{< \kappa}{\kappa} \) of \( {\mathcal{T} \leq_{\max} \mathcal{T}'} \) such that \( \forall  s\in \pre{\SUCC(< \kappa)}{\kappa} \left (\#s \leq \# \varphi(s) \right)  \);
\item \label{theorcomplete2}	
\( {G_{\mathcal{T}} \cong G_{\mathcal{T}'}} \iff {\mathcal{T} \cap \pre{\SUCC(<\kappa)}{(2 \times \kappa)} = \mathcal{T}' \cap \pre{\SUCC(< \kappa)}{(2 \times \kappa)}}  \).
\end{enumerate-(1)}
\end{theorem}

Let \( \kappa \) be any uncountable cardinal satisfying~\eqref{eq:kappa}, \( R  \) be an analytic quasi-order on \( \pre{\kappa}{2} \), and \( T \) a DST-tree on \( 2 \times 2 \times \kappa \) as 
in Lemma~\ref{lemmanormalform}. Recall that in~\eqref{eqs_T} we defined a map \( s_T \) sending \( x \in \pre{\kappa}{2} \) into a DST-tree on \( 2 \times \kappa \) of height \( \kappa \) denoted by \( s_T(x) \). 
Since each tree \( G_{s_T(x)} \) can be easily Borel-in-\( T \) coded into a tree with domain \( \kappa \), henceforth \( G_{s_T(x)} \) will be tacitly identified with such a copy. With this notational convention, the composition of \( s_T \) with the map sending \( \mathcal{T} \) into \( G_{\mathcal{T}} \) gives the function
\begin{equation}\label{eqf} 
f \colon \pre{\kappa}{2} \to \Mod^\kappa_\L \colon x \mapsto G_{s_T(x)}, 
\end{equation}
which will be our reduction of \( R \) to the embeddability relation \( \sqsubseteq \restriction \Mod^\kappa_\L \), for \(\L\) the language of trees.

Let now \( \sqsubseteq^\kappa_{\mathsf{TREE}} \) (\( \sqsubseteq^\kappa_{\mathsf{GRAPH}} \)) denote the relation of embeddability between trees (respectively, graphs) of size \( \kappa \). 
Combining Lemma~\ref{lemmamax} and Theorem~\ref{theorcomplete} we now get (see also~\cite[Corollary 9.5]{mottoros2011}):

\begin{theorem} \label{thm:complete}
Let \( \kappa \), \( R \), and \( T \) be as above. Then the map \( f \) from~\eqref{eqf} is a Borel reduction of \( R \) to \( \sqsubseteq^\kappa_{\mathsf{TREE}} \).
In particular, the relation \( \sqsubseteq^\kappa_{\mathsf{TREE}} \) is complete for analytic quasi-orders.
\end{theorem}

Finally, by~\cite[Remark 9.7]{mottoros2011} we also obtain an analogous result for graphs (see~\cite[Corollary 9.8]{mottoros2011}).

\begin{corollary}\label{cor:completegraphs}
Let \( \kappa \) be any cardinal satisfying~\eqref{eq:kappa}. Then \( \sqsubseteq^\kappa_{\mathsf{GRAPH}} \) is complete for analytic quasi-orders.
\end{corollary}

\section{Strongly invariant universality}\label{sec:universality}

Let \( \L = \{ \preceq \} \) be the tree language consisting of one binary relational symbol, and let \( \kappa \) be an uncountable cardinal satisfying~\eqref{eq:kappa}. %, and let \( \kappa \) be an inaccessible cardinal. 
For the rest of this section, \( X , Y \) will denote arbitrary \( \L \)-structures of size \( \leq \kappa \). As a first step, following~\cite[Section 10]{mottoros2011} we provide an \( \L_{\kappa^+ \kappa} \)-sentence \( \Uppsi \) such that
\( G_{\mathcal{T}} \models \Uppsi \) for every DST-tree \( \mathcal{T} \) on \( 2 \times \kappa \) of height \( \kappa \), and, conversely, every \( X \in \Mod^\kappa_\Uppsi \) is ``very close'' to being a tree of the form \( G_{\mathcal{T}} \).

To simplify the notation, we let \( x \prec y \), \( x \not\preceq y \), \( x \perp y \), and \( x \not\perp y \)  be abbreviations for, respectively, \( {x \preceq y} \wedge {x \neq y} \), \( \neg (x \preceq y ) \), \( x \not\preceq y \wedge y \not\preceq x \), and \( x \preceq y \vee y \preceq x \).
Let \( X \) be an \( \L \)-structure of size \( \leq \kappa \), and let \( i \colon X \to \kappa \) be an injection. We denote by 
\[ 
\uptau^i_{\mathsf{qf}}(X)  (\langle \V_\alpha \mid \alpha \in \range(i) \rangle )
 \] 
the \emph{quantifier free type of \( X \) (induced by \( i \))}, i.e.\ the formula
\[ 
\bigwedge_{\substack{x,y \in X \\  x \neq y}} (\V_{i(x)} \neq \V_{i(y)}) \wedge \bigwedge_{\substack{x,y \in X \\ x \preceq^X y}} (\V_{i(x)} \preceq \V_{i(y)}) \wedge \bigwedge_{\substack{x,y \in X \\ x \not\preceq^X y}} \V_{i(x)} \not\preceq \V_{i(y)}.
 \] 
Notice that \( \uptau^i_{\mathsf{qf}}(X)(\langle \V_\alpha \mid \alpha \in 
\range(i) \rangle) \) 
is an \( \L_{\kappa^+ \kappa} \)-formula if and only if \( |X| < \kappa \). 
Moreover, if \( Y \) is an \( \L \)-structure and \( \langle a_\alpha \mid \alpha 
\in \range(i) \rangle , \langle b_\alpha \mid \alpha \in \range(i) \rangle \) 
are two sequences of elements of 
\( Y \) such that both 
\( Y \models \uptau^i_{\mathsf{qf}}(X)[\langle a_\alpha \mid \alpha \in 
\range(i) \rangle]\) and 
\( Y \models \uptau^i_{\mathsf{qf}}(X)[\langle b_\alpha \mid \alpha \in 
\range(i) \rangle] \), 
then \( Y \restriction \{ a_\alpha \mid \alpha \in \range(i) \} \) and 
\( Y \restriction \{ b_\alpha \mid \alpha \in \range(i) \} \) are isomorphic (in 
fact, they are 
isomorphic to \( X \)). In order to simplify the notation, since the choice 
of \( i \) is often irrelevant we will drop the reference to \( i \), replace variables 
with metavariables, and call the resulting expression \emph{qf-type of \( X \)}. 
Hence in general  we will denote the qf-type of an \( \L \)-structure \( X \) by
\[ 
\uptau_{\mathsf{qf}}(X) (\langle x_i \mid i \in X \rangle ).
 \]

First let \( \Upphi_0 \) be the \( \L_{\kappa^+ \kappa} \)-sentence axiomatizing 
generalized trees, i.e.\ the first order sentence

\begin{multline}\tag{\( \Upphi_0 \)}
\forall x  \left(x \preceq x \right) \wedge
\forall x \, \forall y  \left({{x \preceq y} \wedge {y \preceq x}} \Rightarrow {x = y}\right ) \wedge \\
\forall x \, \forall y \, \forall  z  \left({{x \preceq y} \wedge {y \preceq z}} \Rightarrow {x \preceq z}\right) \wedge 
\forall x \, \forall y \, \forall z  \left({{y \preceq x} \wedge {z \preceq x}} \Rightarrow {y \not\perp z}\right).
\end{multline}

Let \( \mathsf{Seq}(x) \) be the \( \L_{\kappa^+ \kappa} \)-formula
\begin{equation}\tag{\( \mathsf{Seq} \)} 
\neg \exists \langle x_n \mid n < \omega \rangle  \bigwedge_{n < m < \omega} \left ({x_n \preceq x} \wedge {x_m \prec x_n} \right),
 \end{equation}
and let \( \mathsf{Root}(x,y) \) be the \( \L_{\kappa^+ \kappa} \)-formula
\begin{equation} \tag{\( \mathsf{Root} \)}
{\mathsf{Seq}(x)} \wedge {\neg \mathsf{Seq}(y)} \wedge {x \preceq y} \wedge
\neg \exists w \left({x \prec w} \wedge {w \preceq y} \wedge {\mathsf{Seq}(w)}\right).
\end{equation}
\begin{remark}\label{remroot}
Note that if \( X \) is a tree and \( a \in X \), \( X \models \mathsf{Seq}[a] \) if and only if \( \pred(a) \) is well-founded, and that \( X_{\mathsf{Seq}} = \{ a \in X \mid \pred(a) \text{ is well-founded} \} \) is necessarily \( \preceq^X \)-downward closed.
Moreover, if \( a,a',b  \in X \) are such that \( X \models \mathsf{Root}[a,b] \) and \( X \models \mathsf{Root}[a',b] \), then \( a = a' \). This is because \( X \models \mathsf{Root}[a,b] \wedge \mathsf{Root}[a',b] \) implies \( a,a' \preceq^X b \), hence, since \( X \) is a tree, \( a \) and \( a' \) are comparable. Assume without loss of generality that \( a \preceq^X a' \): since \( \pred(a') \) is 
well-founded, \( a \neq a' \) would contradict \( X \models \mathsf{Root}[a,b] \). Therefore \( a  = a' \). 
\end{remark}

Let \( \Upphi_1 \) be the \( \L_{\kappa^+ \kappa} \)-sentence
\begin{equation} \tag{\( \Upphi_1 \)}
\forall y  \left[ \mathsf{Seq}(y) \vee \exists x \, \mathsf{Root}(x,y) \right] .
\end{equation}

\begin{remark}\label{remPhi_1}
Let \( X \) be a tree. Given \( a \in X_{\mathsf{Seq}} \), let \( X_a \) be the substructure of \( X \) with domain 
\begin{align*}
X_a  &= \left\{ b \in X \mid a \preceq^X b \wedge \neg \exists c  \left(a \prec^X c \preceq^X b \wedge c \in X_{\mathsf{Seq}}\right) \right\} \\
& = \left\{ b \in X \mid X \models \mathsf{Root}[a,b] \right\}.
\end{align*}
Assume now that \( X \models \Upphi_1 \). Then for every \( b \in X \) either \( b \in X_{\mathsf{Seq}} \) or \( b \) belongs to \( X_a \) for some \( a \in X_{\mathsf{Seq}} \). Moreover, each \( X_a \) is obviously \( \preceq^X \)-upward closed (i.e.\ for every \(a,b,c \in X \), if \( X \models \mathsf{Root}[a,b] \) and \( b \preceq^X c \) then \( X \models \mathsf{Root}[a,c] \)). This implies that:
\begin{itemize}
\item
if \( a,a' \in X_{\mathsf{Seq}} \) are distinct, \( b \in X_a \), and \( b' \in X_{a'} \), then \( b,b' \) are incomparable;
\item
for every \( a,a' \in X_{\mathsf{Seq}} \) and \( b \in X_a \), 
\[ 
a' \preceq^X b \iff a' \preceq^X a;
\]
\item
by Remark~\ref{remroot}, for \( a,a',b \) as above \( b \not\preceq^X a' \) (otherwise \( b \in X_{\mathsf{Seq}} \), contradicting \( b \in X_a \)).
\end{itemize}
\end{remark}

Consider now the linear order \( \ZZ  = (\ZZ, \leq ) \). Let \( \mathsf{Stem} (\langle x,x_z \mid z \in \ZZ \rangle ) \) be the \( L_{\kappa^+ \kappa} \)-formula
\begin{multline}\tag{\( \mathsf{Stem} \)}
\uptau_{\mathsf{qf}}(\ZZ)(\langle x_z \mid z \in \ZZ \rangle) \wedge  \bigwedge\nolimits_{z  \in \ZZ} \mathsf{Root}(x,x_z) \wedge \\
 \forall y \left [{\mathsf{Root}(x,y)} \Rightarrow \left (  {\bigvee\nolimits_{z \in \ZZ} y = x_z} \vee {\bigwedge\nolimits_{z \in \ZZ} x_z \prec y}\right ) \right ].
 \end{multline}

We also let \( \mathsf{Stem^\in}(x,y) \) be the \( \L_{\kappa^+ \kappa} \)-formula 
\begin{equation}\tag{\( \mathsf{Stem^\in} \)}
\exists \langle x_z \mid z \in \ZZ \rangle  \left ({\mathsf{Stem}(\langle x,x_z \mid z \in \ZZ \rangle)} \wedge {\bigvee\nolimits_{z \in \ZZ} y = x_z} \right ).
 \end{equation}

\begin{lemma} \label{lemmastemisunique}
Let \( X  \) be a tree, \( a \in X \) and \( \langle a_z \mid z \in \ZZ \rangle , \langle b_z \mid z \in \ZZ \rangle\ \in \pre{\ZZ}{X}\). If \( X \models \mathsf{Stem}(\langle a,a_z \mid z \in \ZZ \rangle) \) and \( X \models \mathsf{Stem}(\langle a,b_z \mid z \in \ZZ \rangle) \), then there is \( k \in \ZZ\) such that \( a_z = b_{z+k} \) for every \( z \in \ZZ \). In particular, \( \{ a_z \mid z \in \ZZ \} = \{ b_z \mid z \in \ZZ \} \).
\end{lemma}

\begin{proof}
Fix \( z \in \ZZ \). We claim that there is \( i \in \ZZ \) such that \( a_z = b_i \). If not, 
since \( X \models 
{\bigvee_{i \in \ZZ} a_z = b_i} \vee {\bigwedge_{i \in \ZZ} b_i \preceq a_z} 
\) (because  \( X \models \mathsf{Stem}(\langle a,a_z \mid z \in 
\ZZ \rangle) \), whence \( X \models \mathsf{Root}(a,a_z) \)) we get \( b_i \preceq^X a_z\) for every \( i \in 
\ZZ \), which in particular imply \( a_z \not\preceq^X b_i \) and \( b_i \neq a_j \) for every \( i,j \in \ZZ \) 
with \( z < j \). Since \( X \restriction \{ a_j \mid j<z \} \) has order type \( 
(\omega, \geq) \not\cong (\ZZ, \leq)\), there is \( \bar{\imath} \in \ZZ\) such that \( b_{\bar{\imath}} \neq a_j \) for every \( j 
< z \), and hence also \( b_{\bar{\imath}} \neq a_j \) for every \( j \in \ZZ \). Since \( X \models \mathsf{Stem}
(\langle a,b_z \mid z \in \ZZ \rangle) \), then \( X \models \mathsf{Root}[a,b_{\bar{\imath}}] \): this fact, together with the choice of \( \bar{\imath} \),
contradicts \( X \models {\bigvee_{z \in \ZZ} b_{\bar{\imath}} = a_z} \vee 
{\bigwedge_{z \in \ZZ} a_z \preceq b_{\bar{\imath}}} \). 

A similar argument shows
that for every \( i \in \ZZ \) there is \( z \in \ZZ \) such that \( b_i = a_z \). 
Hence there is a bijection \( f \colon \ZZ \to \ZZ \) such that \( a_z = b_{f(z)} 
\) for every \( z \in \ZZ \). Since \( X \models \uptau_{\mathsf{qf}}(\ZZ) [\langle a_z  \mid z \in \ZZ 
\rangle] \wedge \uptau_{\mathsf{qf}}(\ZZ) [\langle b_z  \mid z \in \ZZ 
\rangle] \), \( f \) must be of the form \( i \mapsto i + k \) for some \( k \in 
\ZZ \).
\end{proof}

Let \( \Upphi_2 \) be the \( \L_{\kappa^+ \kappa} \)-sentence
\begin{equation}\tag{\( \Upphi_2 \)}
\forall x \left(\mathsf{Seq}(x) \Rightarrow \exists \langle x_z \mid z \in \ZZ  \rangle \, \mathsf{Stem}(\langle x, x_z \mid z \in \ZZ \rangle \right).
 \end{equation}
\begin{remark}\label{remPhi_2}
If \( X \) is a tree such that \( X \models \bigwedge_{i \leq 3} \Upphi_i \), then 
at the bottom of each \( X_a \) (for \( a \in X_{\mathsf{Seq}} \)) there is an isomorphic copy \( \sS^X_a \) of \( \ZZ \) (which from now on will be called \emph{stem of \( a \)}) such that all other points in \( X_a \) are \( \preceq^X \)-above (all the points of)  \( \sS^X_a \). To simplify the notation, we will denote by \( \Cone(\sS^X_a) \) the set \( X_a \setminus \sS^X_a \). Notice that the stem of \( a \) is unique by Lemma~\ref{lemmastemisunique}, and for \( a \in X_{\mathsf{Seq}} , b \in X \) 
\[ 
X \models \mathsf{Stem^\in}[a,b] \iff b \in \sS^X_a.
 \] 

\end{remark}

Let \( \mathsf{Min}(x,y) \) and \( \mathsf{Min}^*(x,y,z) \) be the \( \L_{\kappa^+ \kappa} \)-formul\ae{}
\begin{equation} \tag{\( \mathsf{Min} \)}
\mathsf{Root}(x,y) \wedge \neg \mathsf{Stem}^\in (x,y) \wedge  \forall z \left ( \mathsf{Root}(x,z) \wedge \neg \mathsf{Stem}^\in (x,z) \wedge z \not\perp y \Rightarrow y \preceq z  \right)
\end{equation}
and
\begin{equation} \tag{\( \mathsf{Min}^* \)}
\mathsf{Min}(x,y) \wedge y \preceq z.
\end{equation}

Moreover, let \( \Upphi_3 \) be the \( \L_{\kappa^+ \kappa} \)-sentence
\begin{equation} \tag{\( \Upphi_3 \)}
\forall x \forall z \left( \mathsf{Root}(x,z) \wedge \neg \mathsf{Stem}^\in (x,z) \Rightarrow \exists y \, \mathsf{Min}^*(x,y,z)  \right).
\end{equation}

\begin{remark}
If \( X \) is a tree such that \( X \models \bigwedge_{i \leq 2} \Upphi_i \) and \( a,b, c \in X \), then \( X \models \mathsf{Min}[a,b] \) if and only if \( b \) is a \( \preceq^X \)-minimal element in \( \Cone( \sS^X_a ) \), and \( X \models \mathsf{Min}^*[a,b,c] \) if and only if \( c \) is \( \preceq^X \) above the minimal (in the above sense) element \( b \). Thus \( X \models \bigwedge_{i \leq 3} \Upphi_i \) if and only if every \( c \in \Cone( \sS^X_a ) \) is \( \preceq^X \)-above some of these minimal elements \( b \). Notice that since \( X \) is a tree, such a \( b \) is unique, so \( \{ \Cone(b) \mid X \models \mathsf{Min}(a,b) \} \) is a partition of \( \Cone( \sS^X_a ) \) into maximal connected components.
\end{remark}

Given \( s \in \pre{\SUCC(< \kappa)}{\kappa} \), we let \( \mathsf{Lab}_s(\langle x,x_i \mid i \in \sL_s \rangle) \) be the \( \L_{\kappa^+ \kappa} \)-formula
\begin{multline}\tag{\( \mathsf{Lab}_s \)}
\uptau_\mathsf{qf}(\sL_s)(\langle x_i \mid i \in \sL_s \rangle)
\wedge 
\mathsf{Min}(x,x_0) \wedge  
\forall y \left [\mathsf{Min}^*(x,x_0,y)  \Rightarrow \bigvee\nolimits_{i \in \sL_s} y = x_i  \right ].
\end{multline}
Since $|\sL_s|<\kappa$, we can quantify over
$\langle x_i \mid i \in \sL \rangle$ and let  \( \mathsf{Lab}_s^\in(x,y) \) be the \( \L_{\kappa^+ \kappa} \)-formula
\begin{equation}\tag{\( \mathsf{Lab}_s^\in \)}
\exists \langle x_i \mid i \in \sL_s \rangle \left (\mathsf{Lab}_s(\langle x,x_i \mid i \in \sL_s \rangle) \wedge \bigvee\nolimits_{i \in \sL_s} y = x_i \right ).
 \end{equation}

\begin{remark} \label{remLab_s}
If \( X \) is a tree, \( a \in X \), and \( \langle a_i \mid i \in \sL_s  \rangle \) is a sequence of elements of \( X \) such that \( X \models \mathsf{Lab}_s[\langle a,a_i \mid i \in \sL_s \rangle] \) (which implies \( a_i \in X_a \) for every \( i \in \sL_s \)), then the structure \( X \restriction \{ a_i \mid i \in \sL_s \} \) is a label of type II which is a code for \( s \). Moreover, if \( X \models \bigwedge_{i \leq 3} \Upphi_i \) then \( X \restriction \{ a_i \mid i \in \sL_s \} \) is above \( \sS^X_a \) and is one of the maximal connected component of \( \Cone( \sS^X_a ) \) with \( a_0 \) as its \( \preceq^X \)-minimal element. In particular, \( X \models \mathsf{Min}^*[a,a_0,a_i] \) for every \( i \in \sL_s \) and \( X \restriction \{ a_i \mid i \in \sL_s \} \) is \( \preceq^X \)-upward closed in both \( X_a \) and  \( X \). 
\end{remark}

\begin{lemma}\label{lemmadisjointorequal}
Let \( X \) be a tree, \( s,t \in \pre{\SUCC(< \kappa)}{\kappa} \), and \( \langle a,a_i \mid i \in \sL_s \rangle, \langle a,b_j \mid j \in \sL_t \rangle \) be sequences of elements of \( X \) such that both \( X \models \mathsf{Lab}_s[\langle a,a_i \mid i \in \sL_s \rangle] \) and \( X \models \mathsf{Lab}_t[\langle a,b_j \mid j \in \sL_t \rangle] \). Then either the sets \( A = \{ a_i \mid i \in \sL_s \} \) and \( B = \{ b_j \mid j \in \sL_t \} \) are disjoint or they coincide (and in this second case \( s = t \)). 
\end{lemma}

\begin{proof}
Assume \( A \cap B \neq \emptyset \): we claim that \( A \subseteq B \) (the proof of \( B \subseteq A \) can be obtained in a similar way).
Let \( i_0 \in \sL_s, j_0 \in \sL_t \) be such that \( a_{i_0} = b_{j_0} \). Since \( X \models \mathsf{Min}^*[a,a_0, a_{i_0}] \wedge \mathsf{Min}^*[a , b_0 ,b_{j_0}] \), then \( a_0 = b_0 \). It follows that \( X \models \mathsf{Min}^*[a,b_0,a_i] \) for any given \( i \in \sL_s \), whence \( a_i = b_j  \in B\) for some \( j \in \sL_t \).

 The fact that if \( A = B \) then \( s = t \) follows from the fact that  \( X \restriction  A \) and \( X \restriction B \) are isomorphic, respectively, to \( \sL_s \) and \( \sL_t \), and that \( \sL_s \cong \sL_t \iff s=t \) by~\ref{cond:C5}.
\end{proof}

Now let \( \mathsf{Seq}_s(x) \) be the \( \L_{\kappa^+ \kappa} \)-formula
\begin{equation} \tag{\( \mathsf{Seq}_s \)}
\exists \langle x_i \mid i \in \sL_s \rangle \left (\mathsf{Lab}_s(\langle x,x_i \mid i \in \sL_s \rangle) \right).
\end{equation}
Notice that if \( a \) is a point of a tree \( X \), \( X \models \mathsf{Seq}_s[a] \) implies \( X \models \mathsf{Seq}[a] \) (hence \( a \in X_{\mathsf{Seq}} \)).

Let \(\Upphi_4\) be the \( \L_{\kappa^+ \kappa} \)-sentence
\begin{multline} \tag{\( \Upphi_4 \)}
\forall x \bigwedge\nolimits_{s,t \in \pre{\SUCC(< \kappa)}{\kappa}} \forall \langle x_i \mid i \in \sL_s \rangle \, \forall \langle y_j \mid j \in \sL_t \rangle \\
\left [  \mathsf{Lab}_s(\langle x, x_i \mid i \in \sL_s \rangle) \wedge \mathsf{Lab}_t(\langle x, y_j \mid j \in \sL_t \rangle)  
 \Rightarrow \bigvee\nolimits_{\substack{i \in \sL_s \\ j \in \sL_t}} x_i = y_j \right ].
 \end{multline}

\begin{lemma} \label{lemmauniqueL_s}
Let \( X \) be a tree such that \( X \models \Upphi_4 \). Then for every \( a \in X \) there is at most one \( s \in \pre{\SUCC(< \kappa)}{\kappa} \) such that \( X \models \mathsf{Seq}_s[a] \). Moreover, if \( X \models \mathsf{Seq}_s[a] \) then the set of witnesses  \( \{ a_i \mid i \in \sL_s \} \subseteq X \) of this fact is unique.
\end{lemma}

\begin{proof}
Let \( a \in X \) and \( s,t \in \pre{\SUCC(< \kappa)}{\kappa} \) be such that \( X \models \mathsf{Seq}_s[a] \) and \( X \models \mathsf{Seq}_t[a] \),  and let \( \langle a_i \mid i \in \sL_s \rangle, \langle b_j \mid j \in \sL_t \rangle \) be two sequences of points from \( X \) witnessing these facts. Then by \( X \models \Upphi_4 \) the sets \( A = \{ a_i \mid i \in \sL_s \} \) and \( B = \{ b_j  \mid j \in \sL_t \} \) are not disjoint. Therefore  \(A = B \) by Lemma~\ref{lemmadisjointorequal}, and hence \( s = t \), as required.
\end{proof}

\noindent
If \( X,s,a, \{ a_i \mid i \in \sL_s \}  \) are such that \( X \models \bigwedge_{i \leq 4} \Upphi_i \) and \( X \models \mathsf{Lab}_s[\langle a,a_i \mid i \in \sL_s \rangle] \), we denote \( X \restriction \{ a_i \mid i \in \sL_s \}  \) by \( \sL^X_{s,a} \).
Notice also that for \( a,b \in X \), \( X \models 
\mathsf{Lab}^\in_s[a,b] \iff b \in \sL^X_{s,a} \).

Let \( \Upphi_5 \) be the \( \L_{\kappa^+ \kappa} \)-sentence
\begin{multline} \tag{\( \Upphi_5 \)}
\bigwedge\nolimits_{s \in \pre{\SUCC(< \kappa)}{\kappa}} \exists ! x \, \mathsf{Seq}_s(x) \wedge \forall x \left (\mathsf{Seq}(x) \Rightarrow \bigvee\nolimits_{s \in \pre{\SUCC(< \kappa)}{\kappa}} \mathsf{Seq}_s(x) \right ).
 \end{multline}

\begin{remark}\label{remsigma_X}
If \( X \) is a tree such that \( X \models \bigwedge_{i \leq 5}\Upphi_i \), then there is a bijection \( \sigma_X \) between \( \pre{\SUCC(<\kappa)}{\kappa} \) and \( X_{\mathsf{Seq}}  \), namely \( \sigma_X(s) = \) the unique \( a \in X_{\mathsf{Seq}} \) such that \( X \models \mathsf{Seq}_s[a] \). % Moreover, if \( X \) further satisfies \( \Upphi_3 \), then \( \sigma_X \) is also injective, hence a bijection.
\end{remark}

Let \( \Upphi_6 \) be the \( \L_{\kappa^+ \kappa} \)-sentence
\begin{multline} \tag{\( \Upphi_6 \)}
\forall x, y \left [ {\bigwedge\nolimits_{\substack{s,t \in \pre{\SUCC(< \kappa)}{\kappa} \\ s \subseteq t }} ({\mathsf{Seq}_s(x) \wedge \mathsf{Seq}_t(y)} \Rightarrow {x \preceq y})} \, \wedge \right. \\
\left. {\bigwedge\nolimits_{\substack{s,t \in \pre{\SUCC(< \kappa)}{\kappa} \\ s \not\subseteq t }} ({\mathsf{Seq}_s(x) \wedge \mathsf{Seq}_t(y)} \Rightarrow {x \not\preceq y})} \right ].
 \end{multline}

\begin{remark} \label{remPhi_6}
If \( X \) is a tree such that \( X \models \bigwedge_{i \leq 6} \Upphi_i \), then the map \( \sigma_X \) defined in Remark~\ref{remsigma_X} is actually an isomorphism between \( \left( \pre{\SUCC(< \kappa)}{\kappa}, \subseteq \right) \) and \( X \restriction X_{\mathsf{Seq}} \).
\end{remark}

Let \( \mathsf{ImSucc}(x,y) \)  be the \( \L_{\kappa^+ \kappa} \)-formula
\begin{equation} \tag{\( \mathsf{ImSucc} \)}
x \prec y \wedge \neg \exists z (x \prec z \prec y ),
\end{equation}
so that if \( X \) is a tree and \( a,b \in X \), then \( X \models \mathsf{ImSucc}[a,b] \) if and only if \( b \) is an immediate successor of \( a \). Let also \( \mathsf{Spine}(x) \) be the \( \L_{\kappa^+ \kappa} \)-formula
\begin{equation} \tag{\( \mathsf{Spine} \)}
 \exists w \, \exists w' ({x \preceq w} \wedge {x \preceq w'} \wedge {w \perp w'} ),
\end{equation}
so that if \( X \) is a tree and \( a \in X \), then \( X \models \mathsf{Spine}[a] \) if and only if \( \Cone(a) \) is not a linear order.

Let  \( \mathsf{Lab_{III}}^\in(x,y) \) be the \( \L_{\kappa^+ \kappa} \)-formula
\begin{multline} \tag{\( \mathsf{Lab_{III}}^\in \)}
\mathsf{Root}(x,y) \wedge \neg \mathsf{Stem}^\in(x,y) \wedge 
\bigwedge_{s \in \pre{\SUCC(<\kappa)}{\kappa}} \neg 
\mathsf{Lab}^\in_s(x,y) \wedge \\
 \forall x' \left [\mathsf{Min}^*(x,x',y) \Rightarrow \exists w \, \mathsf{ImSucc}(x',w) \right] .
 \end{multline}

\begin{remark} \label{remLab_III}
 Notice that if \( X \) is a tree and \( a,b,c \in X \) are such that \( X \models \mathsf{Lab_{III}}^\in[a,b] \) and \( b \preceq^X c \), then also \( X \models \mathsf{Lab_{III}}^\in[a,c] \). 
\end{remark}

Given \( \alpha < \kappa \), consider the structure \( \alpha= (\alpha, \leq ) \). Let \( \mathsf{Lab_{III}}^{\alpha}(\langle x,y,z_i \mid i \in \alpha \rangle) \) be the 
\( \L_{\kappa^+ \kappa} \)-formula
\begin{multline}\tag{\( \mathsf{Lab_{III}}^{\alpha} \)}
{\mathsf{Lab_{III}}^\in(x,y)} \wedge \mathsf{Spine}(y) \wedge \neg \mathsf{Min}(x,y) \wedge {\bigwedge\nolimits_{i \in \alpha} (y \prec z_i)} \wedge \\
 {\tau_{\mathsf{qf}}(\alpha)(\langle z_i \mid i \in \alpha \rangle)} \wedge {
\forall w \left ( {{y \prec w} \wedge {\bigvee\nolimits_{i \in \alpha} w \not\perp z_i}} \Rightarrow {\bigvee\nolimits_{i \in \alpha} w = z_i} \right )}.
 \end{multline}

Let \( \Upphi_7 \) be the \( \L_{\kappa^+ \kappa} \)-sentence
\begin{multline}\tag{\( \Upphi_7 \)} 
\forall x \, \forall y \bigwedge\nolimits_{\alpha,\beta < \kappa} \forall \langle z_i \mid i \in \alpha \rangle \, \forall \langle w_j \mid j \in \beta \rangle \\
\left ( \mathsf{Lab_{III}}^\alpha(\langle x,y,z_i \mid i \in \alpha \rangle) \wedge \mathsf{Lab_{III}}^\beta(\langle x,y,w_j \mid j \in \beta \rangle) \Rightarrow \bigvee\nolimits_{\substack{i \in \alpha \\ j \in \beta}} z_i = w_j \right ).
\end{multline}

\begin{remark} \label{remPhiIII_7}
The same argument contained in the proof of Lemma~\ref{lemmadisjointorequal} gives the following:
Let \( a, b \in X \) (for \( X \) a tree). Let \( \alpha, \beta  < \kappa\) and \( \langle c_i \mid i \in \alpha \rangle, \langle d_j \mid j \in \beta \rangle \) be sequences of elements of \( X \) such that both \( X \models \mathsf{Lab_{III}}^\alpha[\langle a,b,c_i \mid i \in \alpha \rangle] \) and \( X \models \mathsf{Lab_{III}}^\beta[\langle a,b,d_j \mid j \in \beta \rangle] \). Then the sets \( C = \{ c_i \mid i \in \alpha \} \) and \( D = \{ d_j \mid j \in \beta \} \) are either disjoint or coincide.
Therefore, if \( X \models \bigwedge_{i \leq 7} \Upphi_i \) then \( C = D \).
Since ordinals are determined by their isomporphism types, we get \( \alpha = \beta \).
\end{remark}

Now we formulate how the labels of type III are attached to the root.
Let \( \Upphi_8 \) be the \( \L_{\kappa^+ \kappa} \)-sentence
\begin{multline}\tag{\( \Upphi_8 \)}
\forall x \, \forall x' \, \forall y \Big [  \mathsf{Min}(x,x') \wedge \mathsf{Lab_{III}}^\in(x,x') \wedge x' \prec y  \Rightarrow  \Big (  ( \mathsf{ImSucc}(x',y) \wedge \neg \mathsf{Spine}(y))   \vee \\ 
\mathsf{Spine}(y)  \vee \exists z \bigvee\nolimits_{\alpha < \kappa}\exists \langle w_i \mid i \in \alpha \rangle \left (\mathsf{Lab_{III}}^\alpha(\langle x,z,w_i \mid i \in \alpha \rangle) \wedge \bigvee\nolimits_{i \in \alpha} y = w_i \right ) \Big ) \Big ].
 \end{multline}

Notice that the three conditions on $y$ in the disjunction after the implication are mutually exclusive.

For \( \alpha < \kappa \), let \( \mathsf{Lab_{III}}^{= \alpha } (x,y) \) be the 
\( \L_{\kappa^+ \kappa} \)-formula
\begin{equation}\tag{\( \mathsf{Lab_{III}}^{= \alpha }  \)}
\exists \langle z_i \mid i \in \alpha \rangle\left (\mathsf{Lab_{III}}^\alpha(\langle x,y,z_i \mid i \in \alpha \rangle)\right).
 \end{equation}

The next formula describes the connection of elements in the label of type III to unique ordinals. 
Let \( \Upphi_9 \) be the \( \L_{\kappa^+ \kappa} \)-sentence 
\begin{multline}\tag{\( \Upphi_9 \)}
\forall x \forall x' \Big [ \mathsf{Seq}(x) \wedge \mathsf{Min}(x,x') \wedge \mathsf{Lab_{III}}^\in(x,x')  \Rightarrow  \Big ( \exists ! y \, \mathsf{ImSucc}(x',y) \wedge  \\
\forall y (\mathsf{ImSucc}(x',y) \Rightarrow \neg \exists w (y \prec w)) \wedge
 \bigwedge\nolimits_{0 < \alpha < \kappa} \exists ! y \left(x' \prec y \wedge \mathsf{Lab_{III}}^{= \alpha}(x,y) \right) \wedge \\  %\wedge \forall z \left(\mathsf{Min}^*(x,x',z) \wedge \mathsf{Lab_{III}}^{=\alpha}(x,z) \Rightarrow z = y \right)\right) \right )  
 \forall y \Big ( x' \prec y \wedge \mathsf{Spine}(y)  \Rightarrow \bigvee\nolimits_{0 < \alpha < \kappa} \mathsf{Lab_{III}}^{=\alpha}(x,y) \Big ) \Big)\Big ].
 \end{multline}

Now we pin down the $u$ in the label $\sL^*_u$.

Let \( \mathsf{Lab}^*_u(x,x') \) be the  \( \L_{\kappa^+ \kappa} \)-formula

\begin{multline}\tag{\( \mathsf{Lab}^*_u \)}
\mathsf{Seq}(x) \wedge \mathsf{Min}(x,x') \wedge \mathsf{Lab_{III}}^\in(x,x') \wedge \\
 \bigwedge\nolimits_{\substack{0 < \alpha,\beta < \kappa \\ \alpha \preceq^{L_u} \beta}} \forall y \, \forall z \left(x' \prec y \wedge x' \prec z \wedge {\mathsf{Lab_{III}}^{= \alpha}(x,y) \wedge \mathsf{Lab_{III}}^{= \beta}(x,z)} \Rightarrow {y \preceq z}\right)  \wedge \\
  \bigwedge\nolimits_{\substack{0 < \alpha,\beta < \kappa \\ \alpha \not\preceq^{L_u} \beta}} \forall y \, \forall z \left (x' \prec y \wedge x' \prec z \wedge {\mathsf{Lab_{III}}^{= \alpha}(x,y) \wedge \mathsf{Lab_{III}}^{= \beta}(x,z)} \Rightarrow {y \not\preceq z}\right) 
\end{multline}

\begin{multline} \tag{\( \Upphi_{10} \)}
\forall x \, \forall x'\,  \forall x'' \, \Bigg [ \Bigg( \mathsf{Seq}(x) \wedge \mathsf{Min}(x,x') \wedge \mathsf{Lab_{III}}^\in(x,x') \Rightarrow \bigvee_{u \in \pre{\SUCC(< \kappa)}{2}} \mathsf{Lab}^*_u(x,x') \Bigg) \wedge \\
\bigwedge_{u \in \pre{\SUCC(< \kappa)}{2}} \Big( \mathsf{Lab}^*_u(x,x') \wedge \mathsf{Lab}^*_u(x,x'') \Rightarrow x' = x''  \Big)
 \Bigg]
\end{multline}

%Finally, let \( \Upphi_{10} \) be the \( \L_{\kappa^+ \kappa} \)-sentence
%\begin{multline}\tag{\( \Upphi_{10} \)} 
%\forall x \bigwedge_{s \in \pre{\SUCC(< \kappa)}{\kappa}} \left [ \mathsf{Seq}_s(x) \Rightarrow  \left ( \bigwedge\nolimits_{\substack{\alpha,\beta < \kappa \\ \alpha \preceq^{L_{\bar{\gamma}(s)}} \beta}} \forall y \, \forall z \left({\mathsf{Lab_I}^{= \alpha}(x,y) \wedge \mathsf{Lab_I}^{= \beta}(x,z)} \Rightarrow {y \preceq z}\right) \right ) \wedge\right. \\
%\left. \left ( \bigwedge\nolimits_{\substack{\alpha,\beta < \kappa \\ \alpha \not{\preceq^{L_{\bar{\gamma}(s)}}} \beta}} \forall y \, \forall z \left ({\mathsf{Lab_I}^{= \alpha}(x,y) \wedge \mathsf{Lab_I}^{= \beta}(x,z)} \Rightarrow {y \not\preceq z}\right) \right ) \right ].
% \end{multline}
%
%\begin{remark} \label{remPhi_10}
%Let \( X \) be a tree such that \( X \models \Upphi_7 \wedge \Upphi_8 \wedge \Upphi_9 \wedge \Upphi_{10} \). Using Remark~\ref{remPhi_7-9}, if \( s \in \pre{\SUCC(< \kappa)}{\kappa} \) and \( a \in X_{\mathsf{Seq}} \) are such that \( X \models \mathsf{Seq}_s[a] \), then  \( X \restriction X'_a \)  is isomorphic to \( \sL_{\bar{\gamma}(s)} \). In this case, the structure \( X \restriction X'_a \) will be denoted by \( \sL^X_{\bar{\gamma}(s),a} \).
%\end{remark}

\begin{remark} \label{remPhiIII_7-9}
Let \( X \) be a tree such that \( X \models \bigwedge_{i \leq 10} \Upphi_i \) and \( a \in X_{\mathsf{Seq}} \). 
Then some of the points in 
\( X_a \) belong to the stem \( \sS^X_a \) of \( a \) and \( \Cone( \sS^X_a ) \) is partitioned 
in maximal connected components each of which has a minimum. One of these components is 
a label of type II (namely, to \( \sL^X_{s,a} \), where \( s \in 
\pre{\SUCC(< \kappa)}{\kappa} \) is the unique sequence such that \( X 
\models \mathsf{Seq}_s[a] \)). 
Suppose now that \( b \) is the  minimal element of some of the other connected components, namely \( \Cone(b) \), and suppose that \( b \) has an immediate successor.
 Then by \( X 
\models \Upphi_9 \) there is a bijection \( l_b \) from \( \kappa \) onto the points 
\( c \in \Cone(b) \) such that \( \Cone(c) \) is not a linear order, namely \( l_b(\alpha) = \) the unique \( c \in \Cone(b) \) such that \( X \models \mathsf{Lab_{III}}^{=\alpha}[a,c] \) (for \(\alpha < \kappa \)). By \( X \models \Upphi_{10} \) we actually get that \( l_b \) is an isomorphism between \( L_u \) and its range (for some \( u \in \pre{\SUCC(< \kappa)}{2} \)), and by \( X \models \Upphi_8 \wedge \Upphi_9 \)  each remaining point of \( \Cone(b) \), i.e.\ each point \( c \in \Cone(b) \) 
such that \( \Cone(c) \) is a linear order, either it is the unique immediate successor of \( b \) (and it is terminal in \( X \)), or else it  belongs to the unique (by Remark~\ref{remPhiIII_7}) sequence witnessing \( X \models \mathsf{Lab_{III}}^{=\alpha}[a,l_b(\alpha)] \) 
(for some \( \alpha < \kappa \)). It follows that \( l_b \) can be extended to a (unique) isomorphism, which we denote by \( l_{u,b} \), between \( \sL^*_u \) and \( X \restriction \Cone(b) \). 
Moreover, by \( X \models \Upphi_{10} \) for every \( u \in \pre{\SUCC(<\kappa)}{2} \) there is at most one \( b \) as above such that \( \sL^*_u \cong X \restriction \Cone(b) \): therefore we can unambiguously denote the last structure by \( \sL^{*\, X}_{u,a} \).
\end{remark}

Using similar ideas, we now provide \( \sL_{\kappa^+ \kappa} \)-sentences asserting that there is just one maximal connected component of each \( \Cone( \sS^X_s ) \) which is not of the form \( \sL^X_{s,a} \) or \( \sL^{*\, X}_{u,a} \), and that such component is isomorphic to \( \sL_{\lh(s)} \), where \( s \in \pre{\SUCC(< \kappa)}{2} \) is unique such that \( X \models \mathsf{Seq}_s[a] \). 
Let \( \mathsf{Lab_I}^\in(x,y) \) be the \( \L_{\kappa^+ \kappa} \)-formula
\begin{multline} \tag{\( \mathsf{Lab_I}^\in \)}
\mathsf{Root}(x,y) \wedge \neg \mathsf{Stem}^\in(x,y) \wedge 
\bigwedge_{s \in \pre{\SUCC(<\kappa)}{\kappa}} \neg 
\mathsf{Lab}^\in_s(x,y) \wedge \\
 \forall x' \left [\mathsf{Min}^*(x,x',y) \Rightarrow \neg \exists w \, \mathsf{ImSucc}(x',w) \right] .
 \end{multline}

Let \( \Upphi_{11} \) be the \( \L_{\kappa^+ \kappa} \)-sentence
\begin{equation}\tag{\( \Upphi_{11} \)}
\forall x (\mathsf{Seq}(x) \Rightarrow \exists ! x' (\mathsf{Min}(x,x') \wedge \mathsf{Lab_I}^\in(x,x')).
\end{equation}

\begin{remark} \label{remLab_I}
 Notice that also in this case if \( X \) is a tree and \( a,b,c \in X \) are such that \( X \models \mathsf{Lab_I}^\in[a,b] \) and \( b \preceq^X c \), then \( X \models \mathsf{Lab_I}^\in[a,c] \). Moreover, if \( X \models \bigwedge_{i \leq 11} \Upphi_i \), then for each \( a \in X_{\mathsf{Seq}} \) there is a unique \( a \preceq^X b \) such that \( \Cone(b) \) is a maximal connected component of \( \Cone( \sS^X_a ) \) but \( b \) has no immediate successor. 
\end{remark}

Given \( \alpha < \kappa \), consider the structure \( \alpha= (\alpha, \leq ) \). Then let \( \mathsf{Lab_I}^{\alpha}(\langle x,y,z_i \mid i \in \alpha \rangle) \) be the 
\( \L_{\kappa^+ \kappa} \)-sentence
\begin{multline}\tag{\( \mathsf{Lab_I}^{\alpha} \)}
{\mathsf{Lab_I}^\in(x,y)} \wedge \mathsf{Spine}(y) \wedge {\bigwedge\nolimits_{i \in \alpha} (y \prec z_i)} \wedge \\
 {\tau_{\mathsf{qf}}(\alpha)(\langle z_i \mid i \in \alpha \rangle)} \wedge {
\forall w \left ( {{y \prec w} \wedge {\bigvee\nolimits_{i \in \alpha} w \not\perp z_i}} \Rightarrow {\bigvee\nolimits_{i \in \alpha} w = z_i} \right )}.
 \end{multline}

Let \( \Upphi_{12} \) be the \( \L_{\kappa^+ \kappa} \)-sentence
\begin{multline}\tag{\( \Upphi_{12} \)} 
\forall x \, \forall y \bigwedge\nolimits_{\alpha,\beta < \kappa} \forall \langle z_i \mid i \in \alpha \rangle \, \forall \langle w_j \mid j \in \beta \rangle \\
\left ( \mathsf{Lab_I}^\alpha(\langle x,y,z_i \mid i \in \alpha \rangle) \wedge \mathsf{Lab_I}^\beta(\langle x,y,w_j \mid j \in \beta \rangle) \Rightarrow \bigvee\nolimits_{\substack{i \in \alpha \\ j \in \beta}} z_i = w_j \right ).
\end{multline}

\begin{remark} \label{remPhi_7}
Arguing again as in the proof of Lemma~\ref{lemmadisjointorequal} we have the following:
Let \( a, b \in X \) (for \( X \) a tree). Let \( \alpha, \beta  < \kappa\) and \( \langle c_i \mid i \in \alpha \rangle, \langle d_j \mid j \in \beta \rangle \) be sequences of elements of \( X \) such that both \( X \models \mathsf{Lab_I}^\alpha[\langle a,b,c_i \mid i \in \alpha \rangle] \) and \( X \models \mathsf{Lab_I}^\beta[\langle a,b,d_j \mid j \in \beta \rangle] \). Then the sets \( C = \{ c_i \mid i \in \alpha \} \) and \( D = \{ d_j \mid j \in \beta \} \) are either disjoint or coincide.
Therefore, if \( X \models \Upphi_7\) then \( C = D \), and hence \( \alpha = \beta \).
\end{remark}

Let \( \Upphi_{13} \) be the \( \L_{\kappa^+ \kappa} \)-sentence
\begin{multline}\tag{\( \Upphi_{13} \)}
\forall x \, \forall y \left [  \mathsf{Lab_I}^\in(x,y)  \Rightarrow  \left ( \vphantom{\bigvee\nolimits_{\alpha < \kappa}} \mathsf{Spine}(y) \vee \right. \right. \\ 
\left. \left. \exists z \bigvee\nolimits_{\alpha < \kappa}\exists \langle w_i \mid i \in \alpha \rangle \left (\mathsf{Lab_I}^\alpha(\langle x,z,w_i \mid i \in \alpha \rangle) \wedge \bigvee\nolimits_{i \in \alpha} y = w_i \right ) \right ) \right ].
 \end{multline}

For \( \alpha < \kappa \), let \( \mathsf{Lab_I}^{= \alpha } (x,y) \) be the 
\( \L_{\kappa^+ \kappa} \)-formula
\begin{equation}\tag{\( \mathsf{Lab_I}^{= \alpha }  \)}
\exists \langle z_i \mid i \in \alpha \rangle\left (\mathsf{Lab_I}^\alpha(\langle x,y,z_i \mid i \in \alpha \rangle)\right).
 \end{equation}

Let \( \Upphi_{14} \) be the \( \L_{\kappa^+ \kappa} \)-sentence
\begin{multline}\tag{\( \Upphi_{14} \)}
\forall x  \Big [ \mathsf{Seq}(x)  \Rightarrow  \Big (  \bigwedge\nolimits_{0 < \alpha < \kappa} \exists ! y \, \mathsf{Lab_{I}}^{= \alpha}(x,y)  \wedge \\  %\wedge \forall z \left(\mathsf{Min}^*(x,x',z) \wedge \mathsf{Lab_{III}}^{=\alpha}(x,z) \Rightarrow z = y \right)\right) \right )  
 \forall y \Big ( \mathsf{Lab_I}^\in(x,y) \wedge \mathsf{Spine}(y)  \Rightarrow \bigvee\nolimits_{0 < \alpha < \kappa} \mathsf{Lab_{I}}^{=\alpha}(x,y) \Big ) \wedge   \Big)\Big ].
 \end{multline}

Finally, let \( \Upphi_{15} \) be the \( \L_{\kappa^+ \kappa} \)-sentence
\begin{multline}\tag{\( \Upphi_{15} \)} 
\forall x \bigwedge_{s \in \pre{\SUCC(< \kappa)}{\kappa}} \left [ \mathsf{Seq}_s(x) \Rightarrow  \left ( \bigwedge\nolimits_{\substack{\alpha,\beta < \kappa \\ \alpha \preceq^{L_{\bar{\gamma}(s)}} \beta}} \forall y \, \forall z \left({\mathsf{Lab_I}^{= \alpha}(x,y) \wedge \mathsf{Lab_I}^{= \beta}(x,z)} \Rightarrow {y \preceq z}\right) \right ) \wedge\right. \\
\left. \left ( \bigwedge\nolimits_{\substack{\alpha,\beta < \kappa \\ \alpha \not{\preceq^{L_{\bar{\gamma}(s)}}} \beta}} \forall y \, \forall z \left ({\mathsf{Lab_I}^{= \alpha}(x,y) \wedge \mathsf{Lab_I}^{= \beta}(x,z)} \Rightarrow {y \not\preceq z}\right) \right ) \right ].
 \end{multline}

\begin{remark} \label{remPhi_10}
Arguing as in Remark~\ref{remPhiIII_7-9}, we get that if \( X \) is a tree such that \( X \models \bigwedge_{i \leq 10} \Upphi_i \) and \( a \in X_{\mathsf{Seq}} \) is such that \( X \models \mathsf{Seq}_s[a] \) (for the appropriate \( s \in \pre{\SUCC(< \kappa)}{\kappa} \)), then among the maximal connected components of \( \Cone( \sS^X_a ) \) we find a unique label of type II coding \( s \) and, possibly, some labels of type III coding certain \( u \in \pre{\SUCC(< \kappa)}{2} \). If moreover \( X \models \bigwedge_{11 \leq i \leq 15} \), then in \( \Cone( \sS^X_a ) \) there is also a  unique maximal connected component \( \Cone(b) \) (for some \( b \) minimal in \( \Cone( \sS^X_a ) \)) of type I coding exactly \( \bar{\gamma}(s) \), which will be denoted by \( \sL^X_{\bar{\gamma}(s),a} \). To fix the notation, we let \( l_{\bar{\gamma},a} \) denote the (unique) isomorphism between \( \sL_{\bar{\gamma}} \) and \( \sL^X_{\bar{\gamma}(s),a} \).
\end{remark}

\begin{definition}\label{defPsi}
Let now \( \Uppsi \)  be the \( \L_{\kappa^+ \kappa} \)-sentence given by the  conjunction
\begin{equation}\tag{\( \Uppsi \)}
\bigwedge\nolimits_{i \leq 15} \Upphi_i.
 \end{equation}
\end{definition}

\begin{remark} \label{remrecap}
Suppose \( X \models \Uppsi \). Collecting all the remarks above, we have the following description of \( X \):
\begin{enumerate-(1)}
\item \label{condtree}
\( X \) is a tree (by \( X \models \Upphi_0 \));
\item \label{condsigma_X}
there is an isomorphism \( \sigma_X \) between \( \left(\pre{\SUCC(< \kappa)}
{\kappa} , \subseteq \right) \) and the substructure of \( X \) with domain \( 
X_{\mathsf{Seq}} = \{ a \in X \mid \pred(a) \text{ is well-founded} \} \), 
which is a \( \preceq^X \)-downward closed subset of \( X \) (Remarks~\ref{remsigma_X},~\ref{remPhi_6}, and~\ref{remroot});
\item \label{condX_a}
by Remark~\ref{remPhi_1}, for every point \( b \) in \( X \setminus X_{\mathsf{Seq}} \) there is a 
(unique) \( \preceq^X \)-maximal element \( a^b \) in \( \pred(b) \) which is 
in \( X_{\mathsf{Seq}} \): given \( s \in \pre{\SUCC(< \kappa)}
{\kappa} \), denote by \( X_{\sigma_X(s)} \) the collection of all
\( b  \in X \setminus X_{\mathsf{Seq}} \) such that \( a^b = \sigma_X(s) \), and notice that \( X_{\sigma_X(s)} 
\) is necessarily \( \preceq^X \)-upward closed. 
Moreover, for every \( s,t  \in \pre{\SUCC(< \kappa)}{\kappa}\) we have (see Remark~\ref{remPhi_1}):
\begin{enumerate-(a)}
\item
if \( s,t \) are distinct then for every \( b \in X_{\sigma_X(s)} , b' \in X_{\sigma_X(t)} \), \( b \) and \( b' \) are incomparable;
\item
  if \( b  \in X_{\sigma_X(s)} \) then \( \sigma_X(t) \preceq^X b \) if and only if \( t \subseteq s \)% and \( b \not\preceq^X \sigma_X(t) \)
  ;
\end{enumerate-(a)}
\item \label{condstem}
at the bottom of each \( X_{\sigma_X(s)} \) there is an isomorphic copy of \( ( \ZZ, \leq ) \), called stem of \( \sigma_X(s) \) and denoted by \( \sS^X_{\sigma_X(s)} \): all other elements of \( X_{\sigma_X(s)} \) are \( \preceq^X \)-above (all the points in) \( \sS^X_{\sigma_X(s)} \) (Remark~\ref{remPhi_2}), and their collection is denoted by \( \Cone( \sS^X_{\sigma_X(s)} ) \);
\item \label{condabovestem}
call a substructure \( X' \) of \( \Cone( \sS^X_{\sigma_X(s)} ) \) \emph{maximal} if it is a maximal connected component of \( \Cone( \sS^X_{\sigma_X(s)} ) \). Moreover, let \( U^X_s \) be the collection of all \( u \in \pre{\SUCC(< \kappa)}{2} \) for which there is a maximal substructure of \( \Cone( \sS^X_{\sigma_X(s)} ) \) which is a code for \( u \). Then above the stem of \( \sigma_X(s) \) there is 
\begin{enumerate-(a)}
\item \label{condL_s}
a (unique) maximal substructure \( \sL_{s,\sigma_X(s)}^X \) of \( X_{\sigma_X(s)} \) which is a code for \( s \), i.e.\ it is isomorphic to \( \sL_s \) (Lemma~\ref{lemmauniqueL_s});
\item \label{condL_u}
for each \( u \in U^X_s \), a (unique) maximal substructure \( \sL^{*\, X}_{u,\sigma_X(s)} \) of \( X_{\sigma_X(s)} \) which is a code for \( u \), i.e.\ it is isomorphic to \( \sL^*_u \) (Remark~\ref{remPhiIII_7-9});
\end{enumerate-(a)}
\item \label{condL_gamma}
the remaining points above  \( \sS^X_{\sigma_X(s)} \) form a maximal substructure \( \sL^X_{\bar{\gamma}(s),\sigma_X(s)} \) of \( X_{\sigma_X(s)} \) which is a code for \( \bar{\gamma}(s) \), i.e.\ it is isomorphic to \( \sL_{\bar{\gamma}(s)} \) (Remark~\ref{remPhi_10}).
\end{enumerate-(1)}
\end{remark}

Therefore one immediately gets:

\begin{lemma}\label{lemmarangef} 
Let \( \kappa \) be an uncountable cardinal satisfying~\eqref{eq:kappa}, \( R \) be an analytic quasi-order on \( \pre{\kappa}{2} \), and \( f \) be the function defined in~\eqref{eqf}. Then \( \range(f) \subseteq \Mod^\kappa_\Uppsi \). 
\end{lemma}

A structure \( X \in \Mod^\kappa_\Uppsi \) may fail to be in \( \range(f) \) only because its substructures of the form \( \sL^{*\, X}_{u,\sigma_X(s)} \) (or, more precisely, the sets \( U^X_s \), see Remark~\ref{remrecap}\ref{condabovestem}) are not coherent with any of the \( x \in \pre{\kappa}{2} \). Indeed, if \( X = f(x) = G_{s_T(x)} \), or even just \( X \cong f(x) \), then we have the following:
\begin{itemizenew}
\item
by Lemma~\ref{lemmanormalform}\ref{lemmanormalformc2} and the definition of \( f \), for each \( \omega \leq \gamma < \kappa \) the set \( U^X_{0^{(\gamma+1)}} \) contains a unique element, namely \( x \restriction (\gamma+1) \); clearly, all the elements in these singletons are pairwise comparable with respect to inclusion;
\item
by definition of \( f \) again, for all other \( s \in \pre{\SUCC(<\kappa)}{\kappa} \) the set \( U^X_s \) can be canonically recovered from the unique element in \( U^X_{0^{(\gamma+1)}} \), where \( \gamma < \kappa \) is any infinite ordinal such that \( \lh(s) \leq \gamma + 1 \): in fact, \( U^X_s \) consists of all \( u \in \pre{\lh(s)}{2} \) such that \( (u,s) \in S^y_T \), where \( S^y_T = s_T(y) \) is as in~\eqref{eqs_T}, for some/any \( y \in \pre{\kappa}{2} \) such that \( y \restriction (\gamma+1) = x \restriction (\gamma+1) \) (equivalently: \( y \restriction (\gamma+1) \in U^X_{0^{(\gamma+1)}} \)).
\end{itemizenew}

The above two conditions actually characterize the elements in (the closure under isomorphism of)  \( \range(f) \), and can thus be used to detect whether a given \( X \in \Mod^\kappa_\Uppsi \) is isomorphic to an element of \( \range(f) \) or not: First one  requires that each \( U^X_{0^{(\gamma+1)}} \) is a singleton \( \{ u_\gamma \} \) with \( u_\gamma \in \pre{\gamma+1}{2} \), and that all the \( u_\gamma \)'s are compatible (for all infinite \( \gamma < \kappa \)). This allows one to isolate the unique candidate \( x = \bigcup_{ \omega \leq \gamma < \kappa} u_\gamma \in \pre{\kappa}{2} \) for which it could happen that \( X \cong f(x) \). Then it only remains to check whether all other \( U^X_s \) are actually constructed coherently to the guess \( X \cong f(x) \). 

We are now going to show that this ``recovering procedure'' can described within the logic \( \L_{\kappa^+ \kappa} \). In what follows, we adopt the notation and terminology introduced in this chapter, and in particular in Remark~\ref{remrecap}.

Given \( u \in \pre{\SUCC(< \kappa)}{2} \) let \( \mathsf{Lab}^{**}_u(x) \) be the \( \L_{\kappa^+ \kappa} \)-formula
\begin{equation}\tag{\( \mathsf{Lab}^{**}_u \)}
\exists y \, \mathsf{Lab}_u^*(x,y).
\end{equation}

\begin{remark}
Given \( X \in \Mod^\kappa_\Uppsi \) and \( a \in X \), we have \( X \models \mathsf{Lab}^{**}_u[a] \) if and only if \( u \in U^X_s \), where \( s = \sigma_X^{-1}(a) \).
\end{remark}

Let now \( \Upphi_{16} \) and \( \Upphi_{17} \) be the \( \L_{\kappa^+ \kappa} \)-sentences

\begin{multline}\tag{\( \Upphi_{16} \)}
 \bigwedge_{\omega \leq \gamma<\kappa} \Bigg[ \forall x \big(\mathsf{Seq}_{0^{(\gamma+1)}}(x) \Rightarrow 
\bigvee_{u \in \pre{\gamma+1}{2} } \mathsf{Lab}_u^{**}(x)\big) \wedge
\\ 
\bigwedge_{u,v\in \pre{\gamma+1}{ 2}, 
u \neq v} \forall x \big(\mathsf{Seq}_{0^{(\gamma+1)}}(x) \wedge
\mathsf{Lab}^{**}_u(x) \Rightarrow \neg \mathsf{Lab}_v^{**}(x) \big )\Bigg].
\end{multline}

\begin{equation}\tag{\( \Upphi_{17} \)}
 \bigwedge_{\omega \leq \gamma\leq\delta <\kappa} 
\bigwedge_{\substack{u \in {}^{\gamma+1} \kappa \\
v\in {}^{\delta+1}\kappa \\
u \not\subseteq v}}
 \forall x\forall y
(\mathsf{Seq}_{0^{(\gamma+1)}}(x) \wedge\mathsf{Seq}_{0^{(\delta+1)}}(y) 
\wedge \mathsf{Lab}_u^{**}(x) \rightarrow \neg \mathsf{Lab}_v^{**}(y))).
\end{equation}

\begin{remark} \label{rem:16}
If a structure \( X \in \Mod^\kappa_\Uppsi \) satisfies \( \Upphi_{16} \), then for any infinite \( \gamma < \kappa \) the set \( U^X_{0^{(\gamma+1)}} \) is a singleton \( \{ u^X_\gamma \} \) with \( \lh(u^X_\gamma) = \gamma + 1 \). If moreover \( X \models \Upphi_{17} \), then \( u^X_\gamma \subseteq u^X_\delta \) whenever \( \omega \leq \gamma \leq \delta < \kappa \).
\end{remark}

Finally, we introduce one last \( \L_{\kappa^+ \kappa} \)-sentence which, together with all the previous ones, identifies the structures which are isomorphic to an element of \( \range(f) \). Let \( R \) be an analytic quasi-order, and \( T \) be a DST-tree on \( 2 \times 2 \times \kappa \) as in Lemma~\ref{lemmanormalform}, and let \( S_T \) be the tree obtained from \( T \) as in~\eqref{eqS_T}. Finally, for every \( \gamma <\kappa \), \( s \in \pre{\gamma+1}{\kappa} \), and \( v \in \pre{\gamma+1}{2} \), let \( S^{v,s}_T = \{ u \in \pre{\gamma+1}{2} \mid (u,v,s) \in S_T \} \). Then \( \Upphi_{T} \) is the \( \L_{\kappa^+ \kappa} \)-sentence

\begin{multline}\tag{\( \Upphi_{T} \)}
\bigwedge_{\gamma<\kappa} \bigwedge _{s\in \pre{\gamma+1}{ \kappa}} 
\bigwedge_{v \in \pre{\gamma+1}{2}}
\forall x,y
\Bigg(\mathsf{Seq}_s(x) \wedge\mathsf{Seq}_{0^{(\gamma+1)}}(y)
\wedge \mathsf{Lab}_v^{**}(y) \\
\Rightarrow \Bigg (\bigwedge_{u \in S^{v,s}_T} \mathsf{Lab}_u^{**}(x)
\wedge \bigwedge_{u \notin S^{v,s}_T} \neg \mathsf{Lab}_u^{**}(x) \Bigg) \Bigg).
\end{multline}

\begin{definition} \label{def:phi_R}
Given an analytic quasi-order \( R \),
let \( \upvarphi_R \) be the \( \L_{\kappa^+ \kappa} \)-sentence
\begin{equation} \tag{\( \upvarphi_R \)}
\Uppsi \wedge \Upphi_{16} \wedge \Upphi_{17} \wedge \Upphi_{T}.
 \end{equation}
 \end{definition}
Define a map
\begin{equation} \label{eqh}
h \colon \Mod^\kappa_{\upvarphi_R} \to 2^\kappa \colon X \mapsto \bigcup_{\omega \leq \gamma < \kappa} u^X_\gamma,
\end{equation}
where \( u^X_\gamma \) is as in Remark~\ref{rem:16} --- the map \( h \) is well-defined because \( X \models \Upphi_{17} \).

\begin{proposition} \label{prop:inversemap}
Let \( R \) be an analytic quasi-order, let \( \upvarphi_R \) the \( \L_{\kappa^+ \kappa} \)-sentence from Definition~\ref{def:phi_R}, and let \( f \) and \( h \) be defined as in~\eqref{eqf} and~\eqref{eqh}, respectively.
\begin{enumerate-(i)}
\item \label{prop:inversemap-i}
\( \range{f} \subseteq \Mod^\kappa_{\upvarphi_R} \).
\item \label{prop:inversemap-ii}
The map \( h \) is a right-inverse of \( f \) modulo isomorphism, i.e.\ \( f(h(X)) \cong X \) for every \( X \in \Mod^\kappa_{\upvarphi_R} \).
\end{enumerate-(i)}
In particular, \( \Mod^\kappa_{\upvarphi_R} \) is the closure under isomorphism of \( \range(f) \).
\end{proposition}

\begin{proof}
Part~\ref{prop:inversemap-i} directly follows from the definition of \( f(x) = G_{s_T(x)} \) (see the paragraph after Lemma~\ref{lemmarangef}). For part~\ref{prop:inversemap-ii}, notice that by Remark~\ref{remrecap} there is an isomorphism between 
\[ 
X_{\mathsf{Seq}} \cup \bigcup \{ \sS^X_{\sigma_X(s)} \cup \sL^X_{s, \sigma_X(s)} \cup \sL^X_{\bar{\gamma}(s), \sigma_X(s)} \mid s \in \pre{\SUCC(< \kappa)}{2} \}
 \] 
and \( G_0 \). Such an isomorphism can clearly be extended to an isomorphism between \( X \) and \( f(h(X)) = G_{\sigma_T(h(X))} \) as soon as 
\[ 
U^X_s = S^{h(X) \restriction \lh(s),s}_T 
\]
for all \( s \in \pre{\SUCC(< \kappa)}{\kappa} \).
But this is guaranteed by \( X \models \Upphi_{T} \) (together with the definition of \( h \) in~\eqref{eqh}), hence we are done.
\end{proof}

\begin{corollary} \label{cor:inversemap}
Let \( R \) be an analytic quasi-order, and let \( f \) and \( h \) be defined as in~\eqref{eqf} and~\eqref{eqh}, respectively. Then \( h \) is a left-inverse of \( f \), and  \( h \) reduces the embeddability relation on \( \Mod^\kappa_{\upvarphi_R} \) to \( R \).
\end{corollary}

\begin{proof}
Towards a contradiction, assume that \( h(f(x)) \neq x \) for some \( x \in \pre{\kappa}{2} \), and let \( \gamma < \kappa \) be an infinite successor ordinal such that \( h(f(x)) \restriction \gamma \neq x \restriction \gamma \). Then 
\[ 
s_T(h(f(x))) \restriction \pre{\SUCC(< \kappa)}{(2 \times \kappa)} \neq s_T(x) \restriction \pre{\SUCC(< \kappa)}{(2 \times \kappa)} , 
\]
because by Lemma~\ref{lemmanormalform}\ref{lemmanormalformc2} the former would contain \(  (h(f(x)) \restriction \gamma, 0^{(\gamma)}) \) while the latter not. Thus setting \( X = f(x) \) we would get \( f(h(X)) \not\cong X \) by 
Theorem~\ref{theorcomplete}\ref{theorcomplete2}, contradicting Proposition~\ref{prop:inversemap}\ref{prop:inversemap-ii}. The fact that \( h \) is a reduction easily follows from Proposition~\ref{prop:inversemap}\ref{prop:inversemap-ii} and the fact that \( f \) reduces \( R \) to \( \sqsubseteq \restriction \Mod^\kappa_{\upvarphi_R} \) by Theorem~\ref{thm:complete} and Proposition~\ref{prop:inversemap}\ref{prop:inversemap-i}.
\end{proof}

Using essentially the same trick employed to obtain the \( \L_{\kappa^+ \kappa} \)-sentence \( \upvarphi_R \), one can show that the map \( h \) from~\eqref{eqh} is Borel. Indeed, notice that 
\[
\{ \bN_u \mid u \in \pre{\SUCC(< \kappa)}{2} \},
\]
where \( \bN_u \) is as in~\eqref{eqN_S}, is a basis of size \( \kappa \) for the bounded topology \( \mathscr{O}(\pre{\kappa}{2}) \) on \( \pre{\kappa}{2} \), so that it is enough to show that for each \( u \in \pre{\SUCC(< \kappa)}{2} \) the set \( h^{-1}(\bN_u) \) is Borel. By the (generalized) Lopez-Escobar theorem (see Section~\ref{subsec:inflogic}), this amounts to find an \( \L_{\kappa^+ \kappa} \)-sentence \( \upvarphi^u_R \) such that \( h^{-1}(\bN_u) = \Mod^\kappa_{\upvarphi^u_R} \).

\begin{proposition} \label{prop:hisborel}
Let \( R \) be an analytic quasi-order, let \( \upvarphi_R \) the \( \L_{\kappa^+ \kappa} \)-sentence from Definition~\ref{def:phi_R}, and let \( h \) be defined as in~\eqref{eqh}. Then for every \( u \in \pre{\SUCC(< \kappa)}{2} \) 
\[ 
h^{-1}(\bN_u) = \Mod^\kappa_{\upvarphi^u_R},
 \] 
where \( \upvarphi^u_R \) is the \( \L_{\kappa^+ \kappa} \)-sentence
\[ 
\upvarphi_R \, \wedge \, \exists x \left(\mathsf{Seq}_{0^{(\lh(u))}}(x) \wedge \mathsf{Lab}^{**}_u(x)\right).
 \] 
\end{proposition}

\begin{proof}
It is enough to observe that for every \( X \in \Mod^\kappa_{\upvarphi_R} \)
\[ 
h(X) \in \bN_u \iff u \in U^X_{0^{(\lh(u))}} \iff X \models \exists x \left(\mathsf{Seq}_{0^{(\lh(u))}}(x) \wedge \mathsf{Lab}^{**}_u(x)\right). \qedhere
 \] 
\end{proof}

We are now ready to prove the main result of this paper (compare it with~\cite[Theorem 10.23]{mottoros2011}).

\begin{theorem} \label{thm:main}
Let \( \kappa \) be any uncountable cardinal satisfying~\eqref{eq:kappa}. Then the embeddability relation \( \sqsubseteq^\kappa_{\mathsf{TREE}} \) is strongly invariantly universal, that is: For every analytic quasi-order \( R \) there is an \( \L_{\kappa^+ \kappa} \)-sentence \( \upvarphi \) (all of whose models are generalised trees) such that \( R \simeq_B {{\sqsubseteq} \restriction \Mod^\kappa_{\upvarphi} }\).

Therefore, also the bi-embeddability relation \( \equiv^\kappa_{\mathsf{TREE}} \) is strongly invariantly universal.
\end{theorem}

\begin{proof}
Given an analytic quasi-order \( R \), let \( \upvarphi \) be the \( \L_{\kappa^+ \kappa} \)-sentence from Definition~\ref{def:phi_R}, and consider the quotient map (with respect to \( E_R \) and \( {\equiv} \restriction \Mod^\kappa_\upvarphi \)) of the Borel function \( f \) from~\eqref{eqf}. Such a map is well defined by Theorem~\ref{thm:complete}, and witnesses \( R \simeq_B {{\sqsubseteq} \restriction \Mod^\kappa_{\upvarphi} } \): indeed, it is an isomorphism of the corresponding quotient orders by Theorem~\ref{thm:complete} again and Proposition~\ref{prop:inversemap}\ref{prop:inversemap-i}, and the function \( h \) from~\eqref{eqh} is a Borel lifting of its inverse by Propositions~\ref{prop:inversemap}\ref{prop:inversemap-ii}, Corollary~\ref{cor:inversemap}, and Proposition~\ref{prop:hisborel}.
\end{proof}

Finally, by~\cite[Remark 9.7]{mottoros2011} again we also obtain the analogous result for graphs (compare it with~\cite[Corollary 10.24]{mottoros2011}).

\begin{corollary} \label{cor:main}
Let \( \kappa \) be any uncountable cardinal satisfying~\eqref{eq:kappa}. Then the embeddability relation \( \sqsubseteq^\kappa_{\mathsf{GRAPH}} \) and  the bi-embeddability relation \( \equiv^\kappa_{\mathsf{GRAPH}} \) are both strongly invariantly universal.
\end{corollary}

%%%%%%%%%%%%%%%%%%%%%%%
\section{Embeddability on uncountable groups} \label{sec:groups}
%%%%%%%%%%%%%%%%%%%%%%%
Let \(\sqsubseteq_\mathsf{GROUPS}^\kappa\) be the embeddability quasi-order on the space of \(\kappa\)-sized groups.

\begin{theorem}[essentially {\cite[Theorem 5.1]{Wil14}}]\label{thm:Williams}
For every infinite cardinal \(\kappa\), the quasi-order
\(\sqsubseteq_\mathsf{GRAPHS}^\kappa\) Borel reduces to \(\sqsubseteq_\mathsf{GROUPS}^\kappa\).
\end{theorem}
Theorem~\ref{thm:Williams} was proved by Williams for \(\kappa=\omega\) but the same argument works for uncountable cardinalities. The proof uses the theory of presented groups and small cancellation theory. Before discussing it we introduce the terminology and recall the main notions.

For any set \( X \), let \(F_{X}\) be the \emph{free group on \(X\)}. The elements of \(F_{X}\), which are called \emph{words}, are finite sequences \(y_{0}\dotsm y_{n}\), where each \(y_{i}=x\) or \(y_{i}=x^{-1}\) for some \(x\in X\). Words are multiplied by concatenation, thus the identity of \(F_{X}\) is the empty sequence, usually denoted by \(1\).
A word \(y_{0}\dotsm y_{n}\) is said to be \emph{reduced} if for all \(i<n\), \((y_{i}, y_{i+1})\) does not form a pair \((x,x^{-1})\) or \((x^{-1},x)\).
As the notation may suggest, we work under the convention that \(xx^{-1}\) gives the empty sequence. Therefore, every element in \(F_{X}\) different than the identity has a unique representation as a reduced word.
Further, we say that a reduced word \(y_{0}\dotsm y_{n}\) is \emph{cyclically reduced} if \(y_{0}\) and \(y_{1}\) are not one the inverse of the other one.

Now suppose that \(R\) is a set of reduced words on \(F_{X}\).
We say that \(R\) is \emph{symmetrized} if it is closed under inverses and cyclic permutations. That is, if \(w=y_{0}\dotsm y_{n}\in R\), then \(w^{-1}=y_{n}^{-1}\dotsm y_{0}^{-1} \in R\) and, for each \(i<n\), we have \(y_{i+1}\dotsm y_{n}y_{0}\dotsm y_{i}\in R\).
We denote by \(N_{R}\) the normal closure of \(R\). As usual,  \(N_{R}\) is defined as \(\{g^{-1} r g \mid g\in F_{X}\text{ and }r\in R\}\), which is the smallest normal subgroup of \(F_{X}\) containing \(R\).

So whenever \(X\) is a set and \(R\subseteq F_{X}\) is symmetrized, we denote by \(\langle X \mid R\rangle\) the group \(H=F_{X}/N_{R}\) and we say that \(H\) is presented by \(\langle X \mid R\rangle\). If \(a, b \in \langle X \mid R\rangle\) such that \(a= u  {N_{R}}\) and \(b = v N_{R}\), then \(a\cdot b =  w N_{R}\) for a reduced \(w = uv\). Clearly the identity of \(F_{X}/N_{R}\) is the equivalence class of the empty word, \(1N_{R}\).

Let us now go back to Theorem~\ref{thm:Williams}. Its proof produces a map sending each graph \(G\) of cardinality \(\kappa\)
with set of vertices \( V = \set{v_{\alpha}}{\alpha<\kappa}\) to the group \(H(G)\) presented by \[
 \langle V \mid R_G\rangle,
 \]
where \(R_G\) is the smallest set which is symmetrized
and contains the following words
\begin{itemizenew}
\item \(v_\alpha^7\) for every \(\alpha<\kappa\);
\item \((v_\alpha v_\beta)^{11}\) for every \((v_\alpha,v_\beta)\in G\);
\item \((v_\alpha v_\beta)^{13}\) for every \((v_\alpha,v_\beta)\notin G\).
\end{itemizenew}

In this section, for a graph \(G\), we will write \(N_{G}\) instead of \(N_{R_{G}}\). So, we have \(H(G) = F_{V}/N_{G}\), and each element of \(H(G)\) is represented by a reduced word \(w\in F_{V}\). We shall differentiate between the word \(w\), which is an element of the free group \(F_{V}\), and \(w  N_{G}\), which is the element of \(H(G)\) represented by \(w\).

When \(G \) is in the space of  graphs
 on \(\kappa\), we can identify \( H(G)\) with a corresponding element in the space of groups on \(\kappa\)
in such a way that the map \( G\mapsto H(G) \) is Borel.
In view of Corollary~\ref{cor:completegraphs}, the following result is immediate.\begin{corollary}\label{cor : Wil}
If \(\kappa\) is a cardinal satisfying~\eqref{eq:kappa}, then the relation \(\sqsubseteq_\mathsf{GROUPS}^\kappa\) is complete for  analytic quasi-order.
\end{corollary}

In this section we strengthen Corollary~\ref{cor : Wil} by proving the analogue of Corollary \ref{cor:main} for embeddability and bi-embeddability on groups.

\begin{theorem}\label{thm:maingroups}
Let \( \kappa \) be any uncountable cardinal satisfying~\eqref{eq:kappa}. Then the embeddability relation \( \sqsubseteq^\kappa_{\mathsf{GROUPS}} \) and  the bi-embeddability relation \( \equiv^\kappa_{\mathsf{GROUPS}} \) are both strongly invariantly universal.
\end{theorem}

We first point out a property satisfied by all \(H(G)\). Recall that a \emph{piece} for the group presented by \( \langle X \mid R \rangle\) is a maximal common initial segment of two distinct \( r_1, r_2 \in R\).
It is easily checked that for every graph \(G\), the set \( R_G \) satisfies the following small cancellation condition:
\begin{equation} \tag*{$C'\left(\frac{1}{6}\right)$}\label{eq : sixth}
\text{if $u$ is a piece and $u$ is a subword of some $r \in R$, then $|u| < \frac{1}{6}|r|$}.
\end{equation}
Groups \( \langle X \mid R \rangle \) whose set of relators \( R \) is symmetrized and satisfies the $C'\left(\frac{1}{6}\right)$ condition are called \emph{sixth groups}.
The only fact that we shall use about sixth groups is the following theorem.

\begin{theorem}[{\cite[Theorem V.10.1]{LynSch}}]\label{Theorem : Greendlinger}
Let $H=\langle X \mid R\rangle$ be a sixth group. If $w$ represents an element of finite order in $H$, then there is some $r\in R$ of the form $r=v^n$ such that $w$ is conjugate to a power of
$v$. 
Thus, 
if furthermore \(w\) is cyclically reduced, then \(w\) is a
 power of some \(v\), with \(v^{n} \in R\). 
\end{theorem}

In the next proposition we use the same terminology as the one of \cite[Section~5.3]{HodgesModel} on interpretations of structures. Recall the following definition.
\begin{definition}
If \(A\) and \(B\) are two structures over the languages \(\mathcal{K}\) and \(\mathcal{L}\), respectively, an \emph{interpretation} \(\Gamma\) of \(A\) into \(B\) is given by
\begin{enumerate-(I)}
\item an \(\mathcal{L}\)-formula \(\partial_{\Gamma}(x)\);
\item an \(\mathcal{L}\)-formula \(\phi_{\Gamma}(x_{0},\dotsc,x_{n})\) for each unnested atomic \(\mathcal{K}\)-formula \(\phi(x_{0},\dotsc,x_{n})\); and
\item a surjective map \(f_{\Gamma}\colon \partial_{\Gamma}(B)\to A\);
\end{enumerate-(I)}
such that for all unnested atomic \(\mathcal{K}\)-formul\ae{} \(\phi(x_{0},\dotsc,x_{n})\) and all \(\bar b=b_{0},\dotsc, b_{n}\in \partial_{\Gamma}(B)\), we have
\[
A\models \phi[f_{\Gamma} (b_{0}),\dotsc, f_{\Gamma}(b_{n})]
\quad\iff\quad
B\models \phi_{\Gamma}[ b_{0},\dotsc, b_{n}].
\]
\end{definition}

We now show that every graph \(G\) of cardinality \(\kappa\) can be interpreted into the group \(H(G)\) in a strong sense. It may be worth pointing out that this fact is true for \emph{any} infinite cardinal \( \kappa \).

Consider the following formul\ae{} in the language of groups (where \( 1 \) is the constant symbol for the unit of the group).
\begin{equation}\label{eq : formula ord}\tag{\(\mathsf{Ord}_{n}(x)\)}
{\bigwedge_{1 \leq k \leq n-1} {x^k\neq 1}}\wedge x^n=1.
\end{equation}
%\begin{equation}\label{eq : formula ord}\tag{\(\mathsf{Ord}_{\infty}(x)\)}
%{\bigwedge_{k=1}^{\omega} {x^k\neq 1}}.
%\end{equation}

\begin{remark}

If \(H(G)\models \mathsf{Ord}_{7}[a]\), then \(a\) has order \(7\) and Theorem~\ref{Theorem : Greendlinger} yields that \(a=uv_{\alpha}^{\pm k}u^{-1}N_{G}\) for some \(\alpha<\kappa\), \(u\in F_{V}\), and \(|k|<7\). Similarly, if \(H(G)\models \mathsf{Ord}_{n}[a]\) for \(n\in\{11,13\}\), then \(a\) has order \(n\) and by Theorem~\ref{Theorem : Greendlinger} there are two distinct \(\alpha,\beta<\kappa\) such that the group element \(v_{\alpha}v_{\beta}\) has order \(n\) and \(a=u(v_{\alpha}v_{\beta})^{\pm k}u^{-1}N_{G}\) for some \(u\in H(G)\) and \(|k|<n\).
\end{remark}

Let \( \mathsf{Same}(x,y) \) be the formula
\begin{multline}\tag{\( \mathsf{Same} (x,y) \)}\label{f:same}
\mathsf{Ord}_{7}(x)\wedge\mathsf{Ord}_{7}(y)\wedge
\left[
\left(
\mathsf{Ord}_{11}(x\cdot y)\wedge \mathsf{Ord}_{11}(y\cdot x)
\right )\vee \right .\\
\left .
\left (
\mathsf{Ord}_{13}(x\cdot y) \wedge \mathsf{Ord}_{13}(y\cdot x)
\right ) 
%\vee x=y
\right ].
\end{multline}

If \(H(G)\models \mathsf{Same}[a,b]\), we say that \(a\) and \(b\) are of the same type. Notice also that the formula \( \mathsf{Same}(x,y) \) is symmetric, i.e.\ for every group \( H \) of size \( \kappa \) and every \( a,b \in H \), one has \( H \models \mathsf{Same}[a,b] \) if and only if \( H \models \mathsf{Same}[b,a] \).

\begin{lemma}\label{lemma:same}
If two distinct \(a,b\) are of the same type in \(H(G)\), then there exist a word \(w\), \(k\in \{-1,1 \}\), and two distinct \(\alpha,\beta<\kappa\) such that \(a\) and \(b\) are represented by \(w v_{\alpha}^{k}w^{-1}\) and \(w v_{\beta}^{k}w^{-1}\), respectively.
\end{lemma}
\begin{proof}
Let \(x_{a}\) and \(x_{b}\) be representatives of \(a\) and \(b\), respectively, i.e.\
\(a = x_{a}  N_{G}\) and \(b = x_{b}  N_{G}\) where \(N_{G}\) is the normal closure of \(R_{G}\).
Since the group elements \(a\) and \( b\) have order \(7\), it follows from Theorem~\ref{Theorem : Greendlinger} that
\(x_{a}=uv_{\alpha}^{k}u^{-1}\) and \(x_{b}=zv_{\beta}^{\ell}z^{-1}\) for some integers \(k,\ell\) such that \(|k|,|\ell|<7\). First let \(\theta_{z}\) be the inner automorphism \(\theta_{z}\colon x\mapsto z^{-1} x z\). Then,
 \(\theta_{z}(x_{b}) = v_{\beta}^{\ell}\) while \(\theta_{z}(x_{a}) = t v_{\alpha}^{k}t^{-1}\) for a reduced word \(t= z^{-1}u\). Now we want to avoid the possibility that \(t\) starts with a power of \(v_{\beta}\) for reasons that will become clear later in the proof. So we consider the inner automorphism \(\theta_{d}\colon x\mapsto v_{\beta}^{-d} xv_{\beta}^{d}\), for the largest \(d\) such that \(v_{\beta}^{d}\) is an initial subword of \(t\). 
So consider the elements \(a_{0}, b_{0}\in H(G)\) such that \(a_{0} = \theta_{d}\circ\theta_{z}(x_{a})  N\) and \(b_{0} = \theta_{d}\circ\theta_{z}(x_{b})\).

We have \(a_{0} = uv_{\alpha}^{k}u^{-1} N\) and \(b_{0}= v_{\beta}^{\ell}N\),
for some reduced \(u\) that does not start with any power of \(v_{\beta}\).
 Then, the product \( a_{0} \cdot b_{0} \) is represented by the word
\begin{equation}\label{eq:ab}%\tag{\(a\cdot b\)}
uv_{\alpha}^{k}u^{-1}v_{\beta}^{\ell}.
\end{equation}
Since \(u\) does not start with any power of \(v_{\beta}\), which in particular implies that \(u^{-1}\) does not end with any power of \(v_{\beta}\), the word \eqref{eq:ab} is cyclically reduced.

Now, notice that \(a_{0}\) and \(b_{0}\) are the images of \(a\) and \(b\) through the inner automorphism of \(H(G)\) \(g\mapsto (v_{\beta}^{-d}z^{-1}N_{G}) g (v_{\beta}^{-d}z^{-1}N_{G})^{-1}\). Therefore, the product \(a_{0}\cdot b_{0}\) has the same order \(a\cdot b\) --- either \(11\) or \(13\). By Theorem~\ref{Theorem : Greendlinger} and the fact that \eqref{eq:ab} is cyclically reduced, it follows that
 \eqref{eq:ab} is the power of \(v_{\gamma}v_{\delta}\) for some \(\gamma,\delta<\kappa\).  Clearly, if a word is a power of \(v_{\gamma}v_{\delta}\) it cannot contain a generator and its inverse. Therefore, \(u\) must be the empty word. We conclude that \eqref{eq:ab} equals \(v_{\alpha}^{k}v_{\beta}^{\ell}\). 

Next, we argue by cases to show that \(\alpha\neq \beta\) and \(k=\ell \in \{1,-1\}\).
First, if \(\alpha = \beta\) then we have the two following possibilities:
\begin{enumerate-(1)}
\item \(k + \ell\) is a multiple of \(7\), which implies that \(a_{0}\cdot b_{0} = 1N_{G}\) because \(v_{\alpha}^{k}v_{\beta}^{\ell} = v_{\alpha}^{n7}\), for some \(n\in \omega\), and every power of \(v_{\alpha}^{7}\) belongs to \(N_{G} \).
\item \(k + \ell\) is  not a multiple of \(7\), but then \(a_{0}\cdot b_{0}\) would have order \(7\) in \(H(G)\) as \((v_{\alpha}^{k}v_{\beta}^{\ell})^{7} = v_{\alpha}^{7(k + \ell)}\) which is in \(N_{G} \).
\end{enumerate-(1)}
In any case, we obtain that \(\alpha\neq \beta\) contradicts the fact that the order of \(a_{0}\cdot b_{0}\) is either \(11\) or \(13\).
So \(\alpha\neq \beta\). Moreover, we have \(|k|,|\ell|\ <7\) by assumption.
In case \(k\neq \ell\) or \(k\notin \{1, -1\} \) every power of \(v_{\alpha}^{k}v_{\beta}^{\ell}\notin N_{G}\), and thus in particular the element \(a_{0}\cdot b_{0} = v_{\alpha}^{k}v_{\beta}^{\ell} N_{G}\), would have infinite order, which again contradicts the hypothesis on the order of \(a_{0}\cdot b_{0}\).
Therefore, we conclude that \(a_{0} = v_{\alpha}^{k}N_{G}\) and \(b_{0}= v_{\beta}^{k}N_{G}\).

Setting \(w= z v_{\beta}^{d}\), we obtain that \(x_{a} =  (\theta_{z})^{-1}\circ(\theta_{d})^{-1}(a_{0})= zv_{\beta}^{d}v_{\alpha}^{k}v_{\beta}^{-d}z^{-1} = w  v_{\alpha}^{k} w^{-1} \). In a similar way we obtain
\(x_{b} = w  v_{\beta}^{k} w^{-1}\).
\end{proof}

Let now  \( \mathsf{gen}(x) \) be the formula

\begin{equation}\label{eq : formula generator}\tag{\(\mathsf{gen}(x)\)}
\exists y (\mathsf{Same}(x, y))
\end{equation}

\begin{remark}\label{rk:generators}
Notice that Lemma~\ref{lemma:same} implies that, whenever \(H(G)\models \mathsf{gen}[a]\), there are \( \alpha < \kappa \), \(k=\pm 1\), and a word \(w\) such that \(a=wv_{\alpha}^{k}w^{-1}  N_{G}\). Viceversa, \(H(G)\models \mathsf{gen}[wv_{\alpha}^{k}w^{-1}  N_{G}]\) for each \(\alpha\), \( k \), and \( w \) as above.
\end{remark}

\begin{proposition}\label{Prop:interpretation}
Let \( \mathcal{K} = \{ R \} \) be the graph language consisting of one binary relational symbol \( R \).
Then there exist three formul\ae{} \(\partial(x), (x=y)_{\Gamma},(R(x,y))_{\Gamma}\) in the language of groups
such that for each graph \(G\) on \( \kappa \), there is a function \(f_{G}\colon \partial(H(G))\to G\) so that the triple consisting of
\begin{enumerate-(I)}
\item \(\partial(x)\),
\item \(\{(x=y)_{\Gamma},(R(x,y))_{\Gamma}\}\), and
\item \( f_{G}\),
\end{enumerate-(I)}
is an interpretation \(\Gamma\) of \(G\) into the group \(H(G)\).
\end{proposition}

\begin{proof}
First let \(\partial(x)\)  be  \( \mathsf{gen}(x)\), and for any graph \( G \) on \( \kappa \) let \(f_{G}\) be 
the map sending each element of \( H(G) \) represented by the word \( wv_{\alpha}^{k}w^{-1}\), where \( \alpha < \kappa \), \( k \in \{ -1,1 \} \), to the vertex  
\(\alpha\) of \( G \). 

Notice that \(f_{G}\) is well defined. Suppose that the words \(wv_{\alpha}^{k}w^{-1}\) and \(uv_{\beta}^{\ell}u^{-1}\) represent the same element \(a\in H(G)\), that is \( wv_{\alpha}^{k}w^{-1}  N_{G}= uv_{\beta}^{\ell}u^{-1}  N_{G}\).
Since \(a\) has order \(7\), the order of \(a \cdot a\) is \(7\) too. 
So we can argue as in Lemma~\ref{lemma:same}.
After applying a suitable inner automorphism of \(H(G)\) we obtain a cyclically reduced reduced word of the kind
\(zv_{\alpha}^{k}z^{-1}v_{\beta}^{\ell}\) that represents an element of order \(7\).
Reasoning exactly as in Lemma~\ref{lemma:same}, we obtain that \(z\) is necessarily the empty word. So we have that \(v_{\alpha}^{k}v_{\beta}^{\ell}\) represents an element of order \(7\). That is,  the word \((v_{\alpha}^{k}v_{\beta}^{\ell})^{7}\) belongs to \(N_{G}\). Clearly we can assume that it is not the case that  \(\alpha=\beta\) and \(k=-\ell\) because we would obtain that \(v_{\alpha}^{k}v_{\beta}^{\ell}\) is the empty word, a contradiction.
In particular, \((v_{\alpha}^{k}v_{\beta}^{\ell})^{7}\) is cyclically reduced, so it belongs to \(R_{G}\).
Notice that Lemma~\ref{lemma:same} implies that \(|k|, |\ell|=1\).
Then the word \((v_{\alpha}^{k}v_{\beta}^{\ell})^{7}\) has exactly fourteen letters.
By definition, the only elements of \(R_{G}\) with this property are the concatenations of two words of the kind \(v_{\gamma}^{\pm7}\) and \(v_{\delta}^{\pm7}\), some \(\gamma,\delta < \kappa\). It follows that \(\alpha = \beta\).

Moreover, by Remark~\ref{rk:generators} the elements of \(H(G)\) satisfying 
\(\partial(x)\) are exactly all the elements of such form, so \( f_G \) is a surjection from 
\( \partial (H(G)) \) onto \( G \).

Consider the following formula in the language of groups:
\begin{equation}\tag{\((x=y)_{\Gamma}\)}
%\partial(x)\wedge\partial(y) \wedge [
\exists z(\mathsf{Ord}_{7}(x\cdot z\cdot y\cdot z^{-1})\vee \mathsf{Ord}_{7}(x^{-1}\cdot z\cdot y\cdot z^{-1}))
%] 
\end{equation}

\begin{claim}\label{claim:interpretationequality}
For every graph \( G \) on \( \kappa \) and every \( a,b \in \partial(H(G)) \)
\[
G\models  f_G(a) =  f_G(b)
\quad \iff\quad
H({G})\models (a =  b)_{\Gamma}.
\]
\end{claim}

\begin{proof}[Proof of the Claim]
Let \( \alpha,\beta < \kappa \), \( k, \ell\in \{ -1,1 \} \), and 
\( w,z \) be such that \( a = wv_\alpha^kw^{-1}  N_{G}\) and \( b = z v_\beta^\ell z^{-1}  N_{G}\), so that 
\( f_G(a) = \alpha \) and \( f_G(b) = \beta \). 
The forward implication is obvious, because \( G \models f_G(a) = f_G(b) \) implies \(\alpha = \beta \). 

For the backward implication, assume that \(H(G)\models (a=b)_{\Gamma}\) and let \(c\in H(G)\) be any 
element witnessing this, say \(c = u  N_{G}\). 
For the sake of definiteness, suppose that the first disjunct is satisfied, so that 
\[
wv_{\alpha}^{k}w^{-1}uzv_{\beta}^{\ell}z^{-1}u^{-1}
\]
represents an element with order \(7\) in \( H(G) \). 
By possibly applying an inner automorphism, we can assume that this element
 is cyclically reduced, and thus we can argue as in the proof of Lemma~\ref{lemma:same} to obtain that
\(\alpha=\beta\) and \(k=\ell\). 
Then \(f_{G}(a)={\alpha}=f_{G}(b)\), which implies that the formula \(f_{G}(a)=f_{G}(b)\) is true in \(G\).
\end{proof}

Now consider the following formula in the language of groups:
\begin{equation}\tag{\((R(x,y))_{\Gamma}\)}
%\partial(x)\wedge\partial(y) \wedge
\neg(x=y)_{\Gamma} \wedge   
\exists z
\left [
\mathsf{Same}(x,z)\wedge (z= y)_{\Gamma} \wedge   \mathsf{Ord}_{11}(x\cdot z) 
%\vee \mathsf{Ord}_{11}(x\cdot z^{-1})
\right]
\end{equation}

\begin{claim}\label{claim:interpretationR}
For every graph \( G \) on \( \kappa \) and every \( a,b \in \partial(H(G)) \)
\[
G\models R[f_G(a), f_G(b)]
\quad\iff\quad
H({G})\models (R[a, b])_{\Gamma}.
\]
\end{claim}

\begin{proof}[Proof of the Claim]
Let \( \alpha,\beta < \kappa \), \( k, \ell\in \{ -1,1 \} \), and 
\( w,z \in H(G) \) be such that \( a = wv_\alpha^kw^{-1} N_{G}\) and \( b = z v_\beta^\ell z^{-1}  N_{G}\). 

Assume 
first that  \( G \models R[f_G(a),f_G(b)] \). Since \( f_G(a) = \alpha \) and \( f_G(b) = \beta \) and the 
graph relation is irreflexive, we have \( \alpha \neq \beta \). By Claim~\ref{claim:interpretationequality},
this implies in particular that \( H(G) \models \neg(a=b)_{\Gamma} \). Set 
\( c = w v_\beta^k w^{-1}  N_{G}\), so that, in particular, \( f_G(c) = \beta = f_G(b) \). Then \( H(G) \models \mathsf{Same}[a,c] \wedge (c = b)_\Gamma \), 
and clearly \( H(G) \models \mathsf{Ord}_{11}[a,c] \) by construction of \( H(G) \) (here we use again 
the fact that \( G \models R[\alpha,\beta] \)). Therefore \( c \) witnesses the existential statement in \( (R[a,b])_\Gamma \), hence \( H(G) \models (R[a,b])_\Gamma \).

Suppose now that
\(G\not\models R[f_{G}(a), f_{G}(b)]\). By the definition of \( H(G) \), it follows that the group element \(a\cdot b =(v_{\alpha}  N_{G})\cdot (v_{\beta}  N_{G}) = v_{\alpha}v_{\beta}  N_{G}\) has order \(13\) in \( H(G) \). Consequently, for any \(c\in H(G)\) of the same type of \(a\) such that \([c=b]_{\Gamma}\) holds in \( H(G) \), we have that \(a\cdot c\) cannot have order \(11\), hence that \( H(G) \not\models (R[a, b])_{\Gamma} \). 
\end{proof}
This concludes the proof of Proposition \ref{Prop:interpretation}.
\end{proof}

\begin{corollary} \label{cor:translation}
For every formula \(\phi(\bar x)\) in the language of graphs there is a formula \(\phi_{\Gamma}(\bar x)\) in the language of groups such that for every graph \(G\) on \( \kappa \)
\[
G\models \phi[f_{G}(\bar a)]\quad\iff\quad H(G)\models \phi_{\Gamma}[\bar a].
\]
\end{corollary}
   
For the sake of brevity, we call a group \( H \) of size \( \kappa \) a \emph{Williams' group} if it is isomorphic to \( H(G) \) for some  graph \( G \) of size \( \kappa \). We are now going to show that when \( \kappa \) is an uncountable cardinal, there is an \( \mathcal{L}_{\kappa^+ \kappa} \)-sentence 
\( \Upphi_{\mathsf{Wil}}\) axiomatizing the Williams' groups of size \( \kappa \). The sentence \( \Upphi_{\mathsf{Wil}}\) will be the conjunction of some sentences considered below.

%\begin{equation}\tag{\(\upvarphi_0\)}\label{f:possible orders}
%\forall x(x=1\vee\mathsf{Ord}_{7}(x) \vee \mathsf{Ord}_{11}(x)\vee \mathsf{Ord}_{13}(x) \vee \mathsf{Ord}_{\infty}(x))
%\end{equation}

%\begin{remark}
%If \(G\) satisfies \eqref{f:possible orders}, then every non neutral element in \(G\) either is of order \(n\), for \(n\in\{ 7, 11, 13\}\), or has infinite order.
%\end{remark}

Let \( \upvarphi_0 \) be the sentence
\begin{equation}\tag{\(\upvarphi_0\)}
\forall x_1, x_2, x_3, x_4 \, \left( x_1 \neq x_4 \wedge \bigwedge_{1 \leq i \leq 3} \mathsf{Same}(x_i, x_{i+1}) \wedge \bigwedge_{1 \leq i \leq 2} \mathsf{Same}(x_i, x_{i+2}) \to \mathsf{Same}(x_1,x_4) \right)
\end{equation}
and %Change 3, blank added by Heike
\( \upvarphi_1 \) be the \( \mathcal{L}_{\kappa^+ \kappa} \)-sentence
\begin{multline}\tag{\(\upvarphi_1\)}\label{f:same generators}
\exists x, x' \left [ 
%\mathsf{gen}(x) \wedge \mathsf{gen}(x') \wedge 
\mathsf{Same}(x,x') \wedge
\forall y 
\left [
  \bigvee_{1 \leq N < \omega} %Change 4, the range including $\omega$ is wrong, Heike Dec 17, 2017
  \exists x_{1},\dots,x_{N}
\left (
\bigwedge_{1 \leq i \leq N}
( \mathsf{Same}(x,x_{i}) \wedge \mathsf{Same}(x',x_{i}))\wedge\right . \right .\right .\\
\left . \left . \left .
\bigwedge_{1\leq i<j\leq N}\mathsf{Same}(x_{i}, x_{j})  \wedge y=x_{1} 
%Change 5, Heike added centred dots
\cdots x_{N}
\wedge\bigwedge_{1\leq i\leq j \leq N} x_{i}\cdots x_{j}\neq 1
\right )
\right ]
\right ].
\end{multline}

Let \( G \) be a \( \kappa \)-sized graph. Although the relation defined by \( \mathsf{Same}(x,y) \) on \( H(G) \) is not transitive,%
\footnote{Given distinct \( \alpha, \beta , \gamma < \kappa \), set \( a = v_\alpha \), \( b = v_\beta \), and \( c = v_\beta v_\gamma v_\beta^{-1} \): then \( H(G) \models \mathsf{Same}[a,b] \wedge \mathsf{Same}[b,c] \), but \( H(G) \not\models \mathsf{Same}[a,c] \) because the product \( a \cdot c \) has infinite order in \( H(G) \), since every power of the word \(v_{\alpha}  v_{\beta} v_{\gamma} v_{\beta}^{-1}\) is not in \(N_{R}\).}
using an argument similar to that of Lemma~\ref{lemma:same}
it is not hard to check that \( H(G) \models \upvarphi_0 \). 
Indeed, suppose that \(a,b,c,d\) satisfy the premise of the implication. Clearly \(a\) and \(d\) have order \(7\). It only remains to prove that \(a\cdot d\) has order 11 or 13. Since \(a, b\) are of the same type, we obtain from Lemma~\ref{lemma:same} and assuming that \(k=1\) for the sake of definiteness that
\(a= u v_{\alpha}u^{-1}  N_{G}\) and \(b= u v_{\alpha}^{d} v_{\beta} v_{\alpha}^{-d} u^{-1}   N_{G}\) for some \(|d|<7\), possibly \(d=0\).
Then, since \(b, c\) are of the same type, we have that
\(c = u v_{\alpha}^{d} v_{\beta}^{e} v_{\gamma} v_{\beta}^{-e} v_{\alpha}^{-d} u^{-1}   N_{G}\).
Now, since \(a,c\) are also of the same type, we have that
\(a\cdot c = u v_{\alpha}^{d+1} v_{\beta}^{e} v_{\gamma} v_{\beta}^{-e} v_{\alpha}^{-d} u^{-1}  N_{G}\) has order \(11\) or \(13\). Notice that this can only happen if \( e= 0 \),
so \(c = u v_{\alpha}^{d} v_{\gamma}v_{\alpha}^{-d}u^{-1}  N_{G}\). Repeating this argument one more time for \(d\), we obtain that \(d = u v_{\alpha}^{d} v_{\delta}v_{\alpha}^{-d}u^{-1}  N_{G}\). Now it is clear that \(a\cdot d = u v_{\alpha}^{d}  v_{\alpha}v_{\delta}v_{\alpha}^{-d}u^{-1}  N_{G}\) has the same order of \(v_{\alpha}  v_{\delta}N_{G}\), which is either 11 or 13.

Moreover, setting e.g.\ \( x = v_0  N_{G} \) and \( x' = v_1  N_{G}\) it is straightforward to check  that \( H(G) \models \upvarphi_1 \). % for every \( \kappa \)-sized graph \( G \).

\begin{remark}\label{rk : genertorssametype}
If \(H\) is a group of cardinality \(\kappa\) and satisfies \( \upvarphi_0 \wedge \upvarphi_1 \), then
there is a set \(W\subseteq H\) such that \(W\) generates \(H\), and all elements of \(W\)
are pairwise of the same type. Such \(W\) can be obtained by
fixing any two witnesses \(a,b\in H\) to the existential quantifier at the beginning of \( \upvarphi_1 \), and then setting
\[ 
W = \{ a,b \} \cup \{ c \in H \mid H \models \mathsf{Same}[a,c] \wedge \mathsf{Same}[b,c] \}.
 \] 
Since the cardinality of \(H\) is \(\kappa\), the set \(W\) has size \(\kappa\) because it has to generate the whole \( H \) by \( H \models \upvarphi_1 \).
The sentence \( \upvarphi_0 \) takes care of the fact that distinct elements in \( W \) are of the same type: if \( c,d \) are distinct elements of  \( W \setminus \{ a,b \} \), then all of \( (c,a) \), \( (c,b) \),\( (a,b) \), \( (a,d) \) and \( (b,d) \) are pairs of elements of the same type, and thus \( H \models \mathsf{Same}[c,d] \) because \( H \) satisfies \( \upvarphi_0 \).
Moreover, notice that, by the way \( \mathsf{Same}(x,y) \) was defined, a group element \(c \) and its inverse are never of the same type because their product does not have order  \( 11 \) or \( 13 \). So the basic fact that when \(c\) has order \(7\) the inverse \(c^{-1}\) equals \(c^{6}\), plays a crucial role to argue that such \(W\) is a set of generators. 
 Finally, notice that when \( H = H(G) \) for some graph \( G \) of size \( \kappa \), the set \( W \) defined in this way will be of the form
 \( W = \{ w v_\alpha^k w^{-1} \mid \alpha < \kappa \} \), for some word \( w \) and \( k \in \{-1,1\} \) only depend on the initial choice of \( a \) and \( b \) -- see Lemma~\ref{lemma:same}.
\end{remark}

Recall that the relators of the group \(H(G)\), for any graph \(G\), are of three possible length: \(7\), \(22\), or \(26\). 
Define the following \( \mathcal{L}_{\kappa^+ \kappa} \)-formul\ae{}.

\begin{equation}\tag{\(\mathsf{Rel_{7}}(x_{1},\dotsc,x_{7})\)}
\bigwedge_{1 \leq i \leq 6} x_{i}=x_{i+1}.
\end{equation}

\begin{equation}\tag{\(\mathsf{Rel_{22}}(x_{1},\dotsc,x_{22})\)}
\mathsf{Ord}_{11}(x_{1}\cdot x_{2})\wedge\bigwedge_{1 \leq i \leq 10} ( x_{2i-1}=x_{2i+1} \wedge  x_{2i}=x_{2i+2} 
%\wedge x_{2i-1}\neq x_{2i}
).
\end{equation}

\begin{equation}\tag{\(\mathsf{Rel_{26}}(x_{1},\dotsc,x_{26})\)}
\mathsf{Ord}_{13}(x_{1}\cdot x_{2})\wedge\bigwedge_{1 \leq i \leq 12} (x_{2i-1}=x_{2i+1} \wedge  x_{2i}=x_{2i+2} 
%\wedge x_{2i-1}\neq x_{2i}
).
\end{equation}

Let now \( \upvarphi_2 \) be the \( \mathcal{L}_{\kappa^+ \kappa} \)-sentence
%\todo[color= white]{I added \(\bigwedge_{1 \leq i, j \leq N} \mathsf{Same}(x_{i}, x_{j}) \). I don't think it is necessary but it simplifies the proof.}
\begin{multline}\tag{\( \upvarphi_2\)}\label{f:rel}
  \bigwedge_{1 \leq N <  \omega}
  % Change 6 by Heike, the former range, including $\omega$, was wrong
  \forall x_{1},\dotsc, x_{N} \left [
x_{1}\cdots x_{N} = 1 \wedge \bigwedge_{1 \leq i \leq N}\mathsf{gen}(x_{i})\wedge
\bigwedge_{1 \leq i, j \leq N} \mathsf{Same}(x_{i}, x_{j}) \wedge
\right .
\\
\bigwedge_{1 \leq i \leq N-1} x_{i}\cdot x_{i+1}\neq 1 \wedge x_{1}\cdot x_{N}\neq 1
\rightarrow
\left . \vphantom{\bigwedge_{i=1}^{N-1}}
\left (
\mathsf{Rel_{7}}(x_{1},\dotsc,x_{7}) \vee
\mathsf{Rel_{22}}(x_{1},\dotsc,x_{22}) \vee
\mathsf{Rel_{26}}(x_{1},\dotsc,x_{26})
\right)
 \right ],
 \end{multline}
 where for each \( n \in \{ 7,22,26 \} \), we stipulate that \( \mathsf{Rel}_{n}(x_1, \dotsc, x_n) \) is a contradiction if \( N < n \). It is not difficult to see that, by construction, \( H(G) \models \upvarphi_2 \) for every graph \( G \) of size \( \kappa \).
To see this, suppose that \(a_{1},\dotsc,a_{N}\) satisfy the premise of the implication inside the square brackets. Each of them has the form
\(a_{i} = wv_{\alpha_{i}}^{k}w^{-1}  N_{G}\) for some reduced word \(w\), \(k\in
\{1,-1\}\) and \(\alpha_{i}<\kappa\).
 Next, the condition \(wv_{\alpha_{1}}^{k}\dotsm v_{\alpha_{N}}^{k}w^{-1}  N_{G}=a_{1}\dotsm a_{N}=1\) implies that  \(wv_{\alpha_{1}}^{k}\dotsm v_{\alpha_{N}}^{k}w^{-1}\in N_{G}\). Notice that \(a_{1}\) is not the inverse of \(a_{N}\) by the last conjunct of the premise. Therefore \(v_{\alpha_{1}}^{k}\dotsm v_{\alpha_{N}}^{k}\) is a cyclically reduced word. Then \(wv_{\alpha_{1}}^{k}\dotsm v_{\alpha_{N}}^{k}w^{-1}\in N_{G}\) if and only if \(v_{\alpha_{1}}^{k}\dotsm v_{\alpha_{N}}^{k}\) belongs to the group generated by \(R(G)\). At this point it is clear that either one of the following hold:
 \(\mathsf{Rel_{7}}(a_{1},\dotsc,a_{7})\), or
\(\mathsf{Rel_{22}}(a_{1},\dotsc,a_{22})\), or
 \(\mathsf{Rel_{26}}(a_{1},\dotsc,a_{26})\).

\begin{lemma}\label{lemma:normalclosure}
Let \(H\) be a group such that \( H \models \upvarphi_2 \), and let \(  a_1  \cdots a_N \) be a product of elements of \( H \) (for some \( 1 \leq N < \omega \)) such that \( H \models \mathsf{gen}[a_i] \) for every \( 1 \leq i \leq N \) and  \(  a_1  \cdots a_N  = 1 \). Then \( a_1 \cdots a_N \) belongs to the normal closure
\(N_{R}\) of the 
%symmetrized 
set 
%of relators 
\(R = R(a_1, \dotsc, a_N) \) consisting of the elements
\begin{enumerate-(i)}
\item \label{item:rel1}\(a_i^7\) for every \( 1 \leq i \leq N \); %\(b\in H\) such that \(H\models \mathsf{gen}[b]\);
\item \((a_i\cdot a_{j})^{11}\) for every \(1 \leq i < j \leq N \) such that \( H \models \mathsf{Ord}_{11}[a_i \cdot a_{j}] \); % such that \(H\models \mathsf{Same}(c,d)\wedge \mathsf{Ord}_{11}[c\cdot d]\);
\item \label{item:rel3} \((a_i\cdot a_{j})^{13}\) for every \( 1 \leq i < j \leq  N \) such that \( H \models \mathsf{Ord}_{13}[a_i \cdot a_{j}] \). 
%\(c,d\in H\) such that \(H\models \mathsf{Same}(c,d)\wedge \mathsf{Ord}_{13}[c\cdot d]\).
\end{enumerate-(i)}
\end{lemma}

\begin{proof}
Suppose towards a contradiction that the lemma fails, and let \( N \) be smallest such that there is a product
\( a_1 \cdots a_N \) satisfying the hypothesis of the lemma, but such that \( a_1 \cdots a_N \notin N_{R} \), where \( R = R(a_1, \dotsc, a_N) \) is as above. 
%that \( that \(H \models \upvarphi_2 \) and % the elements \(a_{1}\cdots a_{N}\in H\) are such that
%%\(a_{1}\cdots a_{N}=1\) and 
By minimality of \( N \), we also have that \(a_{i}\cdot a_{i+1}\neq 1\) for every \( 1 \leq i < N \),
%, and by possibly applying an inner automorphism we can further assume that 
and that \(a_{1}\neq a_{N}^{-1}\).
Since \( H \models \upvarphi_2 \) and the premise of the implication is satisfied when setting \( x_i = a_i \) for every \( 1 \leq i \leq N \), then there is \(n\in \{7,22,26\}\) such that the product of the first \(n\) factors is 
\begin{enumerate-(i)}
\item \label{item:rel1a} \(a_1^{7}\) if \( n = 7 \), or
\item \((a_1 \cdot a_2 )^{11}\) with \( H \models \mathsf{Ord}_{11}[a_1 \cdot a_2] \) if \( n = 11 \), or
\item \label{item:rel3a} \((a_1 \cdot a_2)^{13}\) with \( H \models \mathsf{Ord}_{13}[a_1 \cdot a_2] \) if \( n = 13 \).
\end{enumerate-(i)}
In each of the three cases, it follows that the product of the first \(n\) factors equals \(1\). As a consequence, the product
\[a_{n+1}\dotsm a_{N}\]
 still satisfies the hypothesis of the lemma, and thus \( a_{n+1} \cdots a_N \in N_{ R(a_{n+1}, \dotsc, a_N)} \) by minimality on \( N \). But since \( R(a_{n+1}, \dotsc, a_N) \subseteq R(a_1, \dotsc, a_N) \), this would imply \( a_1 \cdots a_N \in N_{R(a_1, \dotsc a_N)} \), a contradiction. 
%equals \( 1 \) as well and we can repeat the above argument until we exhaust all the product. 
%We conclude that  \( a_{1}\cdots a_{N} \) is (conjugate to) a product of elements as described
%in \ref{item:rel1}--\ref{item:rel3}, whence
%\(a_{1}\cdots a_{N}\in \ncl(R)\).
\end{proof}

Finally, let \(\upvarphi_{\mathsf{gp}}\) the first-order sentence axiomatizing groups. Then \(\Upphi_{\mathsf{Wil}}\) is the \( \mathcal{L}_{\kappa^+ \kappa} \)-sentence

\begin{equation}\tag{\( \Upphi_{\mathsf{Wil}}\)}\label{formula:Wil}
\upvarphi_\mathsf{gp} \wedge \upvarphi_0 \wedge \upvarphi_1 \wedge \upvarphi_2.
\end{equation}

\begin{remark}\label{remark:H(G)modelsWil}
Notice that \( H(G) \models \Upphi_{\mathsf{Wil}} \) for every \( \kappa \)-sized graph \(G\). 
\end{remark}

\begin{lemma}\label{lemma:rightinverseisogroup}
Let \( H \) be a group of size \( \kappa \).
If \(H\models \Upphi_{\mathsf{Wil}}\), then  \( H \) is a Williams' group, i.e.\ \(H\cong H(G)\) for some 
graph \( G \) of size \( \kappa \).
\end{lemma}

\begin{proof}
Given \( H \) such that \(H\models \Upphi_{\mathsf{Wil}}\),
let \(W\) be a set of generators for \(H\) as in Remark~\ref{rk : genertorssametype}, and
let \((w_{\alpha})_{ \alpha<\kappa}\) be an enumeration without repetitions of \(W\).
By the universal property of the free group we have \(H\cong F(W)/N\), where \(F(W)\) denotes the free group on \(W\) and \(N\) is some normal subgroup of \(F(W)\).
Denote by \(R_{H}\) the smallest symmetrized subset of \(F(W)\) containing the words
\begin{itemizenew}
\item \(w_\alpha^7\) for every \(w_\alpha\in W\);
\item \((w_\alpha \cdot w_\beta)^{11}\) if  \(H\models\mathsf{Ord}_{11}[w_{\alpha} \cdot w_{\beta}]\);
\item \((w_\alpha \cdot w_\beta)^{13}\) if  \(H\models\mathsf{Ord}_{13}[w_{\alpha} \cdot w_{\beta}]\).
\end{itemizenew}
For the way \(R_{H}\) is defined, the normal closure of \(R_{H}\), that we denote by \(N_{R_{H}}\), is a (necessarily normal) subgroup of \(F(W)\) and thus is contained in \(N\).
Now we shall show that \(N\subseteq N_{R_{H}}\).
Suppose that \(w\in N\), namely, that the group element \(w\cdot N\) is the unity \(1\cdot N\) of \(H\). Say \(w=w_{\alpha_{1}}\cdots w_{\alpha_{n}}\) for \(w_{\alpha_{1}},\dotsc, w_{\alpha_{n}}\in W\).  We can suppose that \(w_{\alpha_{i+1}}\neq w_{\alpha_{i}}^{-1}\) for every  \(i<n\).
 It follows by Lemma~\ref{lemma:normalclosure} that  \(w\) is contained in the normal closure of \(R(w_{\alpha_1}, \dotsc, w_{\alpha_n}) \), which is included in \( N_{R_{H}} \) by definition of \(R_{H}\).
 
 By the discussion above, it follows that \(N=N_{R_{H}}\), therefore \(H\cong \langle W\mid R_{H}\rangle\). 
 We define a binary relation \( R^G \) on \( \kappa \) by setting for \( \alpha, \beta < \kappa \)
 \[ 
\alpha \mathrel{R^G} \beta \iff \text{\(w_{\alpha}\cdot w_{\beta}\) has order \(11\) in \(H\)}. 
\]
The relation \( R^G \) is irreflexive because for every \(\alpha< \kappa \) we have \( H \models \mathsf{gen}[w_\alpha  N] \), so that \( w_\alpha \) has order \( 7 \) in \( H \) and thus \( (w_\alpha  N) \cdot (w_\alpha  N)\) cannot have order \( 11 \) in \( H \). Moreover, \( R^G \) is also symmetric because by definition of \( W \), for any two distinct \(\alpha,\beta < \kappa\) the group elements \((w_{\alpha} N)\) and \((w_{\beta}N)\) are of the same type, and thus the order of \((w_{\alpha}  N) \cdot (w_{\beta}  N)\) equals the order of \((w_{\beta} N) \cdot (w_{\alpha} + N)\). It follows that the
resulting structure \( G = (\kappa, R^G) \) is a graph on \( \kappa \), and it is easy to check that \( H \cong H(G)\) via the isomorphism \(w_{\alpha}  N\mapsto v_\alpha  N_{G}\).
\end{proof}

\begin{remark}\label{remark:inversemapBorel}
The construction given in the proof of Lemma~\ref{lemma:rightinverseisogroup} 
actually yields a Borel map \( H \mapsto G_H \)
from the space of groups on \(\kappa\)
satisfying \(\Upphi_{\mathsf{Wil}}\)
to the space of graphs on \(\kappa\) such that  \(H\cong H(G_H)\) for each \( H \models \Upphi_{\mathsf{Wil}} \).
\end{remark}

Now we have all the ingredients to prove the main theorem of this section, namely Theorem~\ref{thm:maingroups}. Indeed, it immediately follows from Corollary~\ref{cor:main} and the following proposition.

%  of this section, which states that
%for any uncountable cardinal \( \kappa \) satisfying~\eqref{eq:kappa}, the embeddability relation \( \sqsubseteq^\kappa_{\mathsf{GROUPS}} \) and  the bi-embeddability relation \( \equiv^\kappa_{\mathsf{GROUPS}} \) are both strongly invariantly universal.
%
%It is enough to prove the following proposition.

\begin{proposition}\label{prop:classwisegroup}
For every sentence \(\upvarphi \) in the language of graphs there is a sentence \(\upphi\) in the language of groups such that \({\sqsubseteq_{\upvarphi}^{\kappa}} \simeq_B {\sqsubseteq_{\upphi}^{\kappa}}\).
\end{proposition}

\begin{proof}
Given any sentence \(\upvarphi\) in the language of graphs, let \(\upphi \) be the sentence
\[ 
\upvarphi_{\Gamma}\wedge \Upphi_{\mathsf{Wil}},
\]
where \( \upvarphi_{\Gamma} \) is as in Corollary~\ref{cor:translation}.
Let \(f\) be the quotient map of the Borel function
\[h\colon \mathrm{Mod}^{\kappa}_{\upvarphi}\to \mathrm{Mod}^{\kappa}_{\upphi}\colon G\mapsto H(G)\]
with respect to the bi-embeddability relation (on both sides). The range of \( h \) is contained in 
\(\mathrm{Mod}^{\kappa}_{\upphi}\) by Corollary~\ref{cor:translation} and Remark~\ref{remark:H(G)modelsWil}, and its quotient map \( f \) is well-defined because \(h\) witnesses Theorem~\ref{thm:Williams}. Moreover, by Lemma~\ref{lemma:rightinverseisogroup} and Corollary~\ref{cor:translation} again, for every \( \kappa \)-sized group \( H \) we have that \( H \in \mathrm{Mod}^\kappa_\upphi \) if and only if there is \(G \in \mathrm{Mod}^\kappa_\upvarphi \) such that \( H \cong H(G) \), and by Remark~\ref{remark:inversemapBorel} such \( G  = G_H \) can be recovered in a Borel way. It follows that \( f \) is an isomorphism between the relevant quotient spaces, and that the restriction of the map \( H \mapsto G_H \) to \( \mathrm{Mod}^\kappa_\upphi \) is a Borel lifting of \( f^{-1} \). Therefore the map \(f\) witnesses that \({\sqsubseteq_{\upvarphi}^{\kappa}} \simeq_B {\sqsubseteq_{\upphi}^{\kappa}}\).
%
%the map
%
%Next we define the map
%\begin{equation}
%h \colon \Mod^\kappa_{\upphi} \to \Mod^\kappa_{\upvarphi}  \colon H \mapsto X,
%\end{equation}
%where \(X\) is the graph such that \(H(X)\cong H\) as in the proof of Lemma~\ref{lemma:rightinverseisogroup} which is Borel by
%Remark~\ref{remark:inversemapBorel}.
%Moreover, \(h\) reduces embeddability on \(\Mod^{\kappa}_{\upphi}\) to  embeddability on \(\Mod^{\kappa}_{\upvarphi}\). It follows that the quotient map of \(h\) with respect to bi-embeddability is the inverse of \(f\).
%Therefore, the map \(f\) witnesses that \(\sqsubseteq_{\upvarphi}^{\kappa} \simeq_B \sqsubseteq_{\upphi}^{\kappa}\).
%\({{\sqsubseteq} \restriction \Mod^\kappa_{\upvarphi} }\simeq_B {{\sqsubseteq} \restriction \Mod^\kappa_{\upphi} }\).
\end{proof}

%
%\begin{proof}[Proof of Theorem \ref{thm:maingroups}]
%It follows from Proposition~\ref{prop:classwisegroup} and .
%  \end{proof}

%There is a formula in the language of group that is satisfied by every group that is isomorphic to \(H_{G}\) for some combinatorial tree \(G\). 

%\begin{multline}\label{eq : formula generator}\tag*{\(\Upphi_{\mathsf{Wil}}\)}
%\forall x(\bigvee_{i=0}^\omega(
%\exists y_{0},\dotsc, y_{n}(\bigwedge_{j=0}^i\Upphi_{ord(7)}(y_j)\wedge
 %x=y_0\cdot\dots\cdot y_i )\wedge\\
 %\bigwedge_{k,\ell<i}k\neq\ell\rightarrow \bigvee_{n}^{\omega}
 %\exists z_{1},\dotsc, z_{n+1}(\Upphi_{ord(11)}(y_{k}\cdot z_{1})\wedge\dots\wedge \Upphi_{ord(11)}(z_{n+1}\cdot y_{\ell}))\wedge\\
%\bigwedge_{k,\ell<i} (\bigvee_{m}^{\omega} \exists z_{1},\dotsc, z_{1+m}(
 %\Upphi_{ord(11)}(y_{k}\cdot z_{1})\wedge\dots\wedge \Upphi_{ord(11)}(z_{1+m}\cdot y_{\ell}))
%\rightarrow \Upphi_{ord(13)}(y_{k}\cdot y_{\ell}))).
%\end{multline}

\begin{remark}
It must be stressed that all the results in this section, unlike the preceding ones, are true for \emph{any} infinite cardinal \( \kappa \). Therefore, setting \( \kappa = \omega \) in Proposition~\ref{prop:classwisegroup}
and combining it with~\cite[Theorem 3.9]{friedman-mottoros} we get an alternative proof of~\cite[Theorem 3.5]{calderoni-mottoros}.
\end{remark}

\section{Further results and open problems} \label{sec:questions}

Generalized descriptive set theory not only provides a good framework to deal with uncountable first-order structures, but it also allows us to nicely code various kind of non-separable topological spaces. 

For example, in~\cite[Section 7.2.3]{andmot} it is shown how to construe the space of all  complete metric spaces of density character \( \kappa \) (up to isometry) as a standard Borel \( \kappa \)-space 
\( \mathfrak{M}_\kappa \). 
This is obtained by coding each such space \( M =  (M,d_M) \) as the element  \( x_M \in \pre{\kappa \times \kappa \times \mathbb{Q}^+}{2} \) (where \( \mathbb{Q}^+ = \{ q \in \mathbb{Q} \mid q > 0  \} \)) defined by setting
\[ 
x_M(\alpha,\beta,q)  = 1 \iff d_M(m_\alpha, m_\beta) < q,
 \] 
where \( \{ m_\alpha \mid \alpha < \kappa \} \) is any dense subset of \( M \) of size \( \kappa \). Note that such a code is not unique, as it depends on both the choice of a dense subset of \( M \)  and of a specific 
enumeration of it. The space \( M \) can easily be recovered, up to isometry, from any of its codes \( x_M \) by taking the completion of the metric space \( (\kappa,d_{x_M}) \), where 
\( d_{x_M}(\alpha,\beta) =  \inf \{ q \in \mathbb{Q}^+ \mid x_M(\alpha,\beta,q) = 1 \} \). The space \( \pre{\kappa \times \kappa \times \mathbb{Q}^+}{2} \) is naturally homeomorphic to \( \pre{\kappa}{2} \), and 
it can be straightforwardly checked that the set \( \mathfrak{M}_\kappa \subseteq \pre{\kappa \times \kappa \times \mathbb{Q}^+}{2} \) of all codes for complete metric spaces of density character \( \kappa \) is a Borel
subset of it, 
and hence a standard Borel \( \kappa \)-space.
It immediately follows that the relation \( \sqsubseteq^i_\kappa \) of isometric embeddability on \( \mathfrak{M}_\kappa \) is an analytic quasi-order, whose complexity can then be analyzed in terms of Borel reducibility. An 
easy consequence of Corollary~\ref{cor:completegraphs} is the following (compare it with the main results in~\cite[Section 16.3.1]{andmot}, which deal with the case \( \omega < \kappa < 2^{\aleph_0} \)).

\begin{corollary} \label{cor:completemetric}
Let \( \kappa \) be any cardinal satisfying~\eqref{eq:kappa}. Then \( \sqsubseteq^i_\kappa \) is complete for analytic quasi-orders. Indeed, the same is true when \( \sqsubseteq^i_\kappa \) is restricted to the subclass of  \( \mathfrak{M}_\kappa \) consisting of all discrete spaces.
\end{corollary} 

This result is obtained using and easy and somewhat canonical (continuous) way of transforming a graph \( G \) on \( \kappa \) into a discrete metric space \( (\kappa,d_G) \) (necessarily complete and of density character \( \kappa \)), namely:
Fix strictly positive \( r_0, r_1 \in \mathbb{R} \)  such that \( 0 < r_0 < r_1 \leq  2r_0 \), and set
\[ 
d_G(\alpha, \beta)  = 
\begin{cases}
0 & \text{if }\alpha = \beta \\
r_0 & \text{if } \alpha \neq \beta \text{ and \(\alpha\) and \(\beta\) are adjacent in } G \\
r_1 & \text{if } \alpha \neq \beta \text{ and \(\alpha\) and \(\beta\) are not adjacent in } G.
\end{cases}
\]
(The condition on \( r_0 \) and \( r_1 \) ensures that \( d_G \) satisfies the triangular inequality.) 

Furthermore, the correspondence between graphs and discrete metric spaces just described is so tight that it easily yields the following strengthening of Corollary~\ref{cor:completemetric} (just use Corollary~\ref{cor:main} instead of Corollary~\ref{cor:completegraphs}, plus the fact that any discrete metric space \( M \) on \( \kappa \) isometric to some \( (\kappa,d_G) \) is of the form \( (\kappa, d_{G'}) \) for some \( G' \cong G \)).

\begin{corollary}
Let \( \kappa \) be any uncountable cardinal satisfying~\eqref{eq:kappa}. Then the isometric embeddability relation \( \sqsubseteq^i_\kappa \) is \emph{strongly invariantly universal} in the following sense:
For every (\( \kappa \)-)analytic quasi-order \( R \) there is a Borel \( B \subseteq \mathfrak{M}_\kappa \) closed under isometry such that \( R \simeq_B {\sqsubseteq^i_\kappa \restriction B} \).

Moreover, the same applies to the restriction of \( \sqsubseteq^i_\kappa \) to discrete spaces, and to the isometric bi-embeddability relation on the same classes of metric spaces.
\end{corollary}

Besides discrete spaces, there is another subclass of \( \mathfrak{M}_\kappa \) that has been widely considered in relation to this kind of problems, namely that of \emph{ultrametric spaces}. (Recall that a metric \( d \) is an ultrametric if it satisfies the following strengthening of the triangular inequality: \( d(x,z) \leq \max \{ d(x,y), d(y,z) \} \) for all triple of points \( x,y,z \).) The descriptive set-theoretical complexity of the restriction of \( \sqsubseteq^i_\kappa \) to ultrametric spaces have been fully determined in~\cite{Gao:2003qw,friedman-mottoros,CamerloMarconeMottoros2013}  for the classical case \( \kappa = \omega \), and some results for the case \( \omega < \kappa < 2^{\aleph_0} \) have been presented in~\cite{andmot}. Unfortunately, the completeness results obtained in this paper cannot instead be used to obtain analogous results for the case when \( \kappa \) satisfies~\eqref{eq:kappa}. This is because
%, contrarily to what happened in the previously mentioned papers, 
in our current main construction (Sections~\ref{sec:labels}--\ref{sec:completeness}) we used \emph{generalized} trees with uncountably many levels, and we do not know how to canonically transform such a tree in an ultrametric space in a ``faithful'' way.

A strategy to overcome this difficulty would consist in first proving the completeness of the embeddability relation on a different kind of trees of size \( \kappa \), namely \emph{combinatorial} trees. A \emph{combinatorial tree} is a domain equipped with a relation (not a partial order)
\( \preceq^T \) such that $(T,\preceq^T)$ is a graph (i.e.\ irreflexive and \emph{symmetric}) relation and \( T \) is connected and acyclic. This is exactly the kind of trees used in the previously mentioned papers, and such trees can straightforwardly be transformed in complete ultramentric spaces of density character \( \kappa \) (in fact, even into ultrametric \emph{and} discrete metric spaces of size \( \kappa \) --- see e.g.\ \cite[Section 16.3.2]{andmot} for more details on this construction). A slightly weaker approach would be that of considering descriptive set-theoretical trees of countable height, namely DST-trees \( T \subseteq \pre{< \alpha}{\kappa} \) for some \( \alpha < \omega_1 \). The construction presented in~\cite[Section 4]{MottoRosUMI}  would then allow us to transfer the results concerning these trees to the context of complete (discrete) ultrametric spaces of density character \( \kappa \). This discussion motivates the following questions.

\begin{question}
Let \( \kappa \) be any uncountable cardinal satisfying~\eqref{eq:kappa}. What is the complexity with respect to Borel reducibility of the embeddability relation between combinatorial trees of size \( \kappa \)? What about descriptive set-theoretical trees of size \( \kappa \) and countable height?
\end{question}

A somewhat related, albeit weaker, question is the following:

\begin{question}
Let \( \kappa \) be any uncountable cardinal satisfying~\eqref{eq:kappa}. What is the complexity with respect to Borel reducibility of the embeddability relation between arbitrary set-theoretical trees of size \( \kappa \)?
\end{question}

There are evidences that an answer to this question can be obtained if we replace embeddability with \emph{continuous} embeddability, where ``continuous'' means that the embeddings \( f \) between set-theoretical trees \( T_1 \) and \( T_2 \) must satisfy the following additional condition: If \(s \in T_1 \) has limit height, then \( f(s) = \sup \{ f(t) \mid t \in \pred(s) \} \).

We conclude this section by noticing that our completeness results can be transferred to many other settings. For example, an approach similar to that used in the case of complete metric spaces of density character \( \kappa \) allows us to construe the space of all Banach spaces of density \( \kappa \) as a standard Borel \( \kappa \)-space \( \mathfrak{B}_\kappa \), see~\cite[Section 7.2.4]{andmot} for more details on such coding procedure. It follows that the relation \( \sqsubseteq^{li}_\kappa \) of linear isometric embeddability on \( \mathfrak{B}_\kappa \) is an analytic one, and combining Corollary~\ref{cor:completegraphs} with the construction in~\cite[Section 16.4]{andmot} one easily gets

\begin{corollary}
Let \( \kappa \) be any cardinal satisfying~\eqref{eq:kappa}. Then \( \sqsubseteq^{li}_\kappa \) is complete for analytic quasi-orders.
\end{corollary}

We do not know if this can be further improved to a (strongly) invariant universality result.

\nothing{
\section{Private appendix: Suggestions}
We could do with set-theoretic trees, if we replace the Baumgartner
linear orders $(L_{\kappa, \gamma}, \prec_{L_{\kappa,\gamma}})$ by the following
  trees $T_S$, $S \subseteq \{\alpha < \kappa \such \alpha \mbox{ limit}\}$
  for a set of size $2^\kappa$ of modulo club different stationary subsets of $S$.  (We need only $2 \times \kappa$ many of them in our labels of type I and of type III.)

  \begin{definition}
  $T_S = \kappa \cup \{(0,\alpha) \such \alpha \in S\}$
  $\alpha <_{T_S}  \beta$ if $\alpha < \beta$,
  $\alpha <_{T_S} (0,\beta)$ if $\alpha < \beta$.
  No other relations hold.
  \end{definition}
  
  Note that $(T_S,<_{T_S})$ is  a set-theoretic tree, however, it is
  not a DST-tree.
  For non-stationary $S \triangle S'$, $T_S$ is not continuously embaddable into $T_S'$.

  However, if we ask for just embedding, then I would propose to
  make the tree bushy exactly towards levels $\alpha \in S$.

  Of course now a lot of computations about the labels would have to be altered and checked, and
  non-spines of length 1 have to avoided so that the
  new points $(0,\alpha)$ for $\alpha \in S$ get their role that can be read off in the languare $L_{\kappa, \kappa}$.

  The part called ${\sf Seq}$ needs to be revised as well, as the $\ZZ$-copies are forbidden now.
  In Section 5, we could define $G_0$ just by the
  ``the nodes that have cofinally many predecessors
  who have $\kappa$ many direct successors''
  and leave out the $\ZZ$-copies.
  This is not consistent with $\upvarphi_R$ from Section 6,
 
  However, we might be able to modify the choice of $T$ in Lemma~\ref{lemmanormalform} to include
  for one pair $(x,y) \in R$ many witnesses $(x,y,\xi) \in T$, in order to remedy this.
}

\def\cprime{$'$}

\end{document}